\documentclass{amsart}
\usepackage{amssymb}
\usepackage[v2,cmtip]{xy}
\usepackage{graphicx}

\DeclareMathAlphabet{\scr}{U}{rsfs}{m}{n}

\def\draftdate{January 3, 2021}

\newdir{ >}{{}*!/-5pt/\dir{>}}

\newcommand{\mC}{\aM}
\newcommand{\mS}{\mathbf{1}}
\newcommand{\mE}{\aE}
\newcommand{\mSE}{\mathbf{*}}
\newcommand{\mtimes}{\mathbin{\square}}
\newcommand{\metimes}{\times}
\newcommand{\mtensor}{\otimes}
\newcommand{\kay}{k}
\newcommand{\amap}{a}
\newcommand{\pmap}{c}
\newcommand{\oEnd}{\oE{\mathrm{nd}}}
\newcommand{\oCom}{\oC{\mathrm{om}}}
\newcommand{\oAss}{\oA{\mathrm{ss}}}
\newcommand{\nsAss}{\overline{\oA{\mathrm{ss}}}}
\newcommand{\nsA}{\overline{\oA}}
\newcommand{\nsB}{\overline{\oB}}
\newcommand{\nsC}{\overline{\oC}}
\newcommand{\nsCl}{\overline{\oC}{}^{\ell}}
\newcommand{\nsE}{\overline{\oE}}
\newcommand{\nsO}{\overline{\oO}}
\newcommand{\nbO}{\overline{\bO}}

\newcommand{\oIdp}{\oI_{\mathrm{dbp}}}
\newcommand{\tbC}{\tilde\bC}
\newcommand{\Ln}{\Sigma^{n}\otimes_{\bC_{n}}}

\newcommand{\NS}{_{\not\Sigma}}
\newcommand{\nU}{\overline U\vrule height1em depth0pt width0pt}

\newcommand{\Top}{\aS}
\newcommand{\Set}{\oS\mathrm{et}}

\newcommand{\rU}{\mathbf{U}}

\mathchardef\varDelta="7101
\newcommand{\DDelta}{{\mathbf \varDelta}}

\newcommand{\ssdot}{\bullet}
\newcommand{\subdot}{_{\ssdot}}
\newcommand{\supdot}{^{\ssdot}}

\newcommand{\noloc}{\;{:}\,}

\newcommand{\myop}[1]{\mathop{\textstyle #1}\limits}
\newcommand{\putatop}[2]{\genfrac{}{}{0pt}{}{#1}{#2}}

\newcommand{\phat}{^{\wedge}_{p}}

\newcommand{\tofrom}{\xymatrix@C-1pc{\ar@<.5ex>[r]&\ar@<.5ex>[l]}}
\newcommand{\toto}{\xymatrix@C-1pc{\ar@<.5ex>[r]\ar@<-.5ex>[r]&}}

\newcommand{\bA}{{\mathbb{A}}}
\newcommand{\bB}{{\mathbb{B}}}
\newcommand{\bC}{{\mathbb{C}}}
\newcommand{\bE}{{\mathbb{E}}}
\newcommand{\bF}{{\mathbb{F}}}
\newcommand{\bL}{{\mathbb{L}}}
\newcommand{\bN}{{\mathbb{N}}}
\newcommand{\bO}{{\mathbb{O}}}
\newcommand{\bP}{{\mathbb{P}}}
\newcommand{\bQ}{{\mathbb{Q}}}
\newcommand{\bR}{{\mathbb{R}}}
\newcommand{\bS}{{\mathbb{S}}}
\newcommand{\bT}{{\mathbb{T}}}
\newcommand{\bU}{{\mathbb{U}}}
\newcommand{\bZ}{{\mathbb{Z}}}

\DeclareMathAlphabet{\mathscr}{U}{rsfs}{m}{n}
\let\catsymbfont\mathscr

\newcommand{\aC}{{\catsymbfont{C}}}
\newcommand{\aD}{{\catsymbfont{D}}}
\newcommand{\aE}{{\catsymbfont{E}}}
\newcommand{\aM}{{\catsymbfont{M}}}
\newcommand{\aS}{{\catsymbfont{S}}}

\let\opsymbfont\mathcal 

\newcommand{\oA}{{\opsymbfont{A}}}
\newcommand{\oB}{{\opsymbfont{B}}}
\newcommand{\oC}{{\opsymbfont{C}}}
\newcommand{\oI}{{\opsymbfont{I}}}
\newcommand{\oE}{{\opsymbfont{E}}}
\newcommand{\oK}{{\opsymbfont{K}}}
\newcommand{\oL}{{\opsymbfont{L}}}
\newcommand{\oO}{{\opsymbfont{O}}}
\newcommand{\oP}{{\opsymbfont{P}}}
\newcommand{\oS}{{\opsymbfont{S}}}
\newcommand{\oU}{{\opsymbfont{U}}}
\newcommand{\oX}{{\opsymbfont{X}}}
\newcommand{\oY}{{\opsymbfont{Y}}}

\newcommand{\iso}{\cong}     
\newcommand{\sma}{\wedge}    

\renewcommand{\to}{\mathchoice{\longrightarrow}{\rightarrow}{\rightarrow}{\rightarrow}}
\newcommand{\sto}{\rightarrow}
\newcommand{\from}{\mathchoice{\longleftarrow}{\leftarrow}{\leftarrow}{\leftarrow}}
\newcommand{\overto}[1]{\xrightarrow{\,#1\,}}
\newcommand{\overfrom}[1]{\xleftarrow{\,#1\,}}

\def\quickop#1{\expandafter\DeclareMathOperator\csname #1\endcsname{#1}}
\quickop{Id}\quickop{id}\quickop{Ho}
\quickop{Hom}\quickop{Tor}
\quickop{colim}
\quickop{Sym}
\quickop{op}
\quickop{diag}
\quickop{holim}

\newtheorem{thm}[equation]{Theorem}
\newtheorem{ethm}[equation]{Example Theorem}
\newtheorem{add}[equation]{Addendum}
\newtheorem{cor}[equation]{Corollary}
\newtheorem{lem}[equation]{Lemma}
\newtheorem{prop}[equation]{Proposition}

\theoremstyle{definition}
\newtheorem{defn}[equation]{Definition}
\newtheorem{cons}[equation]{Construction}

\newtheorem{notn}[equation]{Notation}

\theoremstyle{remark}

\newtheorem{example}[equation]{Example}

\numberwithin{equation}{section}

\chardef\tildechar"7E
\ifx\url\undefined\usepackage{url}\fi

\newcommand{\term}[1]{\emph{#1}}

\renewcommand{\theenumi}{\roman{enumi}}

\newcommand{\deflist}{\renewcommand{\theenumi}{\alph{enumi}}}
\long\def\ignore#1\endignore{\relax}

\ifx\texorpdfstring\undefined
\def\texorpdfstring#1#2{#1}
\fi

\begin{document}

\title{Operads and Operadic Algebras in Homotopy Theory}
\author{Michael A. Mandell}
\address{Department of Mathematics, Indiana University, Bloomington, IN}
\email{mmandell@indiana.edu}

\date{\draftdate}

\maketitle
\setcounter{tocdepth}{1}
\tableofcontents

\section{Introduction}

Operads first appear in the book \textit{Geometry of
Iterated Loop Spaces} by J.\ P.\ May~\cite{May-GILS}, though Boardman
and Vogt had earlier implicitly defined a mathematically equivalent
notion as a ``PROP in standard form''~\cite[\S2]{BV-Bull}.  In those
works, operads and operadic algebra structures provide a recognition
principle and a delooping machine for $n$-fold loop spaces and
infinite loop spaces.  The basic idea is that an operad should
encode the operations in some kind of homotopical algebraic structure.
For example, an $n$-fold loop space $\Omega^{n}X$ comes with
$n$ different multiplications $(\Omega^{n}X)^{2}\to \Omega^{n}X$, which
can be iterated and generalized to a space of $m$-ary maps
$\oC_{n}(m)$ (from $(\Omega^{n}X)^{m}$ to $\Omega^{n}X$); here $\oC_{n}$
is the Boardman-Vogt little $n$-cubes operad (see
Construction~\ref{cons:lilcubes} and Section~\ref{sec:loop} below).
The content of the recognition theorem is that $\oC_{n}$ specifies a
structure that is essentially equivalent to the structure of an
$n$-fold loop space for connected spaces.  It was clear even at the
time of introduction 
that operads were a big idea and in the almost 50 years since then,
operads have found a wide range of other uses in a variety of areas of
mathematics: a quick MathSciNet search for papers since 2015 with
``operad'' in the title comes up with papers in combinatorics,
algebraic geometry, nonassociative algebra, geometric group theory,
free probability, mathematical modeling, and physics, as well as in
algebraic topology and homological algebra.

Even the topic of operads in algebraic topology is too broad to cover
or even summarize in a single article.  This expository article
concentrates on what the author views as the basic topics in the
homotopy theory of operadic algebras: the definition of operads, the
definition of algebras over operads, structural aspects of categories
of algebras over operads, model structures on algebra categories, and
comparison of algebra categories when changing operad or underlying
category.  In addition, it includes two applications of the
theory: The original application to $n$-fold loop spaces, and an
application to algebraic models of homotopy types (chosen purely on
the basis of author bias).  This leaves out a long list of other
topics that could also fit in this handbook, such as model structures
on operads, Koszul duality, deformation theory and Quillen
(co)homology, multiplicative structures in stable homotopy theory (for
example, on Thom spectra, $K$-theory spectra, etc.), Deligne and
Kontsevich conjectures, string topology,
factorization homology, construction of moduli spaces, and Goodwillie
calculus, just to name a few areas.

\subsection*{Notation and conventions}
Although we concentrate on operads and operadic algebras in
topology, much of the
background applies very generally.  Because of this and because we
will want to discuss both the case of spaces and the case of spectra,
we will use neutral notation: let $\mC$ denote a symmetric monoidal
category \cite[\S1.4]{Kelly-BasicConcepts}, writing $\mtimes$ for the
monoidal product and $\mS$ for the unit.  (We will uniformly omit
notation for associativity isomorphisms and typically omit notation
for commutativity isomorphisms, but when necessary, we will write
$\pmap_{\sigma}$ for the commutativity isomorphism associated to a
permutation $\sigma$.)  Usually, we will want $\mC$ to have coproducts
and sometimes more general colimits, which we will expect to commute
with $\mtimes$ on each side (keeping the other side fixed).  This
exactness of $\mtimes$ is automatic if the monoidal structure is
closed~\cite[\S1.5]{Kelly-BasicConcepts}, i.e., if for each fixed
object $X$ of $\mC$, the functor $(-)\mtimes X$ has a right adjoint;
this is often convenient to assume, and when we do, we will use
$F(X,-)$ for the right adjoint.  The three basic classes of examples
to keep in mind are:
\begin{enumerate}
\item ``Convenient categories of topological spaces'' including
compactly generated weak Hausdorff spaces~\cite{McCord-ISP}; then $\mtimes$
is the categorical product, $\mS$ is the final object (one point
space), and $F(X,Y)$ is 
the function space, often written $Y^{X}$.
\item ``Modern categories of spectra'' including EKMM
$S$-modules~\cite{EKMM}, symmetric spectra~\cite{HSS}, and orthogonal
spectra~\cite{MMSS}; then $\mtimes$ is the smash product, $\mS$ is the
sphere spectrum, and $F(-,-)$ is the function spectrum.
\item The category of chain complexes of modules over a commutative ring
$R$; then $\mtimes$ is the tensor product over $R$, $\mS$ is the
complex $R$ concentrated in degree zero, and $F(-,-)$ is the
$\Hom$-complex $\Hom_{R}(-,-)$.
\end{enumerate}  
(We now fix a convenient category of spaces and just call it ``the
category of spaces'' and the objects in it ``spaces'', ignoring
the classical category of topological spaces.)

In the context of operadic algebras in spectra (i.e., (ii) above), it
is often technically convenient to use operads of spaces.  However,
for uniformity of exposition, we have written this article in terms of
operads internally in $\mC$.  The unreduced suspension functor
$\Sigma^{\infty}_{+}(-)$ converts operads in spaces to operads in the
given category of spectra.

\subsection*{Outline}
The basic idea of an operad is that the pieces of it should
parametrize a class of $m$-ary operations.  From this perspective, the
fundamental example of an operad is the \term{endomorphism operad} of
an object $X$,
\[
\oEnd_{X}(m):=F(X^{(m)},X)\qquad \qquad
X^{(m)}:=\underbrace{X\mtimes \dotsb \mtimes X}_{m\text{ factors}},
\]
which parametrizes all $m$-ary maps from $X$ to itself.  Abstracting
the symmetry and composition properties leads to the definition of
operad in~\cite{May-GILS}.  We review this definition in
Section~\ref{sec:def1}.

Section~\ref{sec:ainf} presents some basic examples of operads important
in topology, including some $A_{\infty}$ operads, $E_{\infty}$
operads, and $E_{n}$ operads. 

May chose the term ``operad'' to match the term ``monad''
(see~\cite{May-MacLane}), to show their close connection.  Basically,
a monad is an abstract way of defining some kind of structure on
objects in a category, and an operad gives a very manageable kind
of monad. Section~\ref{sec:monad} reviews the monad associated to an operad and
defines algebras over an operad.

Section~\ref{sec:modules} gives the basic definition of a module over
an operadic algebra and reviews the basics of the homotopy theory of
module categories.

Section~\ref{sec:limit} discusses limits and colimits in categories of
operadic algebras.  It includes a general filtration
construction that often provides the key tool to study pushouts of
operadic algebras homotopically in terms of colimits in the underlying
category. Section~\ref{sec:enrich} discusses when
categories of operadic algebras are enriched, and in the case of
categories of algebras enriched over spaces, discusses the geometric
realization of simplicial and cosimplicial algebras.  Although these
sections may seem less basic and more technical than the previous
sections, the ideas here provide the tools necessary 
for further work with operadic algebras using the modern methods of
homotopy theory.  

Model structures on categories of operadic algebras provide a
framework for proving comparison theorems and rectification theorems.
Section~\ref{sec:model} reviews some aspects of model
category theory for categories of operadic algebras.
In the terminology of this article, a \term{comparison theorem} is an
equivalence of homotopy theories between categories of algebras over
different operads that are equivalent in some sense (for example,
between categories of algebras over different $E_{\infty}$ operads) or between
categories of algebras over equivalent base categories (for example,
$E_{\infty}$ algebras in spaces versus $E_{\infty}$ algebras in
simplicial sets).  A \term{rectification theorem} is a comparison
theorem when one of the operads is discrete in some sense: 
a comparison theorem for the category of algebras over an 
$A_{\infty}$ operad and the category of associative algebras is an example
of a rectification theorem, as is the comparison theorem for
$E_{\infty}$ algebras and commutative algebras in modern categories of
spectra.  Section~\ref{sec:compare} discusses these and other examples
of comparison and rectification theorems.  In both
Sections~\ref{sec:model} and~\ref{sec:compare}, 
instead of stating theorems of maximal generality, we have chosen to provide
``Example Theorems'' that capture some examples of particular interest
in homotopy theory and stable homotopy theory.  Both the statements
and the arguments provide examples: the arguments apply or can be
adapted to apply in a wide range of generality.

The Moore space is an early rectification technique (pre-dating operads
and $A_{\infty}$ monoids) for producing a genuine associative monoid
version of the loop space; the construction applies generally to
a little $1$-cubes algebra to produce an associative algebra that we
call the \term{Moore algebra}.  The
concept of modules over an operadic algebra leads to another way of
producing an associative algebra, called the \term{enveloping
algebra}.   Section~\ref{sec:rectass} compares these constructions and
the rectification of $A_{\infty}$ algebras constructed in
Section~\ref{sec:compare}. 

Sections~\ref{sec:loop} and~\ref{sec:cochains} review two significant
applications
of the theory of operadic algebras. Section~\ref{sec:loop} reviews the
original application: the
theory of iterated loop spaces and the recognition principle in terms
of $E_{n}$ algebras.  Section~\ref{sec:cochains} reviews the
equivalence between the rational and $p$-adic homotopy theory of
spaces with the homotopy theory of $E_{\infty}$ algebras.

\subsection*{Acknowledgments}
The author benefited from conversations and advice from Clark Barwick,
Agn\`es Beaudry, Julie Bergner,  Myungsin Cho,
Bj{\o}rn Dundas, Tyler Lawson, Andrey Lazarev, Amnon Neeman, Brooke
Shipley, and Michael Shulman while working on this chapter.  The
author thanks Peter May for his mentorship in the 1990s (and beyond)
while learning these topics and for help straightening out some of the
history described here.  The author thanks the Isaac Newton Institute
for Mathematical Sciences for support and hospitality during the
program ``Homotopy harnessing higher structures'' (HHH) when work on
this chapter was undertaken; this work was supported by: EPSRC Grant
Number EP/R014604/1. The author was supported in part by NSF grants
DMS-1505579 and DMS-1811820 while working on this project.  Finally,
the author thanks Andrew Blumberg for extensive editorial advice.

\section{Operads and Endomorphisms}
\label{sec:def1}

We start with the definition of an operad.  
The collection of $m$-ary endomorphism objects
$\oEnd_{X}(m)=F(X^{(m)},X)$ provides the prototype for the definition,
and we use its intrinsic structure to motivate and explain it. 
Although the endomorphism objects only make sense when the symmetric
monoidal category is ``closed'' (which means
that function objects exist), the definition of operad will not
require or assume function objects, nor will the definition of
operadic algebra in Section~\ref{sec:monad}.
To take in the picture, it might be best just to take $\mC$ to be the
category of spaces, the category of vector spaces over a field, or the
category of sets on first introduction to this material.

In our basic classes of examples, and more generally as a principle of
enriched category theory, function objects behave like sets of
morphisms: the counit of the defining adjunction
\[
F(X,Y)\mtimes X\to Y
\]
is often called the \textit{evaluation map} (and denoted $ev$). It
allows ``element-free'' definition and study of composition: iterating
evaluation maps 
\[
F(Y,Z)\mtimes F(X,Y)\mtimes X\to F(Y,Z)\mtimes Y\to Z
\]
induces (by adjunction) a \term{composition map}
\[
\circ \colon F(Y,Z)\mtimes F(X,Y)\to F(X,Z).
\]
One can check just using the basic properties of adjunctions that this
composition is associative in the obvious sense.  It is also unital:
the identity element of $\mC(X,X)$ specifies a map $1_{X}\colon \mS\to
F(X,X)$,
\[
\id_{X} \in \mC(X,X)\iso \mC(\mS \mtimes X,X)\iso \mC(\mS,F(X,X)), 
\]
where the first isomorphism is induced by the unit isomorphism;
essentially by construction, the composite
\[
\mS\mtimes X\overto{1_{X}\mtimes \id_{X}} F(X,X)\mtimes X\overto{ev} X
\]
is the unit isomorphism.  It follows that the diagram
\[
\xymatrix{%
\mS\mtimes F(X,Y)\ar[r]^{\iso}\ar[d]_{1_{Y}\mtimes \id_{F(X,Y)}}
&F(X,Y)\ar@{=}[d]
&F(X,Y)\mtimes \mS\ar[l]_{\iso}\ar[d]^{\id_{F(X,Y)}\mtimes 1_{X}}\\
F(Y,Y)\mtimes F(X,Y)\ar[r]_-{\circ}&F(X,Y)&F(X,Y)\mtimes F(X,X)\ar[l]^-{\circ}
}
\]
commutes, where the top-level isomorphisms are the unit isomorphisms.
More is true: the function objects enrich the category $\mC$ over
itself, and the $\mtimes,F$ parametrized adjunction is itself enriched
\cite[\S1.5--6]{Kelly-BasicConcepts}. 

In the case when $\mC$ is the category of spaces, the
evaluation map is just the map that evaluates functions on their
arguments; thinking in these terms will make the formulas and checks
clearer for the reader not used to working with adjunctions.  Since in
the category of spaces $\mS$ is the one point space, a map out of
$\mS$ just picks out an element of the target space and the map $\mS\to
F(X,X)$ is just the map that picks out the identity map of $X$.

The basic compositions above generalize to associative and unital
$m$-ary compositions; now for simplicity and because it is the main
case of interest here, we restrict to considering a fixed object $X$.
The $m$-ary composition takes the form
\[
F(X^{(m)},X)\mtimes (F(X^{(j_{1})},X)\mtimes \dotsb \mtimes
F(X^{(j_{m})},X))\to F(X^{(j)},X)
\]
where $j=j_{1}+\dotsb+j_{m}$ and (as in the introduction) $X^{(m)}$
denotes the $m$th $\mtimes$ power of $X$; we think of the $m$-ary
composition as plugging in the $m$ $j_{i}$-ary maps into the first
$m$-ary map; it is adjoint to the map
\begin{multline*}
F(X^{(m)},X)\mtimes F(X^{(j_{1})},X)\mtimes \dotsb \mtimes
F(X^{(j_{m})},X) \mtimes X^{(j)}
\iso\\
F(X^{(m)},X)\mtimes F(X^{(j_{1})},X)\mtimes \dotsb \mtimes
F(X^{(j_{m})},X) \mtimes X^{(j_{1})}\mtimes \dotsb \mtimes X^{(j_{m})}
\to X
\end{multline*}
that does the evaluation map
\[
F(X^{(j_{i})},X)\mtimes X^{(j_{i})}\to X,
\]
then collects the resulting $m$ factors of $X$ and does the evaluation
map 
\[
F(X^{(m)},X)\mtimes X^{(m)}\to X.
\]
In this double evaluation, implicitly we have shuffled some of the
factors of $X$ past some of the endomorphism objects, but we take care
not to permute factors of $X$ among themselves or the endomorphism
objects among themselves.
This defines a composition map 
\[
\Gamma^{m}_{j_{1},\dotsc,j_{m}}\colon \oEnd_{X}(m)\mtimes
\oEnd_{X}(j_{1})\mtimes \dotsb \mtimes \oEnd_{X}(j_{m})\to \oEnd_{X}(j).
\]
The composition is associative and unital in the obvious sense (which
we write out in the definition of an operad,
Definition~\ref{def:operad}, below).  

We now begin systematically writing $\oEnd_{X}(m)$ for
$F(X^{(m)},X)$.  We note that $\oEnd_{X}(m)=F(X^{(m)},X)$ has a right
action by the symmetric group $\Sigma_{m}$ induced by the left action
of $\Sigma_{m}$ on $X^{(m)}$ corresponding to permuting the
$\mtimes$-factors.  In general, for a permutation $\sigma$, we 
write $\pmap_{\sigma}$ for the map that permutes $\mtimes$-factors and
$\amap_{\sigma}$ for the action of $\sigma$ on $\oEnd_{X}(m)$, i.e., the
map that does $\pmap_{\sigma}$ on the domain of $\oEnd_{X}(m)=F(X^{(m)},X)$.  We now study what happens when we permute the
various factors in the formula for $\Gamma$ above.  (As these are a
bit tricky, we do the formulas out here and repeat them below
in the definition of an operad, Definition~\ref{def:operad}.)

First consider what happens when we permute the factors of $X$.  We
have nothing to say for an arbitrary permutation of the factors of
$X$, but in the composition $\Gamma^{m}_{j_{1},\dotsc,j_{m}}$, we can
say something for a permutation that permutes the factors only within
their given blocks of size $j_{1},\dotsc,j_{m}$, i.e., when the
overall permutation $\sigma$ of all $j$ factors is the block sum of
permutations $\sigma_{1}\oplus \dotsb \oplus \sigma_{m}$ with
$\sigma_{i}$ in $\Sigma_{j_{i}}$.  By extranaturality, performing the
right action of $\sigma_{i}$ on $\oEnd_{X}(j_{i})$ and evaluating is
the same as applying the left action of $\sigma_{i}$ on $X^{(j_{i})}$
and evaluating.  It follows that the composition
$\Gamma^{m}_{j_{1},\dotsc,j_{m}}$ is $(\Sigma_{j_{1}}\times \dotsb
\times \Sigma_{j_{m}})$-equivariant where we use the
$\Sigma_{j_{i}}$-actions on the $\oEnd_{X}(j_{i})$'s in the source and
block sum with the $\Sigma_{j}$-action on $\oEnd_{X}(j)$ on the
target.

Permuting the endomorphism object factors is easier to understand when
we also permute the corresponding factors of $X$.  In the context of
$\Gamma^{m}_{j_{1},\dotsc,j_{m}}$, for $\sigma$ in $\Sigma_{m}$, let
$\sigma_{j_{1},\dotsc,j_{m}}$ be the element of $\Sigma_{j}$ that
permutes the blocks $X^{(j_{1})}$,\dots, $X^{(j_{m})}$ as $\sigma$
permutes $1$,\dots,$m$.  So, for example, if $m=3$, $j_{1}=1$,
$j_{2}=3$, $j_{3}=2$, and $\sigma=(23)$, then $\sigma_{1,3,2}$ is the
permutation
\[
(23)_{1,3,2}=
\left\{
\vcenter{\tiny
\xymatrix@C-2pc@R-1.5pc{%
1\ar[d]&2\ar[d]&3\ar[d]&4\ar[d]&5\ar[d]&6\ar[d]\\
1&5&6&2&3&4
}}%
\right\} = (25364).
\]
In 
$\oEnd_{X}(j_{1})\mtimes \dotsb \mtimes \oEnd_{X}(j_{m})\mtimes X^{(j)}$,
if we apply $\sigma$ to permute the endomorphism object factors and
$\sigma_{j_{1},\dotsc,j_{m}}$ to permute the $X$ factors, then
evaluation pairs the same factors as with no permutation and the
diagram
\[
\xymatrix{%
(\oEnd_{X}(j_{1})\mtimes \dotsb \mtimes \oEnd_{X}(j_{m}))\mtimes X^{(j)}
\ar[r]^-{ev}
\ar[d]_{\pmap_{\sigma} \mtimes \pmap_{\sigma_{j_{1},\dotsc,j_{m}}}}
&X^{(m)}\ar[d]^{\pmap_{\sigma}}\\
(\oEnd_{X}(j_{\sigma^{-1}(1)})\mtimes \dotsb \mtimes 
\oEnd_{X}(j_{\sigma^{-1}(m)}))\mtimes X^{(j)}\ar[r]_-{ev}
&X^{(m)}
}
\]
commutes. This now tells us what happens with
$\Gamma^{m}_{j_{1},\dotsc,j_{m}}$ and the permutation action on
$\oEnd_{X}(n)$: the composite of the right action of $\sigma$ on
$\oEnd_{X}(m)$ with $\Gamma^{m}_{j_{1},\dotsc,j_{m}}$
\begin{multline*}
\oEnd_{X}(m)\mtimes 
(\oEnd_{X}(j_{1})\mtimes \dotsb \mtimes \oEnd_{X}(j_{m}))\\
\overto{\amap_{\sigma} \mtimes \id}
\oEnd_{X}(m)\mtimes 
(\oEnd_{X}(j_{1})\mtimes \dotsb \mtimes \oEnd_{X}(j_{m}))
\overto{\Gamma^{m}_{j_{1},\dotsc,j_{m}}}
\oEnd_{X}(j) 
\end{multline*}
is equal to the composite of the $\mtimes$-permutation $\pmap_{\sigma}$ on the
$\oEnd(j_{i})$'s, the composition map
$\Gamma^{m}_{j_{\sigma^{-1}(1)},\dotsc,j_{\sigma^{-1}(m)}}$, and the
right action of $\sigma_{j_{1},\dotsc,j_{m}}$ on $\oEnd_{X}(j)$
\begin{multline*}
\oEnd_{X}(m)\mtimes 
(\oEnd_{X}(j_{1})\mtimes \dotsb \mtimes \oEnd_{X}(j_{m}))\\
\overto{\id\mtimes \pmap_{\sigma}}
\oEnd_{X}(m)\mtimes 
(\oEnd_{X}(j_{\sigma^{-1}(1)})\mtimes \dotsb \mtimes \oEnd_{X}(j_{\sigma^{-1}(m)}))\\
\overto{\Gamma^{m}_{j_{\sigma^{-1}(1)},\dotsc,j_{\sigma^{-1}(m)}}}
\oEnd_{X}(j) \overto{\amap_{\sigma_{j_{1},\dotsc,j_{m}}}}\oEnd_{X}(j).
\end{multline*}
See Figure~\ref{fig:operadhardperm} on
p.~\pageref{fig:operadhardperm} for this equation written as a
diagram.

Although we did not emphasize this above, we need to allow any of $m$,
$j_{1},\dotsc,j_{m}$, or $j$ to be zero, where we understand empty
$\mtimes$-products to be the unit $\mS$.  The formulations above still
work with this extension, using the unit isomorphism where necessary.
The purpose of allowing these ``zero-ary'' operations is that it
allows us to encode a unit object into the structure: For example, in
the context of spaces $\mS$ is the one point space $*$ and to describe
the structure of a topological monoid, not only do we need the binary
operation $X\times X\to X$, but we also need the zero-ary operation
$*\to X$ for the unit.

Rewriting the properties of $\oEnd_{X}$ above as a definition, we get
an element-free version of the definition of operad of
May~\cite[1.2]{May-GILS}.\footnote{In the original definition, May
required $\oO(0)=\mS$ in order to provide $\oO$-algebras with units,
which was desirable in the iterated loop space context,
but standard convention has since dropped this requirement to allow
non-unital algebras and other unit variants.} 

\begin{defn}\label{def:operad}
An operad in a symmetric monoidal category $\mC$ consists of a sequence of
objects $\oO(m)$, $m=0,1,2,3,\dotsc$, together with:
\begin{enumerate}\deflist
\item\label{part:operad:action} A right action of the symmetric group $\Sigma_{m}$ on $\oO(m)$
for all $m$,
\item\label{part:operad:unitmap} A \textit{unit map} $1\colon \mS\to \oO(1)$, and
\item\label{part:operad:comp} A \textit{composition rule} 
\[
\Gamma^{m}_{j_{1},\dotsc,j_{m}}\colon \oO(m)\mtimes \oO(j_{1})\mtimes
\dotsb \mtimes \oO(j_{m})\to \oO(j)
\]
for every $m$, $j_{1},\dotsc,j_{m}$, where $j=j_{1}+\dotsb+j_{m}$, typically written $\Gamma$ when $m$ and
$j_{1},\dotsc,j_{m}$ are understood or irrelevant,
\end{enumerate}
satisfying the following conditions:
\begin{enumerate}
\item\label{part:operad:assoc} The composition rule $\Gamma$ is
associative in the sense that for any $m$, $j_{1},\dotsc,j_{m}$ and
$k_{1},\dotsc,k_{j}$, letting $j=j_{1}+\dotsb +j_{m}$,
$k=k_{1}+\dotsb+k_{j}$, $t_{i}=j_{1}+\dotsb+j_{i-1}$ (with $t_{1}=0$),
and $s_{i}=k_{t_{i}+1}+\dotsb+k_{t_{i}+j_{i}}$, the equation
\begin{multline*}
\hspace{3\listparindent}
\Gamma^{j}_{k_{1},\dotsc,k_{j}}\circ
(\Gamma^{m}_{j_{1},\dotsc,j_{m}}\mtimes \id_{\oO(k_{1})}\mtimes \dotsb
\mtimes \id_{\oO(k_{j})})\\
=
\Gamma^{m}_{s_{1},\dotsc,s_{m}}\circ 
(\id_{\oO(m)}\mtimes\ \Gamma^{j_{1}}_{k_{1},\dotsc,k_{j_{1}}}\mtimes \dotsb
\mtimes \Gamma^{j_{m}}_{k_{t_{m}+1},\dotsc,k_{j}})\circ c
\end{multline*}
holds in the set of maps 
\[
\oO(m)\mtimes \oO(j_{1})\mtimes \dotsb \mtimes \oO(j_{m})\mtimes
\oO(k_{1})\mtimes \dotsb \mtimes \oO(k_{j})\to \oO(k)
\]
where $c$ is the $\mtimes$-permutation
\begin{multline*}
\hspace{3\listparindent}
\oO(m)\mtimes \oO(j_{1})\mtimes \dotsb \mtimes \oO(j_{m})\mtimes
\oO(k_{1})\mtimes \dotsb \mtimes \oO(k_{j})
\to\\
\hspace{3\listparindent}
\oO(m)\mtimes 
(\oO(j_{1})\mtimes \oO(k_{1})\mtimes \dotsb \mtimes \oO(k_{j_{1}}))
\mtimes\dotsb\\
\dotsb \mtimes (\oO(j_{m})\mtimes \oO(k_{t_{m}+1})\mtimes \dotsb \mtimes
\oO(k_{j}))
\end{multline*}
that shuffles the $\oO(k_{\ell})$'s and $\oO(j_{i})$'s as displayed
(see Figure~\ref{fig:operadassoc} on p.~\pageref{fig:operadassoc} for the diagram);

\item\label{part:operad:unit} The unit map $1$ is a left and right unit for the composition
rule $\Gamma$ in the sense that $\Gamma^{1}_{m}\circ (1\mtimes \id)$
\[
\mS\mtimes \oO(m)\overto{1\mtimes \id}\oO(1)\mtimes \oO(m)\overto{\Gamma^{1}_{m}}\oO(m)
\]
is the unit isomorphism and $\Gamma^{m}_{1,\dotsc,1}\circ (\id\mtimes 1^{(m)})$
\[
\oO(m)\mtimes \mS^{(m)}
\overto{\id\mtimes 1^{(m)}}
\oO(m)\mtimes \oO(1)^{(m)}
\overto{\Gamma^{m}_{1,\dotsc,1}}\oO(m)
\]
is the iterated unit isomorphism for $\oO(m)$ for all $m$;
\item\label{part:operad:easyperm} The map $\Gamma^{m}_{j_{1},\dotsc,j_{m}}$ is
$(\Sigma_{j_{1}}\times \dotsb \times \Sigma_{j_{m}})$-equivariant for
the block sum inclusion of $\Sigma_{j_{1}}\times \dotsb \times
\Sigma_{j_{m}}$ in $\Sigma_{j}$; and
\item\label{part:operad:hardperm} For any $m$, $j_{1},\dotsc,j_{m}$ and any $\sigma \in
\Sigma_{m}$, the equation
\begin{multline*}
\hspace{3\listparindent}
\Gamma^{m}_{j_{1},\dotsc,j_{m}}\circ 
(\amap_{\sigma}\mtimes \id_{\oO(j_{1})}\mtimes \dotsb \mtimes\id_{\oO(j_{m})})
\\
=\amap_{\sigma_{j_{1},\dotsc,j_{m}}}\circ 
\Gamma^{m}_{j_{\sigma^{-1}(1)},\dotsc,j_{\sigma^{-1}(m)}}\circ 
(\id_{\oO(m)}\mtimes \pmap_{\sigma})
\end{multline*}
holds in the set of maps
\[
\oO(m)\mtimes \oO(j_{1})\mtimes \dotsb \mtimes \oO(j_{m})\to \oO(j)
\]
where $\sigma_{j_{1},\dotsc,j_{m}}$ denotes the block permutation
in $\Sigma_{j}$ corresponding to $\sigma$ on the blocks of size
$j_{1},\dotsc,j_{m}$, $\amap$ denotes the right action
of~\eqref{part:operad:action}, and $\pmap_{\sigma}$ denotes the
$\mtimes$-permutation corresponding to $\sigma$ (see
Figure~\ref{fig:operadhardperm} 
on p.~\pageref{fig:operadhardperm} for the diagram).
\end{enumerate}
\end{defn}

\begin{figure}
\hrule
\[
\xymatrix@C-4pc@R-1pc{%
&\oO(j)\mtimes \oO(k_{1})\mtimes\dotsb \mtimes\oO(k_{j})
\ar[dd]^{\Gamma^{j}_{k_{1},\dotsc,k_{j}} }\\
\oO(m)\mtimes \oO(j_{1})\mtimes \dotsb \mtimes \oO(j_{m})\mtimes
\oO(k_{1})\mtimes \dotsb \mtimes \oO(k_{j})
\ar `u[u] [ur]^-{\Gamma^{m}_{j_{1},\dotsc,j_{m}}\mtimes \id\mtimes \dotsb\mtimes \id}
\ar[dd]_{c}\\
&\oO(k)\\
\txt{{$\oO(m)\mtimes \left(\begin{aligned}
&\oO(j_{1})\mtimes \oO(k_{1})\mtimes \dotsb \mtimes \oO(k_{j_{1}})
\mtimes\dotsb\cr
&\quad\dotsb\mtimes  \oO(j_{m})\mtimes \oO(k_{t_{m}+1})\mtimes \dotsb \mtimes
\oO(k_{j})
\end{aligned}\right)$}}
\ar `d[dr] [dr]_-{\id\mtimes
\Gamma^{j_{1}}_{k_{1},\dotsc,k_{j_{1}}}\mtimes \dotsb 
\mtimes \Gamma^{j_{m}}_{k_{t_{m}+1},\dotsc,k_{j}} }\\
&\oO(m)\mtimes \oO(s_{1})\mtimes \dotsb \mtimes 
\oO(s_{m})\ar[uu]_{\Gamma^{m}_{s_{1},\dotsc,s_{m}}}
}
\]
\caption{The diagram for \ref{def:operad}.\eqref{part:operad:assoc}}
\label{fig:operadassoc}
\vskip 1ex
\begin{minipage}{.9\hsize}
Here $c$ is
the $\mtimes$-permutation that shuffles $\oO(k_{\ell})$'s past
$\oO(j_{i})$'s as displayed, $j=j_{1}+\dotsb+j_{m}$,
$t_{i}=j_{1}+\dotsb+j_{i-1}$ (with $t_{1}=0$),
$s_{i}=k_{t_{i}+1}+\dotsb+k_{t_{i}+j_{i}}$, and $k=k_{1}+\dotsb+k_{j}=s_{1}+\dotsb+s_{m}$.
\end{minipage}
\vskip 1em
\hrule
\end{figure}

\begin{figure}
\hrule
\[
\xymatrix@R-.5pc{%
\oO(m)\mtimes \oO(j_{1})\mtimes \dotsb \mtimes
\oO(j_{m})\ar[r]^{\amap_{\sigma}\mtimes\id\mtimes \dotsb \mtimes \id}
\ar[dd]_{\id\mtimes \pmap_{\sigma}}
&\oO(m)\mtimes \oO(j_{1})\mtimes \dotsb \mtimes \oO(j_{m})
\ar[d]^{\Gamma^{m}_{j_{1},\dots,j_{m}}}\\
&\oO(j)\\
\oO(m)\mtimes \oO(j_{\sigma^{-1}(1)})\mtimes \dotsb \mtimes
\oO(j_{\sigma^{-1}(m)})
\ar[r]_-{\Gamma^{m}_{j_{\sigma^{-1}(1)},\dotsc,j_{\sigma^{-1}(m)}}}
&\oO(j)\ar[u]_-{\amap_{\sigma_{j_{1},\dots,j_{m}}}}
}
\]
\caption{The diagram for \ref{def:operad}.\eqref{part:operad:hardperm}}
\label{fig:operadhardperm}
\vskip 1ex
\begin{minipage}{.9\hsize}
Here $\sigma\in \Sigma_{m}$, $\pmap_{\sigma}$ is the $\mtimes$-permutation
corresponding to $\sigma$, $\sigma_{j_{1},\dotsc,j_{m}}\in \Sigma_{j}$\break
is the block permutation performing $\sigma$ on blocks of sizes
$j_{1},\dotsc,j_{m}$,\break $j=j_{1}+\dotsb+j_{m}$, and $\amap$ denotes the
$\Sigma_{m}$ action on $\oO(m)$ and the $\Sigma_{j}$-action on
$\oO(j)$. 
\end{minipage}
\vskip 1em
\hrule
\end{figure}

A map of operads consists of a map of each object that commutes with
the structure:

\begin{defn}\label{def:opmap}
A map of operads $(\{\oO(m)\},1,\Gamma)\to (\{\oO'(m)\},1',\Gamma')$
consists of $\Sigma_{m}$-equivariant maps $\phi_{m}\colon \oO(m)\to
\oO'(m)$ for all $m$ such that 
\[
\Gamma^{\prime m}_{j_{1},\dotsc,j_{m}} \circ (\phi_{m}\mtimes \phi_{j_{1}}\mtimes \dotsb
\mtimes \phi_{j_{m}})
=
\phi_{j}\circ \Gamma^{m}_{j_{1},\dotsc,j_{m}}
\]
for all $m$, $j_{1},\dotsc,j_{m}$ and $1'=\phi_{1}\circ 1$; in commuting diagrams:
\[
\def\mystrut{\vphantom{\oO'(m)\mtimes \oO'(j_{1})\mtimes \dotsb \mtimes \oO'(j_{m})}}
\xymatrix@C+1pc{%
\oO(m)\mtimes \oO(j_{1})\mtimes \dotsb \mtimes \oO(j_{m})
\mystrut
\ar[r]^-{\Gamma^{m}_{j_{1},\dotsc,j_{m}}}
\ar[d]_{\phi_{m}\mtimes \phi_{j_{1}}\mtimes \dotsb \mtimes \phi_{j_{m}}}
&\oO(j)\ar[d]^{\phi_{j}}\\
\oO'(m)\mtimes \oO'(j_{1})\mtimes \dotsb \mtimes \oO'(j_{m})
\ar[r]_-{\Gamma^{\prime m}_{j_{1},\dotsc,j_{m}}}
&\oO'(j)
}
\qquad 
\xymatrix@C-.5pc{%
&\mS\mystrut\ar[dl]_{1}\ar[dr]^{1'}\\
\oO(1)\ar[rr]_{\phi_{1}}\mystrut&\relax\hskip-1pc&\oO'(1).
}
\]
\end{defn}

The endomorphism operad $\oEnd_{X}$ gives an example of an operad in
any closed symmetric monoidal category (for any object $X$).  Here are
some additional important examples.

\begin{example}[The identity operad]
Assume the
symmetric monoidal category $\mC$ has
an initial object $\emptyset$. If $\mtimes$ preserves the initial
object in each variable, $\emptyset\mtimes (-)\iso \emptyset\iso
(-)\mtimes \emptyset$ (which is automatic in the closed case, i.e.,
when function objects
exist), we also have the example of the \term{identity operad} $\oI$,
which has $\oI(1)=\mS$ (with $1$ the identity) and $\oI(m)$ the
initial object for $m\neq 1$; this is the initial object in the
category of operads. 
\end{example}

\begin{example}[The commutative algebra operad]
The operad $\oCom$
exists in any symmetric monoidal category:
\[
\oCom(m)=\mS
\]
for all $m$ with the trivial symmetric group actions and composition
law $\Gamma$ given by the unit isomorphism; its category of algebras
(see the next section) is isomorphic to the category of commutative
monoids for $\mtimes$ in $\mC$ (defined in terms of the usual diagrams, i.e.,
\cite[VII\S3]{MacLane-Categories} plus commutativity);
see Example~\ref{ex:comm}. 
\end{example}

\begin{example}[The associative algebra operad]\label{ex:Ass}
If $\mC$ has finite coproducts and
$\mtimes$ preserves finite coproducts in each variable, then we
also have the operad $\oAss$: 
\[
\oAss(m)=\myop\coprod\limits_{\Sigma_{m}}\mS
\]
with symmetric group action induced by the natural (right) action of
$\Sigma_{m}$ on $\Sigma_{m}$ and composition law $\Gamma$ induced by
block permutation and block sum of permutations, 
\[
\sigma \in \Sigma_{m}, \tau_{1}\in \Sigma_{j_{1}},\dotsc, \tau_{m}\in \Sigma_{j_{m}}
\mapsto
\sigma_{j_{1},\dotsc,j_{m}}\circ (\tau_{1}\oplus \dotsb \oplus \tau_{m})\in \Sigma_{j}.
\]
Its category of algebras is isomorphic to the category of monoids for
$\mtimes$ in
$\mC$; see Example~\ref{ex:mon}.
\end{example}

For operads like $\oAss$, it is often useful to work in terms of
\textit{non-symmetric operads}, which come without the permutation action.

\begin{defn}\label{def:nonsym}
A \textit{non-symmetric operad} consists of a sequence of objects
$\oO(m)$, $m=0,1,2,3,\dotsc$, together with a unit map and composition
rule as in \ref{def:operad}.\eqref{part:operad:unitmap} and
\eqref{part:operad:comp} satisfying the associativity and unit rules
of \ref{def:operad}.\eqref{part:operad:assoc}
and~\eqref{part:operad:unit}.  A map of non-symmetric operads consists
of a map of their object sequences that commutes with the unit map and
the composition rule.
\end{defn}

Forgetting the permutation action on $\oCom$ gives a non-unital operad
called $\nsAss$ that is the non-symmetric version of the operad
$\oAss$.  In general, under the finite coproduct assumption in Example~\ref{ex:Ass}, given a non-symmetric operad $\nsO$, the product
$\nsO\mtimes \oAss$ has the canonical structure of an operad; it is
the \term{operad associated to $\nsO$}.  In the category of spaces (or sets, but not in the category of abelian groups, the
category of chain complexes, or the various categories of spectra), an
operad $\oO$ comes from a non-symmetric operad exactly when it admits
a map to $\oAss$: the corresponding non-symmetric operad $\nsO$ has
$\nsO(n)$ the subobject that maps to the identity permutation summand
of $\oAss$, and there is a canonical isomorphism $\oO\iso \nsO\mtimes
\oAss$ (that depends only on the original choice of map $\oO\to
\oAss$).

\section{\texorpdfstring{$A_{\infty}$, $E_{\infty}$, and
$E_{n}$}{Ainfty, Einfty, and En} Operads}
\label{sec:ainf}

This section reviews some of the most important classes of examples of
operads in homotopy theory, the $A_{\infty}$, $E_{\infty}$, and
$E_{n}$ operads.  We concentrate on the case of (unbased) spaces, with
some notes about the appropriate definition of such operads in other
contexts.  For example, in stable homotopy theory, the unbased
suspension spectrum functor $\Sigma^{\infty}_{+}$ 
converts model $E_{n}$ operads into operads in the various modern
categories of spectra.  The universal role played by spaces in
homotopy theory typically allows for reasonable definitions of these
classes of operads in any homotopy theoretic setting.

The terminology of $A_{\infty}$ space and the basic model of an
$A_{\infty}$ operad, due to Stasheff~\cite{Stasheff-Ainfty}, preceded
the definition of operad by several years.

\begin{defn}\label{defn:Ainfty}
An $A_{\infty}$ operad in spaces is a non-symmetric operad whose $m$th
space is contractible for all $m$.
\end{defn}

Informally, an operad (with symmetries) is $A_{\infty}$ when there is
an understood isomorphism to the operad associated to some
$A_{\infty}$ operad.  The definition of $A_{\infty}$ operad usually
has a straightforward generalization to other symmetric monoidal
categories with a notion of homotopy theory: contractibility
corresponds to a weak equivalence with the unit $\mS$ of the symmetric
monoidal structure, and we should add the requirement that the
non-symmetric operad composition rule should be a weak equivalence for
all indexes (which is automatic in spaces). One wrinkle is that a
flatness condition may be needed and should be imposed to ensure that
the functor $\nsO(m)\mtimes X^{(m)}$ is weakly equivalent to $X^{(m)}$
(cf.~Section~\ref{sec:compare}); in the case of spaces,
contractibility implicitly
includes such a condition (although in spaces itself, the monoidal product
$\times$ preserves all weak equivalence in each variable).  In
symmetric spectra and orthogonal spectra, a good flatness condition is
to be homotopy equivalent to a cofibrant object; in EKMM
$S$-modules, a good flatness condition is to be homotopy
equivalent to a semi-cofibrant object
(see~\cite[\S6]{LewisMandell-MMMC}).

We have already seen an example of an $A_{\infty}$ operad: the operad
$\nsAss$ is an $A_{\infty}$ operad.  The associahedra $\oK(m)$ of
Stasheff~\cite[I.\S6]{Stasheff-Ainfty} have the structure of a
non-symmetric operad using the insertion maps [\textit{ibid.}] for the
composition rule, and 
this is an example of an $A_{\infty}$ operad. The
Boardman-Vogt little $1$-cubes (non-symmetric) operad $\nsC_{1}$
described below gives a third example.  

Next we discuss $E_{\infty}$ operads.  Recall that a free
$\Sigma_{m}$-cell complex is a space built by cells of the form
$(\Sigma_{m}\times D^{n},\Sigma_{m}\times S^{n-1})$, where $D^{n}$
denotes the unit disk in $\bR^{n}$.  The definition of $E_{\infty}$
operad asks for the constituent spaces to have the
$\Sigma_{m}$-equivariant homotopy type of a free $\Sigma_{m}$-cell
complex and the non-equivariant homotopy type of a point.

\begin{defn}\label{defn:Einfty}
An operad $\oE$ in spaces is an $E_{\infty}$ operad when for each $m$,
its $m$th space is a universal $\Sigma_{m}$ space: $\oE(m)$ has the
$\Sigma_{m}$-equivariant homotopy type of a free $\Sigma_{m}$-cell
complex and is non-equivariantly contractible.
\end{defn}

Unlike the $A_{\infty}$ case, the operad $\oCom$ is not $E_{\infty}$ as
its spaces do not have free actions.  The Barratt-Eccles operad
$\oE\Sigma$ provides an example: 

\begin{example}[The Barratt-Eccles operad]\label{ex:be}
Let $\oE\Sigma(m)$ denote the
nerve of the category $E\Sigma_{m}$ whose set of objects is
$\Sigma_{m}$ and which has a unique map between any two objects. The
symmetric group $\Sigma_{m}$ acts strictly on the category and the
nerve $\oE\Sigma(m)$ inherits a $\Sigma_{m}$-action; moreover, as the
action of $\Sigma_{m}$ on the simplices is free, the simplicial
triangulation of $\oE\Sigma(m)$ has the structure of a free $\Sigma_{m}$-cell
complex.  It is non-equivariantly contractible because every object of
$E\Sigma_{m}$ is a zero object.  The multiplication is induced by an
operad structure on the sequence of categories using block sums of
permutations as in the operad 
structure on $\oAss$.  The resulting operad is called the
\term{Barratt-Eccles operad}.
\end{example}

Boardman and Vogt~\cite[\S2]{BV-Bull} defined another $E_{\infty}$
operad, built out of linear isometries.

\begin{example}[The linear isometries operad]\label{ex:L}
The Boardman-Vogt linear isometries
operad $\oL$ has its $m$th space the space of
linear isometries 
\[
(\bR^{\infty})^{m}=\bR^{\infty}\oplus\dotsb \oplus \bR^{\infty}\to \bR^{\infty}
\]
(where $\bR^{\infty}=\bigcup \bR^{n}$), with operad structure defined
as in the example of an endomorphism operad.  The topology comes from
the identification
\[
\oL(m)=\lim\nolimits_{k} \colim_{n} \oI((\bR^{k})^{m},\bR^{n})
\]
for $\oI((\bR^{k})^{m},\bR^{n})$ the space of linear isometries
$\bR^{k})^{m}\to \bR^{n}$ (with
the usual manifold topology).  The $\Sigma_{m}$-action induced by
the action on the direct sum $(\bR^{\infty})^{m}$ is clearly free;
each $\oI((\bR^{k})^{m},\bR^{n})$ is a $\Sigma_{m}$-manifold, and
$\oL(m)$ is homotopy equivalent to a free $\Sigma_{m}$-cell complex.
Since $\oI((\bR^{k})^{m},\bR^{n})$ is $(n-km-1)$-connected, it follows
that $\oL(m)$ is non-equivariantly contractible.
\end{example}

The Boardman-Vogt little $\infty$-cubes operad $\oC_{\infty}$ described below gives a
third example of an $E_{\infty}$ operad.

The requirement for freeness derives from infinite loop space theory.
As we review in Section~\ref{sec:loop}, infinite loop spaces are
algebras for the little $\infty$-cubes operad $\oC_{\infty}$.  As we
review in Section~\ref{sec:compare}, for any $E_{\infty}$ operad $\oE$
in spaces, the category of $\oE$-algebras has an equivalent homotopy theory to the category
of $\oC_{\infty}$-algebras.  On the other
hand, any algebra in spaces for the operad $\oCom$ must be a
generalized Eilenberg-Mac~Lane space, and the category of
$\oCom$-algebras does not have an equivalent homotopy theory to the
category of $\oC_{\infty}$-algebras.  In generalizing
the notion of $E_{\infty}$ to other categories, getting the right
category of algebras is key.  For symmetric spectra, orthogonal
spectra, and EKMM $S$-modules and for chain complexes of modules
over a ring containing the rational numbers, it is harmless to allow
$\oCom$ to fit the definition of $E_{\infty}$ operad
(cf.~Examples~\ref{ex:dmeq2}, \ref{ex:dmeq3});
in spaces and chain complexes of modules over a finite field, some
freeness condition is required.  In general, the condition should be a
flatness condition on $\oO(m)$ for $(\oO(m)\mtimes
X^{(m)})/\Sigma_{m}$ as a functor of $X$ (for suitable $X$)
(cf.~Definition~\ref{def:dmeq}).

Unlike the definition of $E_{\infty}$ or $A_{\infty}$ operad, which
are defined in terms of homotopical conditions on the constituent
spaces, the definition of $E_{n}$ operads for other $n$ depends on specific
model operads first defined by Boardman-Vogt~\cite{BV-Bull} 
called the \textit{little $n$-cubes operads} $\oC_{n}$.

\begin{cons}[The little $n$-cubes operad]\label{cons:lilcubes}
The $m$th space $\oC_{n}(m)$ of the little $n$-cubes operad
is the space of $m$ ordered almost disjoint parallel axis affine
embeddings of the unit $n$-cube $[0,1]^{n}$ in itself.  So
$\oC_{n}(0)$ is a single point representing the unique way to embed
$0$ unit $n$-cubes in the unit $n$-cube.  A parallel axis affine
embedding of the unit cube in itself is a map of the form
\[
(t_{1},\dotsc,t_{n})\in [0,1]^{n}\mapsto (x_{1}+a_{1}t_{1},\dotsc,x_{n}+a_{n}t_{n})\in [0,1]^{n}
\]
for some fixed $(x_{1},\dotsc,x_{n})$ and $(a_{1},\dotsc,a_{n})$ with
each $a_{i}>0$, $x_{i}\geq 0$, and $x_{i}+a_{i}\leq 1$; it is
determined by the point $(x_{1},\dotsc,x_{n})$ 
where it sends $(0,\dotsc,0)$ and the point 
\[
(y_{1},\dotsc,y_{n})=(x_{1}+a_{1},\dotsc,x_{n}+a_{n})
\]
where it sends $(1,\dotsc,1)$.  So 
$\oC_{n}(1)$ is homeomorphic to the subspace
\[
\{ ((x_{1},\dotsc,x_{n}),(y_{1},\dotsc,y_{n}))\in [0,1]^{n}\times
[0,1]^{n}\mid x_{1}<y_{1}, x_{2}<y_{2}, \dotsc, x_{n}<y_{n}\}
\]
of $[0,1]^{n}\times [0,1]^{n}$.  For $m\geq 2$, almost disjoint means
that the images of the open subcubes are disjoint (the embedded cubes
only intersect on their boundaries), and $\oC_{n}(m)$ is homeomorphic
to a subset of $\oC_{n}(1)^{m}$.  The map $1$ is specified by the
element of $\oC_{n}(1)$ that gives the identity embedding of the unit
$n$-cube.  The action of the symmetric group is to re-order the
embeddings.  The composition law $\Gamma^{m}_{j_{1},\dotsc,j_{m}}$
composes the $j_{1}$ embeddings in $\oC_{n}(j_{1})$ with the first
embedding in $\oC_{n}(m)$, the $j_{2}$ embeddings in $\oC_{n}(j_{2})$
with the second embedding in $\oC_{n}(m)$, etc., to give
$j=j_{1}+\dotsb+j_{m}$ total embeddings.  See
Figure~\ref{fig:lilcubes} for a picture in the case $n=2$.  Taking
cartesian product with the identity map on $[0,1]$ takes a
self-embedding of the unit $n$-cube to a self-embedding of the unit
$(n+1)$-cube and induces maps of operads $\oC_{n}\to \oC_{n+1}$ that
are closed inclusions of the underlying spaces.  Let
$\oC_{\infty}(m)=\bigcup \oC_{n}(m)$; the operad structures on the
$\oC_{n}$ fit together to define an operad structure on $\oC_{\infty}$.
\end{cons}

\begin{figure}
\hrule
\[
\vcenter{\hbox{\includegraphics[width=.20\hsize]{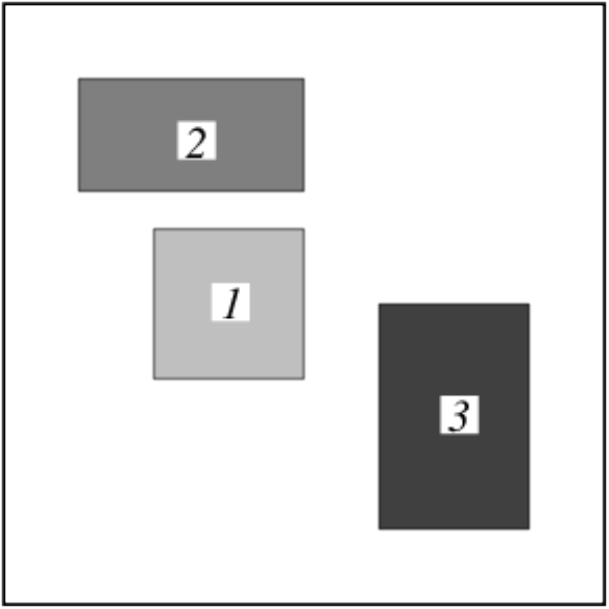}}
\hbox to .20\hsize{\hfil $a$ \hfil}}
\qquad\qquad 
\vcenter{\hbox{\includegraphics[width=.20\hsize]{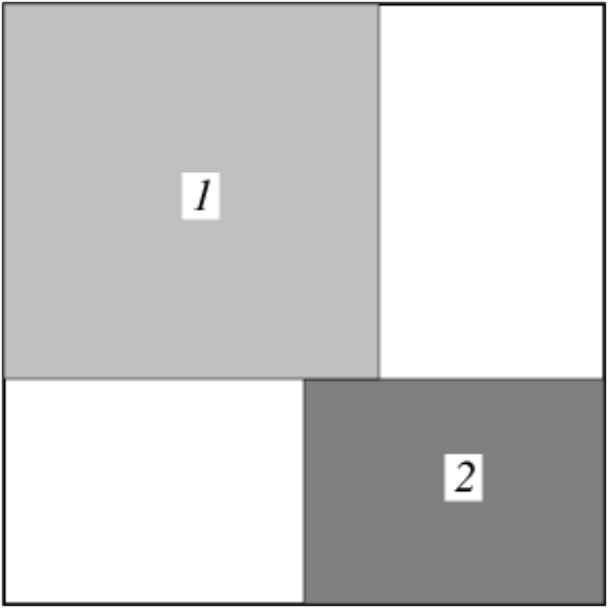}}
\hbox to .20\hsize{\hfil $b$ \hfil}}
\qquad\qquad 
\vcenter{\hbox{\includegraphics[width=.20\hsize]{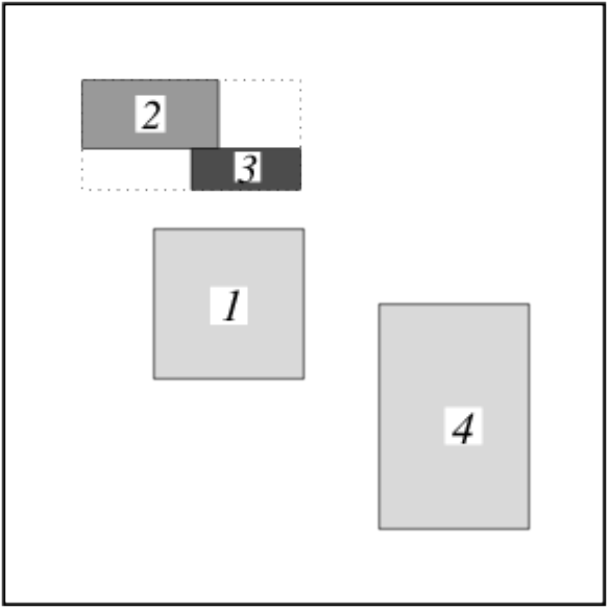}}
\hbox to .20\hsize{\hfil $\Gamma^{3}_{1,2,1}(a;1,b,1)$ \hfil}}
\]
\caption{Composition of Little 2-Cubes}
\label{fig:lilcubes}
\vskip 1ex
\begin{minipage}{.9\hsize}
The figure shows the composition 
\[
\Gamma^{3}_{1,2,1}\colon
\oC_{2}(3)\times \oC_{2}(1)\times \oC_{2}(2)\times \oC_{2}(1)\to \oC_{2}(4)
\]
applied to the elements $a\in \oC_{2}(3)$, $1\in
\oC_{2}(1)$, $b\in \oC_{2}(2)$, and $1\in \oC_{2}(1)$ with $a$ and $b$
as pictured.
\end{minipage}
\vskip 1em
\hrule
\end{figure}

The space $\oC_{n}(m)$ has the $\Sigma_{m}$-equivariant homotopy
type of the configuration space $C(m,\bR^{n})$ of $m$ (ordered) points
in $\bR^{n}$, or equivalently, $C(m,(0,1)^{n})$ of $m$ points in
$(0,1)^{n}$. To see this, since both spaces are free $\Sigma_{m}$-manifolds
(non-compact, and with boundary in the case of $\oC_{n}(m)$), it
is enough to show that they are non-equivariantly weakly equivalent, but
it is in fact no harder to produce a $\Sigma_{m}$-equivariant homotopy
equivalence explicitly. We have a $\Sigma_{m}$-equivariant map
$\oC_{n}(m)\to C(m,(0,1)^{n})$ by taking the center point of each
embedded subcube.  It is easy to define a $\Sigma_{m}$-equivariant
section of this map by continuously choosing cubes centered on the
given configuration; one way to do this is to make them all have the
same equal side length of $1/2$ of the minimum of the distance between
each of the points and the distance from each point to the boundary of
$[0,1]^{n}$.  A $\Sigma_{m}$-equivariant homotopy from the composite
map on $\oC_{n}(m)$ to the identity could (for example) first linearly
shrink all sides that are bigger than their original length and then
linearly expand all remaining sides to their original length.  In
particular, $\oC_{n}(1)$ is always contractible and $\oC_{n}(2)$ is
$\Sigma_{2}$-equivariantly homotopy equivalent to the sphere $S^{n-1}$
with the antipodal action.  For $m>2$, the configuration spaces can be
described in terms of iterated fibrations, and their Borel homology
was calculated by Cohen in~\cite{Cohen-Bulletin}
and~\cite[IV]{CohenLadaMay}.  

We can say more about the homotopy types in the cases $n=1$, $n=2$,
and $n=\infty$.  For $n=1$, the natural order of the interval $[0,1]$,
gives a natural order to the embedded sub-intervals ($1$-cubes); let
$\nsC_{1}(m)$ denote the subspace of $\oC_{1}(m)$ where the
sub-intervals are numbered in their natural order.  The spaces
$\nsC_{1}(m)$ are contractible and form a non-symmetric operad with
$\oC_{1}$ (canonically) isomorphic to the associated operad.  In other
words, the map of operads $\oC_{1}\to \oAss$ that takes a sequence of
embeddings and just remembers the order they come in is a
$\Sigma_{m}$-equivariant homotopy equivalence at each level.  In
particular $\oC_{1}$ is an $A_{\infty}$ operad.

For $n=2$, the configuration space $C(m,\bR^{2})$ is easily seen to be
an Eilenberg-Mac~Lane space $K(A_{m},1)$ where $A_{m}$ is the pure
braid group (of braids with fixed endpoints) on $m$ strands (see, for
example, \cite[\S 4]{May-GILS}).  

For $n=\infty$, $\oC_{\infty}$ is an $E_{\infty}$ operad; each $\oC_{\infty}(m)$ is a universal
$\Sigma_{m}$-space.  To see this, it is easier to work with
\[
C(m,\bR^{\infty}):=\bigcup C(m,\bR^{n}).
\]
Choosing a homeomorphism $(0,1)\iso \bR$ that sends $1/2$ to $0$, the
induced homeomorphisms
$\oC_{n}(m)\to C(m,\bR^{n})$ are compatible with the inclusions
$\oC_{n}(m)\to \oC_{n+1}(m)$ and $C(m,\bR^{n})\to C(m,\bR^{n+1})$; as
these inclusions are embeddings of closed submanifolds (with boundary
in the case of $\oC_{n}(m)$), the induced map 
\[
\oC_{\infty}(m)=\bigcup \oC_{n}(m)\to \bigcup C(m,\bR^{n})=C(m,\bR^{\infty})
\]
remains a homotopy equivalence.  One way to see that $C(m,\bR^{\infty})$ is
non-equivariantly contractible is to start by choosing a homotopy though injective
linear maps from the identity on $\bR^{\infty}$ to the shift map that
on basis elements sends $e_{i}$ to $e_{i+m}$.  We then homotope the
configuration (which now starts with the first $m$ coordinates all
zero) so that the $i$th point has $i$th coordinate $1$ and the
remainder of the first $m$ coordinates zero.  Finally, we homotope the
configuration to the configuration with $i$th point at $e_{i}$.  

We use the operads $\oC_{n}$ to define $E_{n}$ operads:

\begin{defn}\label{def:En}
An operad $\oE$ in spaces is an $E_{n}$ operad when there is a zigzag
of maps of operads relating it to $\oC_{n}$, each of which is a
$\Sigma_{m}$-equivariant homotopy equivalence on $m$th spaces for
all $m$.
\end{defn}

This definition is standard, but a bit awkward, because it defines a
property, whereas a better definition would define a structure and ask
for at least a preferred equivalence class of zigzag.

As we review in Section~\ref{sec:compare}, such maps induce
equivalences of homotopy categories of algebras (indeed, Quillen
equivalences).    We have implicitly given two different definitions of
$E_{\infty}$ operad; the following proposition justifies this.

\begin{prop}
An operad $\oE$ of spaces is $E_{\infty}$ in the sense of
Definition~\ref{defn:Einfty} if and only if it is $E_{\infty}$ in the
sense of Definition~\ref{def:En}.
\end{prop}

Before reviewing the proof, we state the following closely related proposition.

\begin{prop}\label{prop:E1}
An operad $\oE$ of spaces is $E_{1}$ if and only if it is isomorphic
to the associated operad of an $A_{\infty}$ operad.
\end{prop}

The previous two propositions (and their common proof) are the gist of
the second half of \S3 of May~\cite{May-GILS}.  In each case one direction
is clear, since $\oC_{1}$ and $\oC_{\infty}$ are $A_{\infty}$ and
$E_{\infty}$ (respectively), and the conditions of
Definitions~\ref{defn:Ainfty} and~\ref{defn:Einfty} are preserved by
the zigzags considered in Definition~\ref{def:En}.  The proof of the
other direction is to exhibit an explicit zigzag:

\begin{proof}
Let $\oE$ be the operad in question and assume it is either
$E_{\infty}$ in the sense of Definition~\ref{defn:Einfty} (for the
first proposition) or $A_{\infty}$ in the sense of
Definition~\ref{defn:Ainfty}ff (for the second proposition).  In the
case of the first proposition, consider the product in the category
of operads $\oC_{\infty}\times \oE$; it satisfies
\[
(\oC_{\infty}\times \oE)(m)=\oC_{\infty}(m)\times \oE(m)
\]
with the diagonal $\Sigma_{m}$-action and the unit and composition maps
the product of those for $\oC_{\infty}$ and $\oE$.  The projections
\[
\oC_{\infty}\from \oC_{\infty}\times \oE\to \oE
\]
give a zigzag as required by Definition~\ref{def:En}.  For the second
proposition, do the same trick with the non-symmetric operads $\nsE$
and $\nsC_{1}$ and then pass to the associated operads.
\end{proof}

Definitions~\ref{defn:Ainfty} and~\ref{defn:Einfty} mean that
identifying $A_{\infty}$ and $E_{\infty}$ operads is pretty
straightforward.  In unpublished work,
Fiedorowicz~\cite{Fiedorowicz-BraidedOperad} defines 
the notion of a \term{braided operad}, which
provides a good criterion for identifying $E_{2}$ operads. For $n>2$
(finite), the spaces $\oC_{n}(m)$ are 
not Eilenberg-Mac\ Lane spaces (for $m>1$), and that makes
identification of such operads much harder; however,
Berger~\cite[1.16]{Berger-CombConfig} proves a theorem (that he
attributes to Fiedorowicz) that gives a method to identify
$E_{n}$ operads that seems to work well in practice;
see~\cite[\S14]{McClureSmith-CosimplicialCubes},
\cite[\S1.6]{BergerFresse-Combinatorial}. 

The work of Dunn~\cite{Dunn-Tensor} and Fiedorowicz-Vogt~\cite{FV-Add}
is the start of an abstract identification of $E_{n}$ operads: The
derived tensor product of $n$ $E_{1}$ operads is an $E_{n}$ operad.
Here ``tensor product'' refers to the Boardman-Vogt tensor product of
operads (or PROPs) in \cite[2\S3]{BV-PROP}, which is the universal pairing
subject to ``interchange'', meaning that an $\oO\otimes \oP$-algebra
structure consists of an $\oO$-algebra and a $\oP$-algebra structure
on a space where the $\oO$- and $\oP$-structure maps commute (see
\textit{ibid.}\ for more details on the construction of the tensor
product).  This still essentially defines $E_{n}$ operads in terms of
reference models, though in principle, it gives a wide range of
additional models.  (The author does not know an example where this is
actually put to use, but \cite{BFV-THH} comes close.)  The concept of
interchange makes sense in any cartesian symmetric monoidal structure,
so this also in principle tells how to extend the notion of $E_{n}$ to
other cartesian symmetric monoidal categories with a reasonable
homotopy theory of operads for which the Boardman-Vogt tensor product
is reasonably well-behaved.  (Again, the author knows no examples
where this is put to use, but perhaps work by Barwick (unpublished),
Gepner (unpublished), and Lurie~\cite{Lurie-DAGVI} on $E_{n}$
structures is in a similar spirit.)

In categories suitably related to spaces, 
$E_{n}$ algebras are defined by a reference model suitably
related to $\oC_{n}$.  For example, in the context of simplicial sets, the total
singular complex of the little $n$-cubes operad has the canonical
structure of an operad of simplicial sets, and we define $E_{n}$
operads in terms of this reference model.  In 
symmetric spectra and orthogonal spectra, we have the reference model
given by the unbased suspension spectrum functor: an operad is an
$E_{n}$ operad when it is related to $\Sigma^{\infty}_{+}\oC_{n}$ by a
zigzag of operad maps that are (non-equivariant) weak equivalences on
$m$th objects for all $m$.  
For categories of chain complexes, we
use the singular chain complex of the little $n$-cubes operad to
define the reference model.  To make the singular chains an operad, we use
the Eilenberg-Mac Lane shuffle map to relate tensor product of chains
to chains on the cartesian product; the shuffle map is a lax symmetric
monoidal natural transformation
\[
C_{*}(X)\otimes C_{*}(Y)\to C_{*}(X\times Y),
\]
meaning that it commutes strictly with the
symmetry isomorphisms 
\[
C_{*}(X)\otimes C_{*}(Y)\iso C_{*}(Y)\otimes C_{*}(X) \qquad
C_{*}(X\times Y) \iso C_{*}(Y\times X)
\]
and makes the following associativity diagram commute.
\[
\xymatrix{%
C_{*}(X)\otimes C_{*}(Y)\otimes C_{*}(Z)\ar[r]\ar[d]
&C_{*}(X\times Y)\otimes C_{*}(Z)\ar[d]\\
C_{*}(X)\otimes C_{*}(Y\times Z)\ar[r]&C_{*}(X\times Y\times Z)
}
\]
See, for example, \cite[\S29]{May-Simplicial}.  

The fact that $E_{n}$ operads need to be defined in terms of a
reference model is not entirely satisfactory, especially in
homotopical contexts that are not topological.  Nevertheless, the
definition for spaces, simplicial sets, or chain complexes seems to
suffice to cover all other contexts that arise in
practice.\footnote{In theory, the definition for simplicial sets
should suffice for all homotopical contexts, but this may require
changing models, which for a particular problem may be inconvenient
or more complicated, or make it less concrete.}

\section{Operadic Algebras and Monads}
\label{sec:monad}

In the original context of iterated loop spaces and in many current
contexts in homotopy theory and beyond, the main purpose of
operads is to parametrize operations, which is to say, to define
operadic algebras.  For a closed symmetric monoidal category, there
are three equivalent definitions, one in terms of operations, one in
terms of endomorphism operads, and one in terms of monads.  This
section reviews the three definitions.

Viewing $\oO(m)$ as parametrizing some $m$-ary operations on an object
$X$ means that we have an \textit{action map}
\[
\oO(m)\mtimes X^{(m)}\to X.
\]
Since the right action of $\Sigma_{m}$ on $\oO(m)$ corresponds to
reordering the arguments of the operations, applying $\sigma \in
\Sigma_{m}$ to $\oO(m)$ (and then performing the action map) should
have the same effect as applying $\sigma$ to permute the factors in
$X^{(m)}$.  A concise way of saying this is to say that the map is
equivariant for the diagonal (left) action on the source $\oO(m)\mtimes X^{(m)}$
and the trivial action on the target $X$ (using the standard convention that the
left action $\sigma$ on $\oO(m)$ is given by the right action of
$\sigma^{-1}$).  The action map should also respect the composition law
$\Gamma$, making $\Gamma$ correspond to composition of operations, and
respect the identity $1$, making $1$ act by the identity operation.
The following gives the precise definition:

\begin{defn}\label{def:algebra}
Let $\mC$ be a symmetric monoidal category and
$\oO=(\{\oO(m)\},\Gamma,1)$ an operad in $\mC$.  An $\oO$-algebra (in
$\mC$) consists of an object $A$ in $\mC$ together with \textit{action
maps}
\[
\xi_{m}\colon \oO(m)\mtimes A^{(m)}\to A
\]
that are equivariant for the diagonal (left) $\Sigma_{m}$-action on
the source and the trivial $\Sigma_{m}$-action on the target and that
satisfy the following associativity and unit conditions:
\begin{enumerate}
\item For all $m$, $j_{1},\dotsc,j_{m}$, 
\[
\xi_{m}\circ (\id_{\oO(m)}\mtimes \xi_{j_{1}}\mtimes \dotsb \mtimes
\xi_{j_{m}})=\xi_{j}\circ (\Gamma^{m}_{j_{1},\dotsc,j_{m}}\mtimes
\id_{A}^{(j)}),
\]
i.e., the diagram
\[
\xymatrix@C+3pc{%
\oO(m)\mtimes \oO(j_{1})\mtimes \dotsb \mtimes \oO(j_{m})\mtimes A^{(j)}
\ar[r]^-{\Gamma^{m}_{j_{1},\dotsc,j_{m}}\mtimes \id_{A}^{(j)}}
\ar[d]_{\id_{\oO(m)}\mtimes \xi_{j_{1}}\mtimes \dotsb \mtimes \xi_{j_{m}}}
&\oO(j)\mtimes A^{(j)}\ar[d]^{\xi_{j}}\\
\oO(m)\mtimes A^{(m)}\ar[r]_{\xi_{m}}
&A
}
\]
commutes.

\item The map $\xi_{1}\circ (1\mtimes \id_{A})\colon \mS\mtimes A\to A$ is the unit
isomorphism for $\mtimes$.

\end{enumerate}
A map of $\oO$-algebras from $(A,\{\xi_{m}\})$ to $(A',\{\xi'_{m}\})$
consists of a map $f\colon A\to A'$ in $\mC$ that commutes with the
action maps, i.e., that make the diagrams
\[
\xymatrix{%
\oO(m)\mtimes A^{(m)}
  \ar[r]^-{\xi_{m}}
  \ar[d]_{\id_{\oO(m)}\mtimes f^{(m)}}
&A\ar[d]^{f}\\
\oO(m)\mtimes A^{\prime(m)}\ar[r]_-{\xi'_{m}}
&A'
}
\]
commute for all $m$.  We write $\mC[\oO]$ for the category of
$\oO$-algebras. 
\end{defn}

\begin{example}
When $\mC$ has an initial object and $\mtimes$ preserves the initial
object in each variable, the structure of an algebra over the identity
operad $\oI$ is no extra structure on an object of $\mC$.
\end{example}

As per~(ii) above and as illustrated in the previous example, the $1$
in the structure of the operad corresponds to the identity operation.
In some contexts algebras have units; when 
that happens, the unit is encoded in $\oO(0)$ as in the examples of
monoids and commutative monoids.  Recall that a \term{monoid object
for $\mtimes$ in $\mC$}
(or \term{$\mtimes$-monoid} for short)
consists of an object $M$ together with a \textit{multiplication map}
$\mu \colon M\mtimes M\to M$ and \textit{unit map} $\eta \colon \mS\to
M$ satisfying the following associativity and unit diagrams
\[
\xymatrix{%
M\mtimes M\mtimes M\mathstrut
\ar[r]^-{\mu \mtimes \id}\ar[d]_{\id\mtimes \mu}
&M\mtimes M\ar[d]^{\mu}\\
M\mtimes M\mathstrut\ar[r]_-{\mu}&M
}\qquad 
\xymatrix{%
\mS\mtimes M\mathstrut\ar[r]^-{\eta \mtimes \id}\ar[dr]_{\iso}
&M\mtimes M\ar[d]_{\mu}
&M\mtimes \mS\ar[l]_-{\id\mtimes \eta}\ar[dl]^{\iso}\\
&M
}
\]
(where the diagonal maps are the unit isomorphisms in $\mC$).  The
\textit{opposite multiplication} is the composite of the symmetry
morphism $\pmap\colon M\mtimes M\to M\mtimes M$ with $\mu$, and a
$\mtimes$-monoid is \textit{commutative} when $\mu=\mu \circ \pmap$.

\begin{example}\label{ex:comm}
Given a $\oCom$-algebra $A$, defining $\eta$ to be the action map $\xi_{0}$
\[
\eta \colon \mS=\oCom(0)\overto{\xi_{0}} A
\]
and $\mu$ to be the composite of the (inverse) unit isomorphism and
the action map $\xi_{2}$
\[
\mu \colon  A\mtimes A \iso \oCom(2)\mtimes A\mtimes A\overto{\xi_{2}}A
\]
endows $A$ with the structure of a commutative $\mtimes$-monoid: associativity
follows from the fact that the maps $\Gamma^{2}_{1,2}$ and
$\Gamma^{2}_{2,1}$ are both unit maps for $\mtimes$ so under the
canonical isomorphisms
\begin{align*}
A\mtimes A\mtimes A
&\iso
\oCom(2)\mtimes (\oCom(1)\mtimes \oCom(2))
\mtimes (A\mtimes A\mtimes A)\\
A\mtimes A\mtimes A
&\iso
\oCom(2)\mtimes (\oCom(2)\mtimes \oCom(1))
\mtimes (A\mtimes A\mtimes A)
\end{align*}
both maps induce the same map $A\mtimes A\mtimes A\to A$. Likewise,
the unit condition follows from the fact that 
\begin{align*}
\Gamma^{2}_{0,1}&\colon \oCom(2)\mtimes (\oCom(0)\mtimes \oCom(1))\to \oCom(1)\\
\Gamma^{2}_{1,0}&\colon \oCom(2)\mtimes (\oCom(1)\mtimes \oCom(0))\to \oCom(1)
\end{align*}
are both unit maps.  The multiplication is commutative because the
action of the symmetry map on $\mS=\oCom(2)$ is trivial.
Conversely, we can convert a commutative $\mtimes$-monoid to a $\oCom$-algebra
by taking $\xi_{0}$ to be the unit $\eta$, $\xi_{1}$ to be the unit
isomorphism for $\mtimes$, $\xi_{2}$ to be induced by the unit
isomorphism for $\mtimes$ and the multiplication, and all higher
$\xi_{m}$'s induced by the unit isomorphism for $\mtimes$ and (any)
iterated multiplication.  This defines a bijective correspondence
between the set of commutative $\mtimes$-monoid structures and the set of
$\oCom$-algebra structures on a fixed object and an isomorphism
between the category of commutative $\mtimes$-monoids and the category of
$\oCom$-algebras. 
\end{example}

For a non-symmetric operad, defining an algebra in terms of the
associated operad or in terms of the analogue of
Definition~\ref{def:algebra} without the equivariance requirement
produce the same structure.  

\begin{example}\label{ex:mon}
The constructions of Example~\ref{ex:comm} applied to the
non-symmetric operad $\nsAss$ give a bijective correspondence between 
the set of $\mtimes$-monoid structures and the set of $\oAss$-algebra structures
on a fixed object and an isomorphism between the category of
$\mtimes$-monoids and the category of $\oAss$-algebras.
\end{example}

The monoid and commutative monoid objects in the category of sets
(with the usual symmetric monoidal structure given by cartesian
product) are just the monoids and commutative monoids in the usual
sense.  Likewise, in spaces, they are the topological monoids and
topological commutative monoids.  In the category of abelian groups
(with the usual symmetric monoidal structure given by the tensor
product), the monoid objects are the rings and the commutative monoid
objects are the commutative rings. In the category of chain complexes
of $R$-modules for a commutative ring $R$ (with the usual symmetric
monoidal structure given by tensor product over $R$), the monoid objects are the
differential graded $R$-algebras and the commutative monoid objects
are the commutative differential graded $R$-algebras.  In a modern
category of spectra, the monoid objects are called
\term{$\bS$-algebras} or sometimes \term{strictly associative ring
spectra}.  Some authors take the term ``ring spectrum'' to be
synonymous with $\bS$-algebra, but others take it to mean the
weaker notion of monoid object in the stable category (or even weaker
notions).  Work of Schwede-Shipley~\cite[3.12.(3)]{SS-Monoidal} shows that the
homotopy category of monoid objects in any modern category of spectra
is equivalent to an appropriate full subcategory of the (mutually
equivalent) homotopy category of monoid objects in EKMM $S$-modules,
symmetric spectra, or orthogonal spectra (at least when ``modern
category of spectra'' is used as a technical term to mean a model
category with a preferred equivalence class of symmetric monoidal
Quillen equivalence to the currently known modern categories of
spectra); cf.~Example Theorem~\ref{ethm:compare} below. The analogous result does not hold
for commutative monoid objects; see~\cite{Lawson-CommGamma}.  The term
``commutative $\bS$-algebra'' is typically reserved for examples where
the homotopy category of commutative monoid objects is equivalent to
an appropriate full subcategory of the (mutually equivalent) homotopy
category of commutative monoid objects in EKMM $S$-modules,
symmetric spectra, or orthogonal spectra.

Returning to the discussion of operadic algebras,
in the case when $\mC$ is a closed symmetric monoidal category,
adjoint to the action map 
\[
\xi_{m}\colon \oO(m)\mtimes A^{(m)}\to A
\]
is a map
\[
\phi_{m}\colon \oO(m)\to F(A^{(m)},A)=\oEnd_{A}(m).
\]
Equivariance for $\xi_{m}$ is equivalent to equivariance for
$\phi_{m}$.  Similarly, conditions~(i) and~(ii) in the definition of
$\oO$-algebra (Definition~\ref{def:algebra}) are adjoint to the
diagrams in the definition of map of operads
(Definition~\ref{def:opmap}).  This proves the following proposition,
which gives an alternative definition of $\oO$-algebra.

\begin{prop}\label{prop:algend}
Let $\mC$ be a closed symmetric monoidal category, let $\oO$ be an
operad in $\mC$, and let $X$ be an object in $\mC$.  The adjunction
rule $\xi_{m} \leftrightarrow \phi_{m}$ above defines a bijection
between the set of $\oO$-algebra structures on $X$ and the set of maps
of operads $\oO\to \oEnd_{X}$.
\end{prop}

In the case when $\mC$ is (countably) cocomplete (has (countable)
colimits) and $\mtimes$ preserves (countable) colimits in each variable (which
includes the case when it is closed), algebras can also be defined in
terms of a monad.  The idea for the underlying functor is to gather
the domains of all the action maps into a coproduct; since the action
maps are equivariant with target having the trivial action, they
factor through the coinvariants (quotient by the symmetric group
action), and this goes into the definition.

\begin{notn}\label{notn:monad}
Let $\mC$ be a symmetric monoidal category with countable colimits,
and let $\oO$ be an operad in $\mC$.  Define the endofunctor $\bO$ of
$\mC$ (i.e., functor $\bO\colon \mC\to \mC$) by
\[
\bO X = \myop\coprod_{m=0}^{\infty} \oO(m)\mtimes_{\Sigma_{m}} X^{(m)}
\]
(where $\oO(m)\mtimes_{\Sigma_{m}}X^{(m)}:=(\oO(m)\mtimes X^{(m)})/\Sigma_{m}$).
\end{notn}

(When we use other letters for operads, we typically use the
corresponding letters for the associated monad; for example, we write
$\bA$ for the monad associated to an operad $\oA$, $\bB$ for the monad
associated to an operad $\oB$, etc.)

The action maps for an $\oO$-algebra $A$ then specify a map $\xi\colon
\bO A\to A$; the conditions for defining an $\oO$-structure also admit
a formulation in terms of this map.  
The basic observation is that we have a canonical isomorphism 
\begin{multline*}
(\bO X)^{(m)}\iso
\myop\coprod_{j_{1}=0}^{\infty}\dotsb \myop\coprod_{j_{m}=0}^{\infty}
(\oO(j_{1})\mtimes_{\Sigma_{j_{1}}}X^{(j_{1})})\mtimes \dotsb \mtimes 
(\oO(j_{m})\mtimes_{\Sigma_{j_{1}}}X^{(j_{m})})
\\\iso
\def\mystrut{\vrule height0pt width0pt depth1ex}
\myop\coprod_{j=0}^{\infty}\,\myop\coprod_{\putatop{j_{1},\dotsc,j_{m}\mystrut}{\sum
j_{i}=j}}
(\oO(j_{1})\mtimes \dotsb \oO(j_{m}))\mtimes_{\Sigma_{j_{1}}\times \dotsb \times \Sigma_{j_{m}}}X^{(j)}
\end{multline*}
using the symmetry isomorphism to shuffle
like factors without permuting them.  We can use this isomorphism to give $\bO X$
the canonical structure of an $\oO$-algebra, defining the action map
\[
\mu_{m}\colon \oO(m)\mtimes (\bO X)^{(m)}\to \bO X
\]
by commuting the coproduct
past $\mtimes$, using the operad composition law, and passing to the
quotient by the full permutation group
\begin{multline*}
\oO(m)\mtimes (\bO X)^{(m)}\iso 
\def\mystrut{\vrule height0pt width0pt depth1ex}
\myop\coprod_{j=0}^{\infty}\,\myop\coprod_{\putatop{j_{1},\dotsc,j_{m}\mystrut}{\sum
j_{i}=j}}
  \oO(m)\mtimes  (\oO(j_{1})\mtimes \dotsb \oO(j_{m}))
     \mtimes_{\Sigma_{j_{1}}\times \dotsb \times
     \Sigma_{j_{m}}}X^{(j)}
\\
\overto{\coprod\coprod \Gamma^{m}_{j_{1},\dotsc,j_{m}}\mtimes \id_{X}^{(j)}}
\myop\coprod_{j=0}^{\infty} \oO(j) \mtimes_{\Sigma_{j_{1}}\times \dotsb \times
     \Sigma_{j_{m}}}X^{(j)}
\to \myop\coprod_{j=0}^{\infty} \oO(j) \mtimes_{\Sigma_{j}}
X^{(j)}=\bO X.
\end{multline*}
The pictured map is well-defined because of the $(\Sigma_{j_{1}}\times
\dotsb \times \Sigma_{j_{m}})$-equivariance of
$\Gamma^{m}_{j_{1},\dotsc,j_{m}}$
(\ref{def:operad}.\eqref{part:operad:easyperm}).  
The other permutation rule
(\ref{def:operad}.\eqref{part:operad:hardperm}) implies that $\mu_{m}$
is $\Sigma_{m}$-equivariant.  The remaining two parts of the definition of
operad show that the
$\mu_{m}$ define an $\oO$-algebra structure map:
\ref{def:operad}.(\ref{part:operad:assoc})--(ii) imply
\ref{def:algebra}.(i)--(ii), respectively.  This $\oO$-algebra
structure then defines a map
\[
\mu\colon \bO\bO X\to \bO X
\]
as above, which is natural in $X$.  The map $1\mtimes \id_{X}$ also
induces a natural transformation 
\[
\eta \colon X\to \bO X.
\]
These two maps together give $\bO$ the structure of a monad.

\begin{prop}\label{prop:monad}
Let $\mC$ be a symmetric monoidal category with countable colimits and
assume that $\mtimes$ commutes with countable colimits in each variable.  For an
operad $\oO$, the functor $\bO$ and natural transformations $\mu$,
$\eta$ form a monad: the diagrams
\[
\xymatrix{%
\bO\bO\bO X\mathstrut\ar[r]^{\mu}\ar[d]_{\bO\mu}&\bO\bO X\ar[d]^{\mu}\\
\bO\bO X\mathstrut\ar[r]_{\mu}&\bO X
}\qquad 
\xymatrix{%
\bO X\mathstrut\ar[r]^{\eta}\ar@{=}[dr]&\bO\bO X\ar[d]^{\mu}\\
&\bO X\mathstrut
}
\]
commute (where the top map in the lefthand diagram is the map $\mu$
for the object $\bO X$).
\end{prop}

The proof is applying \ref{def:algebra}.(i)--(ii) for $\bO X$.

\begin{example}
Under the hypotheses of the previous proposition, the monad associated
to the identity operad $\oI$ is canonically isomorphic (via the unit
isomorphism) to the identity monad $\Id$.  The monad
associated to the operad $\oCom$ is canonically isomorphic to the free
commutative monoid monad
\[
\bP X = \myop\coprod_{j=0}^{\infty} X^{(j)}/\Sigma_{j}.
\]
The monad
associated to the algebra $\oAss$ is canonically isomorphic to the free
monoid monad
\[
\bT X = \myop\coprod_{j=0}^{\infty} X^{(j)}.
\]
\end{example}

An algebra over the monad $\bO$ consists of an object $A$ and a map
$\xi \colon \bO A\to A$ such that the diagrams
\[
\xymatrix{%
\bO\bO A\mathstrut\ar[r]^{\mu}\ar[d]_{\bO \xi}&\bO A\ar[d]^{\xi}\\
\bO A\mathstrut\ar[r]_{\xi}&A
}\qquad 
\xymatrix{%
A\mathstrut\ar[r]^{\eta}\ar@{=}[dr]&\bO A\ar[d]^{\xi}\\
\relax\mathstrut&A
}
\]
commute.  Given an $\oO$-algebra $(A,\{\xi_{m}\})$, the map $\xi\colon
\bO A\to A$ constructed as the coproduct of the induced maps on
coinvariants then is an $\bO$-algebra action map.  Conversely, given
an $\bO$-algebra $(A,\xi)$, defining $\xi_{m}$ to be the composite
\[
\oO(m)\mtimes A^{(m)}\to \bO A\overto{\xi}A,
\]
the maps $\xi_{m}$ make $A$ an $\oO$-algebra.  This gives a second
alternative definition of $\oO$-algebra.

\begin{prop}\label{prop:algmon}
Under the hypotheses of Proposition~\ref{prop:monad}, for $X$ an
object of $\mC$, the
correspondence $\{\xi_{m}\}\leftrightarrow \xi$ above defines a
bijection between the set of $\oO$-algebra structures on $X$ and the
set of $\bO$-algebra structures on $X$ and an isomorphism between the
category of $\oO$-algebras and the category of $\bO$-algebras.
\end{prop}

\section{Modules over Operadic Algebras}\label{sec:modules}

Just as an operad defines a category of algebras, an algebra defines a
category of modules.  Because this chapter concentrates on the theory of
operadic algebras, we will only touch on the theory of modules. A
complete discussion could fill a book and many of the aspects of the
theory of operads we omit in this chapter (including Koszul duality,
Quillen (co)homology, Deligne and Kontsevich conjectures) correspond
to statements about categories of modules; even an overview could form
its own chapter.  We restrict to giving a brief review of the
definitions and the homotopy theory.

The original definition of modules over an operadic algebra seems to
be due to Ginzburg and Kapranov~\cite[\S1.6]{GinzburgKapranov}.

\begin{defn}\label{def:module}
For an operad $\oO$ and an $\oO$-algebra $A$, an $(\oO,A)$-module (or
just $A$-module when $\oO$ is understood) consists of an object
$M$ of $\mC$ and structure maps
\[
\zeta_{m} \colon \oO(m+1)\mtimes (A^{(m)}\mtimes M)\to M
\]
for $m\geq 0$ such that the associativity, symmetry, and unit diagrams
in Figure~\ref{fig:module} commute. A map of $A$-modules is a map
of the underlying objects of $\mC$ that commutes with the structure maps.
\end{defn}

\begin{figure}
\hrule
\begin{gather*}
\xymatrix@C-3pc@R-1pc{%
&\oO(j+1)\mtimes (A^{(j)}\mtimes M)
\ar[dd]^{\zeta_{j}}\\
\oO(m+1)\mtimes (\oO(j_{1})\mtimes \dotsb \mtimes \oO(j_{m})\mtimes \mS)\mtimes
(A^{(j)}\mtimes M)
\ar `u[u] [ur]^(.25){\left(\Gamma^{m+1}_{j_{1},\dotsc,j_{m},1}\circ
   (\id\mtimes \dotsb\mtimes \id \mtimes 1)\right)\mtimes \id}
\ar[dd]_{c}\\
&M\\
\txt{{$\oO(m+1)\mtimes \left(\begin{aligned}
&(\oO(j_{1})\mtimes A^{(j_{1})})\mtimes \dotsb\cr
&\quad\dotsb\mtimes  (\oO(j_{m})\mtimes A^{(j_{m})})\mtimes M
\end{aligned}\right)$}}
\ar `d[dr] [dr]_-{\id\mtimes(\xi_{j_{1}}\mtimes \dotsb \mtimes \xi_{j_{m}}\mtimes \id)}\\
&\oO(m+1)\mtimes (A^{(m)}\mtimes M)\ar[uu]_{\zeta_{m}}
}\\[1em]
\xymatrix@R-1pc{%
\oO(m+1)\mtimes A^{(m)}\mtimes M\ar[dd]_{a_{\sigma^{-1}}\mtimes c_{\sigma}\mtimes \id}
\ar[dr]^-{\zeta_{m}}
&&\mS\mtimes M\ar[ddr]_{\iso}\ar[r]^-{1\mtimes \id}
&\oO(1)\mtimes M\ar[dd]^{\zeta_{0}}\\
&M\\
\oO(m+1)\mtimes A^{(m)}\mtimes M\ar[ur]_-{\zeta_{m}}
&&&M
}
\end{gather*}
\caption{The diagrams for Definition~\ref{def:module}} 
\label{fig:module}
\vskip 1ex
\begin{minipage}{.9\hsize}
In the first diagram, $j=j_{1}+\dotsb+j_{m}$ and $c$ is the
$\mtimes$-permutation that shuffles the $\oO(j_{i})$'s past the $M$
and $A$'s 
as displayed composed with the unit isomorphism for
$\mtimes$; $\xi_{i}$ denote the $\oO$-algebra structure maps for $A$.
In the second diagram, $\sigma$ is a permutation of $\{1,\dotsc,m\}$,
permuting the factors of $A$,
viewed as an element of $\Sigma_{m+1}$ for 
permutation action on $\oO(m+1)$.
In the third diagram, the diagonal isomorphism is the 
unit isomorphism for $\mtimes$.
\end{minipage}
\vskip 1em
\hrule
\end{figure}

Although the definition appears to favor $A$ on the left, we obtain
analogous righthand structure maps
\[
\oO(m+1)\mtimes (M\mtimes A^{(m)})\to M
\]
satisfying the analogous righthand version of the diagrams in
Figure~\ref{fig:module} by applying an appropriate permutation.  Thus,
an $A$-module structure can equally be regarded as either a left or
right module structure.  The following example illustrates
this point.

\begin{example}
When $\oO=\oAss$, the
(symmetric) operad for associative algebras and $A$ is an
$\oO$-algebra (i.e., $\mtimes$-monoid), an $(\oO,A)$-module in the
sense of the previous definition is precisely an $A$-bimodule in the
usual sense: it has structure maps
\[
\lambda \colon A\mtimes M\to M\qquad \text{and}\qquad \rho\colon
M\mtimes A\to M
\]
satisfying the following associativity, unity, and interchange diagrams
\begin{gather*}
\xymatrix{%
A\mtimes A\mtimes M\ar[r]^-{\mu\mtimes\id}\ar[d]_-{\id\mtimes\lambda}
&A\mtimes M\ar[d]^-{\lambda}
&M\mtimes A\mtimes A\ar[r]^-{\id\mtimes\mu}\ar[d]_-{\rho\mtimes\id}
&M\mtimes A\ar[d]^-{\rho}\\
A\mtimes M\ar[r]_-{\lambda}&M&M\mtimes A\ar[r]_-{\rho}&M
}\\
\xymatrix{%
A\mtimes M\ar[dr]_-{\lambda}
&\mS\mtimes M\iso M\mtimes \mS
   \ar[d]\ar[l]_-{\eta\mtimes\id}\ar[r]^-{\id\mtimes\eta}
&M\mtimes A\ar[dl]^-{\rho}\\
&M
}\qquad 
\xymatrix{%
A\mtimes M\mtimes A\ar[r]^-{\lambda \mtimes \id}\ar[d]_-{\id\mtimes\rho}
&M\mtimes A\ar[d]^-{\rho}\\
A\mtimes M\ar[r]_-{\lambda}&M
}
\end{gather*}
where $\mu$ denotes the multiplication and $\eta$ the unit for $A$ and
the unlabeled arrow is the unit isomorphism for $\mtimes$.
\end{example}

Obtaining a theory of modules closer to the idea of a left module (or
right module) over an associative algebra requires working with
non-symmetric operads.  

\begin{defn}\label{def:leftmodule}
Let $\nsO$ be a non-symmetric operad and let
$A$ be an $\nsO$-algebra.  A left $(\nsO,A)$-module (or
just left $A$-module when $\nsO$ is understood) consists of an object
$M$ of $\mC$ and structure maps
\[
\zeta_{m} \colon \nsO(m+1)\mtimes (A^{(m)}\mtimes M)\to M
\]
for $m\geq 0$ such that the associativity and unit diagrams
in Figure~\ref{fig:module} commute (with $\nsO$ in place of $\oO$). A
map of left $A$-modules is a map of the underlying objects of $\mC$ that
commutes with the structure maps. 
\end{defn}

We also have the evident notion of a right $A$-module defined in terms
of structure maps
\[
\zeta_{m} \colon \nsO(m+1)\mtimes (M\mtimes A^{(m)})\to M
\]
and the analogous righthand associativity and unit diagrams.

Unlike in the case of operadic algebras, where 
working with a non-symmetric operad and its corresponding symmetric
operad results in the same theory, in the case of modules, the results
are very different.  

\begin{example}
When $\nsO=\nsAss$, the
non-symmetric operad for associative algebras and $A$ is an
$\nsO$-algebra (i.e., a $\mtimes$-monoid), a left $(\nsAss,A)$-module in the
sense of the previous definition is precisely a left $A$-module in the
usual sense defined in terms of an associative and unital left action
map $A\mtimes M\to M$.  Likewise, a right $(\nsAss,A)$-module is
precisely a right $A$-module in the usual sense.
\end{example}

Under mild hypotheses, the category of $(\oO,A)$-modules is a category of
modules for a $\mtimes$-monoid called the
enveloping algebra of $A$.

\begin{cons}[The enveloping algebra]\label{cons:UA}
Assume that $\mC$ admits countable colimits and $\mtimes$ preserves
countable colimits in each variable.  For an operad $\oO$ and an
$\oO$-algebra $A$, let $U^{\oO}A$ (or $UA$ when $\oO$
is understood) be the following coequalizer
\[
\xymatrix@C-1pc{%
\myop\coprod_{m=0}^{\infty}\oO(m+1)\mtimes_{\Sigma_{m}} (\bO A)^{(m)}\ar@<.5ex>[r]\ar@<-.5ex>[r]
&\myop\coprod_{m=0}^{\infty}\oO(m+1) \mtimes_{\Sigma_{m}} A^{(m)}\ar[r]
&U^{\oO}A
}
\]
where we regard $\Sigma_{m}$ as the usual subgroup of $\Sigma_{m+1}$
of permutations that fix $m+1$.  Here one map is induced by the action map $\bO A\to A$ and the other
is induced by the operadic multiplication
\begin{multline*}
\oO(m+1)\mtimes (\nbO A)^{(m)}
\iso \myop\coprod_{j_{1},\dotsc,j_{m}}\oO(m+1)\mtimes
(\oO(j_{1})\mtimes A^{(j_{1})})\mtimes\dotsb \mtimes
(\oO(j_{m})\mtimes A^{(j_{m})})\\
\iso
\myop\coprod_{j_{1},\dotsc,j_{m}}\bigl(\oO(m+1)\mtimes (\oO(j_{1})\mtimes\dotsb \mtimes
\oO(j_{m})\mtimes \mS)\bigr)\mtimes A^{(j)}\\
\overto{\coprod \Gamma^{m+1}_{j_{1},\dotsc,j_{m},1}\mtimes \id}
\oO(j+1)\mtimes A^{(j)}
\end{multline*}
(where we have omitted writing $1\colon \mS\to \oO(1)$ and as
always $j=j_{1}+\dotsb+j_{m}$).  Let $\eta \colon \mS\to UA$ be the
map induced by $1\colon \mS\to \oO(1)$ and the inclusion of the $m=0$
summand and let $\mu\colon UA\mtimes UA\to UA$ be the map induced from the maps
\[
(\oO(m+1)\mtimes A^{(m)})\mtimes (\oO(n+1)\mtimes A^{(n)})\to
\oO(m+n+1)\mtimes A^{(m+n)}
\]
obtained from the map $\circ_{m+1}\colon \oO(m+1)\mtimes \oO(n+1)\to
\oO(m+n+1)$ defined as the composite
\[
\oO(m+1)\mtimes \oO(n+1)\iso
\oO(m+1)\mtimes(\mS\mtimes \dotsb \mtimes\mS \mtimes \oO(n+1))
\overto{\!\Gamma^{m+1}_{1,\dotsc,1,n+1}\!}
\oO(m+n+1)
\]
(where again we have omitted writing $1\colon \mS\to \oO(1)$).
Associativity of the operad multiplication implies that $\eta$ and
$\mu$ give $UA$ the structure of an associative monoid for $\mtimes$ and
the resulting object is called the \term{enveloping algebra} of $A$
over $\oO$.
\end{cons}

An easy argument from the definitions and universal property of the
coequalizer proves the following proposition. 

\begin{prop}\label{prop:modulecat}
Assume $\mC$ admits countable coproducts and $\mtimes$ preserves them
in each variable. Let $\oO$ be an operad and let $A$ be an
$\oO$-algebra.  For an object $X$ of $\mC$, 
$(\oO,A)$-module structures on $X$ are in bijective correspondence with left
$U^{\oO}A$-module structures.  In particular, the category of 
$(\oO,A)$-modules is isomorphic to the category of left $U^{\oO}A$-modules.
\end{prop}

Similarly, in the case of non-symmetric operads, we can construct a
left module enveloping algebra $\nU^{\nsO}A$ (denoted $\nU A$ when
$\nsO$ is understood) as the following coequalizer
\begin{equation}\label{eq:nUA}
\xymatrix@C-1pc{%
\myop\coprod_{m=0}^{\infty}\nsO(m+1)\mtimes (\nbO A)^{(m)}\ar@<.5ex>[r]\ar@<-.5ex>[r]
&\myop\coprod_{m=0}^{\infty}\nsO(m+1) \mtimes A^{(m)}\ar[r]
&\nU^{\nsO}A
}
\end{equation}
with maps defined as in Construction~\ref{cons:UA}.  The analogous
identification of module categories holds.

\begin{prop}\label{prop:modulecat2}
Assume $\mC$ admits countable coproducts and $\mtimes$ preserves them
in each variable.  Let $\nsO$ be a non-symmetric operad and let $A$ be
an
$\nsO$-algebra.  For an object $X$ of $\mC$, left
$(\nsO,A)$-module structures on $X$ are in bijective correspondence with left
$\nU A$-module structures.  In particular, the category of left
$(\nsO,A)$-modules is isomorphic to the category of left $\nU A$-modules.
\end{prop}

We develop some tools to study enveloping algebras in the next
section.  In the meantime, we can identify the enveloping algebra in
some specific examples.

\begin{example}
For $\oO=\oAss$ and $A$ an $\oAss$-algebra (a $\mtimes$-monoid), $U^{\oAss}A$ is $A\mtimes
A^{\op}$, the usual enveloping algebra for a $\mtimes$-monoid.
Viewing $A$ as an $\nsAss$-algebra, $\nU^{\nsAss}A$ is the 
$\mtimes$-monoid $A$.  If $A$ is a $\oCom$-algebra (a
commutative $\mtimes$-monoid), then
$U^{\oCom}A$ makes sense and is also the $\mtimes$-monoid $A$.
\end{example}

\begin{example}
Let $\oL$ denote the Boardman-Vogt linear isometries operad of
Example~\ref{ex:L}.  For an $\oL$-algebra, the underlying space of
$U^{\oL}A$ is the pushout 
\[
\xymatrix{%
\oL(2)\times_{\oL(1)} \oL(0)\ar[d]_{\circ_{1}}
   \ar[r]^-{\id\times\xi_{0}}&\oL(2)\times_{\oL(1)} A\ar[d]\\
\oL(1)\ar[r]&U^{\oL}A
}
\]
where $\circ_{1}$ is the map induced by $1\colon *\to \oL(1)$ and
$\Gamma^{2}_{0,1}$ (as in Construction~\ref{cons:UA}) and the right
action on $\oL(2)$ of $\oL(1)\iso 
\oL(1)\times *$  is via $\Gamma^{2}_{1,1}\circ(\id\times 1)$. 
The inclusions of the $m=0$ and $m=1$ summands induce the map from the pushout
above to the coequalizer
defining $U^{\oL}A$; the inverse isomorphism uses the 
``Hopkins' Lemma'' \cite[I.5.4]{EKMM} isomorphism
\begin{equation}\tag{HL}
\oL(2)\times_{\oL(1)\times \oL(1)}(\oL(i)\times \oL(j))\iso \oL(i+j)
\end{equation}
for $i,j\geq 1$.   The pushout explicitly admits maps in from the
$m=0$ and $m=1$ summands of the coequalizer and for $m>1$, we have the
following map.
\begin{multline*}
\oL(m+1)\times_{\Sigma_{m}} A^{(m)}
\iso \oL(m+1)\times_{\Sigma_{m}\times \oL(1)} (A^{(m)}\times \oL(1))\\
\mathrel{\mathop{\iso}\limits_{\textrm{(HL)}}}
\oL(2)\times_{\oL(1)\times \oL(1)}((\oL(m)\times_{\Sigma_{m}}A^{(m)})\times \oL(1))\\
\overto{\id\times (\xi_{m}\times \id)}
\oL(2)\times_{\oL(1)\times \oL(1)}(A\times \oL(1))
\iso \oL(2)\times_{\oL(1)}A
\end{multline*}
The previous display also indicates how the multiplication of $U^{\oL}A$ works in the
pushout description: it is induced by the map
\begin{multline*}
(\oL(2)\times_{\oL(1)} A)\times (\oL(2)\times_{\oL(1)} A)
\iso (\oL(2)\times \oL(2))\times A^{(2)}
\\\overto{\circ_{2}\times \id}
\oL(3)\times A^{(2)}\to \oL(3)\times_{\Sigma_{2}} A^{(2)}\to
 \oL(2)\times_{\oL(1)}A
\end{multline*}
where the last map is the $m=2$ case of the map above.  It turns out
that the map $U^{\oL}A\to A$ induced by the operadic algebra action
maps is always a weak equivalence. (The proof is not obvious but uses
the ideas from EKMM, especially~\cite[I.8.5,XI.3.1]{EKMM} in the
context of the theory of $\oL(1)$-spaces as in for
example~\cite[\S6]{BasterraMandell-Homology},
\cite[\S4.6]{Blumberg-Progress}, or
\cite[\S4.3]{BlumbergCohenS-THHThom}.)  If we forget the symmetries in
$\oL$ to create a non-symmetric operad $\oL\NS$, then
$\nU^{\oL\NS}A\iso U^{\oL}A$.  Even when $A$ is just an
$\oL\NS$-algebra, $\nU^{\oL\NS}A$ can still be identified as
the same pushout construction pictured above using the analogous
comparison isomorphisms with the coequalizer
definition~\eqref{eq:nUA}.  Analogous formulations also hold in the
context of orthogonal spectra, symmetric spectra, and EKMM $S$-modules
using the operad $\Sigma^{\infty}_{+}\oL$ (in the respective
categories).  In the context of Lewis-May spectra, these observations
are closely related to the foundations of EKMM $S$-modules and the
properties of the smash product ($\sma_{\oL}$, $\sma$, and
$\sma_{A}$); this is the start of a much longer story on monoidal
products and balanced products for $A_{\infty}$ module
categories (see, for example, \cite{Smash}
and~\cite[\S17-18]{BlumbergMandell-TTHH}).
\end{example}

Although in both of the previous two examples, we had an isomorphism
of enveloping algebras for symmetric and non-symmetric constructions,
this is not a general phenomenon, as can be seen, for example, by
comparing $U^{\oAss}$ and $\nU^{\oAss\NS}$ where $\oAss\NS$ is the
non-symmetric operad formed from $\oAss$ by forgetting the symmetry.
(In this style of notation, $\nsAss=\oCom\NS$.)  

The left module enveloping algebra construction for the non-symmetric
little $1$-cubes operad, $\nU^{\nsC_{1}}(-)$, also admits a concrete
description~\cite[\S2]{Smash}, which we review in
Section~\ref{sec:rectass}.  It shares the feature with the
previous two examples that for any $\nsC_{1}$-algebra $A$,
$\nU^{\nsC_{1}}A$ is weakly equivalent to $A$ (see
[\textit{ibid.},1.1] or Proposition~\ref{prop:UAhty}).

Given Propositions~\ref{prop:modulecat} and~\ref{prop:modulecat2}, the
homotopy theory of modules over operadic algebras reduces to (1) the
homotopy theory of modules over $\mtimes$-monoids and (2) the
homotopy theory of $U^{\oO}A$ (or $\nU^{\nsO}A$) as a functor of
$\oO$ (or $\nsO$) and $A$.  The latter first requires a study of the
homotopy theory of operadic algebras that we review (in part) in the
next few sections before returning to this question in
Corollary~\ref{cor:modcomp}.  On the other hand the homotopy theory of modules
over $\mtimes$-monoids is very straightforward, and we give
a short review of the
main results in the remainder of this section.  
We discuss this in terms of closed model
categories.  (For an overview of closed model categories as a setting
for homotopy theory, we refer the reader to~\cite{DwyerSpalinski}.)  
The following theorem gives a comprehensive result in some categories
of primary interest.

\begin{thm}\label{thm:modone}
Let $(\mC,\mtimes,\mS)$ be the category of simplicial sets, spaces, symmetric
spectra, orthogonal spectra, EKMM $S$-modules, simplicial abelian
groups, chain complexes, or any category of modules over a commutative
monoid object in one of these categories, with the usual
monoidal product and one of the standard cofibrantly generated model
structures.  Let $A$ be a monoid object in $\mC$.  The category 
of $A$-modules is a closed model category with weak equivalences and
fibrations created in $\mC$.
\end{thm}

The proof in all cases is much like the argument
in~\cite[VII\S4]{EKMM} or~\cite[2.3]{SS-AlgMod}.  Heuristically,
whenever the small object argument applies and $\mtimes$ behaves well
with respect to weak equivalences, pushouts, and sequential or
filtered colimits, a version of the previous theorem should hold.
For an example of a more general statement, see~\cite[4.1]{SS-AlgMod}.

A map of monoid objects $A\to B$ induces an obvious \term{restriction
of scalars} functor from the category of $B$-modules to the category
of $A$-modules.  When $\mC$ admits coequalizers and $\mtimes$ preserves coequalizers in
each variable (as is the case in the examples in the previous
theorem), the restriction of scalars functor admits a left adjoint \term{extension
of scalars} functor $B\mtimes_{A}(-)$ which on the underlying objects
is constructed as the coequalizer
\[
\xymatrix@C-1pc{%
B\mtimes A\mtimes M\ar@<-.5ex>[r]\ar@<.5ex>[r]&B\mtimes M\ar[r]&B\mtimes_{A}M,
}
\]
where one map is induced by the $A$-action on $M$ and the other by the
$A$ action on $B$ (induced by the map of monoid objects).  In the case
when the categories of modules have closed model structures with weak
equivalences and fibrations created in the underlying category $\mC$,
this adjunction is automatically a Quillen adjunction, which implies
a derived adjunction on homotopy categories.  When the map $A\to B$ is
a weak equivalence, we can often expect the Quillen adjunction to be a
Quillen equivalence and induce an equivalence of homotopy categories;
this is in particular the case in the setting of the previous theorem.

\begin{thm}
Let $\mC$ be one of the symmetric monoidal model categories of
Theorem~\ref{thm:modone}.  A weak equivalence of monoid objects
induces a Quillen equivalence on categories of modules.
\end{thm}

Again, significantly more general results hold; see, for
example, \cite{LewisMandell-MMMC}, especially Theorem~8.3 and the
subsection of Section~1 entitled ``Extension of scalars''.

\section{Limits and Colimits in Categories of Operadic Algebras}
\label{sec:limit}

Before going on to the homotopy theory of categories of operadic
algebras, we say a few words about certain constructions,
limits and colimits in this section, and geometric realization in the
next section.  While limits of operadic algebras are
pretty straightforward (as explained below), colimits tend to be more
complicated and we take some space to describe in detail what certain
colimits look like.  

We start with limits.  Let $D\colon \aD\to \mC[\oO]$ be a diagram,
i.e., a functor from a small category $\aD$, where $\mC$ is a
symmetric monoidal category and $\oO$ is an operad in $\mC$.  By
neglect of structure, we can regard $D$ as a diagram in $\mC$, and
suppose the limit $L$ exists in $\mC$.  Then for each $d\in \aD$,
we have the canonical map $L\to D(d)$, and using the $\oO$-algebra
structure map for $D(d)$, we get a map
\[
\oO(m)\mtimes L^{(m)}\to \oO(m)\mtimes D(d)^{(m)}\to D(d).
\]
These maps satisfy the required compatibility to define a map
\[
\oO(m)\mtimes L^{(m)}\to L,
\]
which together are easily verified to provide structure maps
for an $\oO$-algebra structure on $L$. This $\oO$-algebra structure
has the universal property for the limit of $D$ in $\mC[\oO]$.

\begin{prop}
For any symmetric monoidal category $\mC$, any operad $\oO$ in $\mC$,
and any diagram of $\oO$-algebras, if the limit exists in $\mC$, then
it has a canonical $\oO$-algebra structure that gives the limit in $\mC[\oO]$.
\end{prop}

We cannot expect general colimits of operadic algebras to be formed in
the underlying category, as can be seen from the examples of
coproducts of $\mtimes$-monoids ($\oAss$-algebras) or of commutative
$\mtimes$-monoids ($\oCom$-algebras).  The discussion of colimits
simplifies if we assume that $\mC$ has countable colimits and that
$\mtimes$ preserves countable colimits in each variable, so that
Proposition~\ref{prop:algmon} holds and the category of $\oO$-algebras
is the category of algebras over the monad $\bO$.  The main technical
tool in this case is the following proposition; because we have
assumed in particular that $\mtimes$ preserves coequalizers
in each variable, it follows that the $m$th $\mtimes$-power functor
preserves reflexive coequalizers (see \cite[II.7.2]{EKMM} for a proof)
and the filtered colimits that exist (by an easy cofinality argument).

\begin{prop}\label{prop:reflex}
If $\mC$ satisfies the hypotheses of Proposition~\ref{prop:monad},
then for any operad $\oO$, the monad $\bO$ preserves reflexive
coequalizers in $\mC$ and the filtered colimits that exist in $\mC$. 
\end{prop}

Recall that a reflexive coequalizer is a coequalizer
\[
\xymatrix@C-1pc{%
X\ar@<.5ex>[r]^{a}\ar@<-.5ex>[r]_{b}&Y\ar[r]^{c}&C
}
\]
where there exists a map $r\colon Y\to X$ such that $a\circ r=\id_{Y}$
and $b\circ r=\id_{Y}$; $r$ is called a \term{reflexion}.  The
proposition says that if the above coequalizer exists in $\mC$ and is
reflexive then the diagram obtained by applying $\bO$
\[
\xymatrix@C-1pc{%
\bO X\ar@<.5ex>[r]^{\bO a}\ar@<-.5ex>[r]_{\bO b}&\bO Y\ar[r]^{\bO c}&\bO C
}
\]
is also a (reflexive) coequalizer diagram in $\mC$.  Now suppose that 
$a$ and $b$ are maps of $\oO$-algebras.  Then the diagrams 
\[
\xymatrix@-1pc{%
\bO X\ar[r]^{\bO a}\ar[d]&\bO Y\ar[d]
&&\bO X \ar[r]^{\bO b}\ar[d]&\bO Y\ar[d]\\
X\ar[r]_{a}&Y&&X\ar[r]_{b}&Y
}
\]
commute (where the vertical maps are the $\oO$-algebra structure maps)
and we get an induced map 
\[
\bO C\to C.
\]
Repeating this for $\bO X\toto \bO Y$ and the two maps $\bO\bO X\toto
\bO\bO Y$ to $\bO X\toto \bO Y$, we see that the map $\bO C\to C$
constructed above is an $\oO$-algebra structure map and an easy check
of universal properties shows that $C$ with this $\oO$-algebra
structure is the coequalizer in $\mC[\oO]$.  This shows that if a pair
of parallel arrows in $\mC[\oO]$ has a reflexion in $\mC$, then the
coequalizer in $\mC$ has the canonical structure of an $\oO$-algebra
and is the coequalizer in $\mC[\oO]$. 

We can turn the observation in the previous paragraph into a
construction of colimits of arbitrary shapes in $\mC[\oO]$.  Given a
diagram $D\colon \aD\to \mC[\oO]$, assume that the colimit of the
underlying functor to $\mC$ exists, and denote it by
$\colim^{\mC}D$.  If $\colim^{\mC}\bO D$ also exits, then we get a
pair of parallel arrows
\begin{equation}\label{eq:coeq}
\xymatrix@C-1pc{%
\bO(\colim^{\mC} \bO D)\ar@<.5ex>[r]\ar@<-.5ex>[r]&\bO(\colim^{\mC} D)
}
\end{equation}
where one arrow is induced by the $\oO$-algebra structure maps $\bO
D(d)\to D(d)$ and the other is the composite
\[
\bO(\colim^{\mC} \bO D)\overto{\bO i} \bO\bO(\colim^{\mC} D)\overto{\mu} \bO(\colim^{\mC} D)
\]
where $\mu$ is the monadic multiplication $\bO\bO\to \bO$ and 
\[
i\colon \colim^{\mC} \bO D\to \bO(\colim^{\mC} D)
\]
is the map assembled from the maps $\bO D(d)\to \bO(\colim^{\mC} D)$
induced by applying $\bO$ to the canonical maps $D(d)\to \colim^{\mC}
D$.  We also have a reflexion
\[
\bO(\colim^{\mC} D)\to \bO(\colim^{\mC} \bO D)
\]
induced by the unit map $D(d)\to \bO D(d)$.  Thus, the coequalizer
of~\eqref{eq:coeq} in $\mC$ has the canonical structure of an
$\oO$-algebra which provides the coequalizer in $\mC[\oO]$; a check of
universal properties shows that the coequalizer is the colimit in
$\mC[\oO]$ of~$D$.

\begin{prop}
Assume $\mC$ satisfies the hypotheses of
Proposition~\ref{prop:monad}.  For any operad $\oO$ and any diagram
$D\colon \aD\to \mC[\oO]$, if the colimit of $D$ and the colimit
of $\bO D$ exist in $\mC$, then the colimit of $D$ exists in
$\mC[\oO]$ and is given by the coequalizer of the reflexive pair
displayed in~\eqref{eq:coeq}.
\end{prop}

For example, the coproduct $A\amalg^{\mC[\oO]} B$ in $\mC[\oO]$ can be constructed as the
coequalizer 
\[
\xymatrix@C-1pc{%
\bO(\bO A \amalg \bO B)\ar@<.5ex>[r]\ar@<-.5ex>[r]
&\bO(A \amalg B)\ar[r]&A\amalg^{\mC[\oO]}B.
}
\]
In the case when $B=\bO X$ for some $X$ in $\mC$, we can say more by
recognizing that the category of $\oO$-algebras under $A$ is itself
the category of algebras over an operad.  

\begin{cons}[The enveloping operad]\label{cons:eop}
For
$m\geq 0$, define $\oU^{\oO}_{A}(m)$ by the coequalizer diagram
\begin{equation*}\label{eq:env}
\xymatrix@C-1pc{%
\myop\coprod_{\ell=0}^{\infty}\oO(\ell+m)\mtimes_{\Sigma_{\ell}} (\bO A)^{(\ell)}\ar@<.5ex>[r]\ar@<-.5ex>[r]
&\myop\coprod_{\ell=0}^{\infty}\oO(\ell+m)\mtimes_{\Sigma_{\ell}} A^{(\ell)}\ar[r]
&\oU^{\oO}_{A}(m)
}
\end{equation*}
where one arrow is induced by the operadic multiplication 
\[
\Gamma^{\ell+m}_{j_{1},\dotsc,j_{\ell},1,\dotsc,1}\colon  \oO(\ell+m)\mtimes
\oO(j_{1})\mtimes \dotsb \mtimes \oO(j_{\ell})\mtimes 
\mS\mtimes\dotsb \mtimes\mS\to \oO(j+m)
\]
and the other by the $\oO$-algebra action map $\bO A\to A$.  We think
of the $\ell$ factors of $A$ (or $\bO A$) as being associated with the
first $\ell$ inputs of $\oO(\ell+m)$, leaving the last $m$ inputs open.  We
then have a $\Sigma_{m}$-action induced from the $\Sigma_{m}$-action
on $\oO(\ell+m)$ on the open inputs, unit map $1\colon \mS\to
\oU^{\oO}_{A}(1)$ induced by the unit map of $\oO$ (on the summand
$\ell=0$), and operadic composition $\Gamma$ induced by applying the
operadic multiplication of $\oO$ using the open inputs.  
\end{cons}

This operad is called the \term{enveloping operad} of $A$ and
generalizes the enveloping algebra $U^{\oO}A$ of
Construction~\ref{cons:UA}: for $m=1$, $\oU^{\oO}_{A}(1)$ is precisely the
coequalizer defining $U^{\oO}A$ and the operad unit and multiplication
$\Gamma^{1}_{1}$ coincide with the $\mtimes$-monoid unit and
multiplication. 

To return to the discussion of the category of $\oO$-algebras under $A$,
we note that for $m=0$, the coequalizer in Construction~\ref{cons:eop} is
\[
\xymatrix@C-1pc{%
\bO \bO A\ar@<.5ex>[r]\ar@<-.5ex>[r]&\bO A\ar[r]&\oU^{\oO}_{A}(0),
}
\]
giving a canonical isomorphism $A\to \oU^{\oO}_{A}(0)$, and so a
$\oU^{\oO}_{A}$-algebra $T$ comes with a structure map $A\to T$.
Looking at the summands with $\ell=0$ above, we get a map of operads
$\oO\to \oU^{\oO}_{A}$, giving $T$ an
$\oO$-algebra structure; the map $A\to T$ is a map of $\oO$-algebras.
On the other hand, given an $\oO$-algebra $B$ and a map of
$\oO$-algebras $A\to B$, we have maps
\[
\oO(\ell+m)\mtimes A^{(\ell)} \mtimes B^{(m)}\to 
\oO(\ell+m)\mtimes B^{(\ell)} \mtimes B^{(m)} \to B
\]
which together induce maps $\oU^{\oO}_{A}(m)\mtimes B^{(m)}\to B$ that are
easily checked to provide
$\oU^{\oO}_{A}$-algebra structure maps.  This gives a bijection
between the structure of an $\oO$-algebra under $A$ and the structure
of a $\oU^{\oO}_{A}$-algebra.

\begin{prop}
When $\mC$ satisfies the hypotheses of Proposition~\ref{prop:monad},
then for an object $X$ of $\mC$, the set of $\oU^{\oO}_{A}$-algebra
structures on $X$ is in bijective correspondence with the set of
ordered pairs consisting of an $\oO$-algebra structure on $X$ and a
map of $\oO$-algebras $A\to X$ for that structure.
\end{prop}

As a consequence we have the following description of the coproduct of
$\oO$-algebras $A\amalg^{\mC[\oO]} \bO X$, since $A\amalg^{\mC[\oO]}\bO(-)$ is
the left adjoint of the forgetful functor from $\oO$-algebras under
$A$ to $\mC$.

\begin{prop}
When $\mC$ satisfies the hypotheses of Proposition~\ref{prop:monad},
\[
A\amalg^{\mC[\oO]} \bO X\iso \bU^{\oO}_{A}X = 
\myop\coprod_{m=0}^{\infty}\oU^{\oO}_{A}(m)\mtimes_{\Sigma_{m}}X^{(m)}
\]
(where the coproduct symbol undecorated by a category denotes coproduct
in $\mC$).
\end{prop}

The decomposition above can be useful even without further
information on $\oU^{\oO}_{A}$, but in fact we can be more concrete
about what $\oU^{\oO}_{A}$ looks like in the case when $A$ is built up
iteratively from pushouts of free objects in $\mC[\oO]$.  As a base
case, an easy calculation gives
\[
\oU^{\oO}_{\bO X}(m)=\myop\coprod_{\ell=0}^{\infty}\oO(\ell+m)\mtimes_{\Sigma_{\ell}} X^{(\ell)}.
\]
Now suppose $A'=A\amalg^{\mC[\oO]}_{\bO X}\bO Y$ for some maps $X\to
A$ and $X\to Y$ in $\mC$; we can then describe $\oU^{\oO}_{A'}$ in
terms of $\oU^{\oO}_{A}$ and pushouts in $\mC[\oO]$ as follows.  (In
particular, the calculation of $\oU^{\oO}_{A'}(0)$ describes $A'$ in
these terms and the calculation of $\oU^{\oO}_{A'}(1)$ describes $UA'$
in these terms.)  First, using the observations on colimits above, a
little work shows that the coequalizer defining $\oU^{\oO}_{A'}$
simplifies in this case to
\[
\xymatrix@C-1pc{%
\myop\coprod_{\ell=0}^{\infty}\oU^{\oO}_{A}(\ell+m)
  \mtimes_{\Sigma_{\ell}}(X\amalg Y)^{(\ell)}
  \ar@<.5ex>[r]\ar@<-.5ex>[r]
&\myop\coprod_{\ell=0}^{\infty}\oU^{\oO}_{A}(\ell+m)
  \mtimes_{\Sigma_{\ell}}Y^{(\ell)}
   \ar[r]&\oU^{\oO}_{A'}(m)
}
\]
where one map is induced by the map $X\to A$ ($=\oU^{\oO}_{A}(0)$) and
the other is induced by the map $X\to Y$.
We then have a filtration on $\oU^{\oO}_{A'}(m)$ by powers of $Y$;
specifically, define $F^{\kay}\oU^{\oO}_{A'}(m)$ by the coequalizer
\[
\xymatrix@C-1pc{%
\myop\coprod_{\ell=0}^{\kay}\oU^{\oO}_{A}(\ell+m)
  \mtimes_{\Sigma_{\ell}}(X\amalg Y)^{(\ell)}
  \ar@<.5ex>[r]\ar@<-.5ex>[r]
&\myop\coprod_{\ell=0}^{\kay}\oU^{\oO}_{A}(\ell+m)
  \mtimes_{\Sigma_{\ell}}Y^{(\ell)}
   \ar[r]&F^{\kay}\oU^{\oO}_{A'}(m)
}
\]
Then $\colim_{\kay} F^{\kay}\oU^{\oO}_{A'}(m)=\oU^{\oO}_{A'}(m)$.
Comparing the universal properties for $F^{\kay-1}\oU^{\oO}_{A'}(m)$
and $F^{\kay}\oU^{\oO}_{A'}(m)$, we see that
the following diagram is a pushout (in $\mC$).
\[
\xymatrix{%
\oU^{\oO}_{A}(\kay+m)\mtimes_{\Sigma_{\kay-1}}(X\mtimes Y^{(\kay-1)})
  \ar[r]\ar[d]
&\oU^{\oO}_{A}(\kay+m)\mtimes_{\Sigma_{\kay}}Y^{(\kay)}\ar[d]\\
F^{\kay-1}\oU^{\oO}_{A'}(m)\ar[r]&F^{\kay}\oU^{\oO}_{A'}(m)
}
\]

This describes $\oU^{\oO}_{A'}$ in terms of iterated pushouts in
$\mC$, but we can do somewhat better, as can be seen in the example
where $\mC$ is the category of spaces and $X\to Y$ is a closed
inclusion.  In the pushout above, the top horizontal map comes from
the map
\[
\Sigma_{\kay}\times_{\Sigma_{\kay-1}}(X\times Y^{\kay-1})\to Y^{\kay}
\]
which fails to be an inclusion for $\kay>1$ except in trivial cases;
however, the image of this map can be described as an iterated pushout, starting
with $X^{\kay}$ and gluing in higher powers of $Y$.  This works as
well in the general case (which we now return to).

Let $Q^{\kay}_{0}(X\sto Y)=X^{(\kay)}$, an object of $\mC$ with
a $\Sigma_{\kay}$-action and a $\Sigma_{\kay}$-equivariant map to
$Y^{(\kay)}$.  Inductively, for $i>0$, define $Q^{\kay}_{i}(X\sto Y)$
as the pushout 
\begin{equation}\label{eq:Ql}
\begin{gathered}
\xymatrix@C-1ex{%
\Sigma_{\kay}\times_{\Sigma_{\kay-i}\times\Sigma_{i}}
  (X^{(\kay-i)}\mtimes Q^{i}_{i-1}(X\sto Y))\ar[r]\ar[d]&
\Sigma_{\kay}\times_{\Sigma_{\kay-i}\times \Sigma_{i}} 
  (X^{(\kay-i)}\mtimes Y^{(i)})\ar[d]\\
Q^{\kay}_{i-1}(X\sto Y)\ar[r]
&Q^{\kay}_{i}(X\sto Y)
}
\end{gathered}
\end{equation}
with the evident $\Sigma_{\kay}$-action and
$\Sigma_{\kay}$-equivariant map
\[
Q^{\kay}_{i}(X\sto Y)\to Y^{(\kay)}.
\]
Then for all $j>0$, we have a $(\Sigma_{j}\times
\Sigma_{\kay})$-equivariant map 
\[
X^{(j)}\mtimes Q^{\kay}_{i}(X\sto Y)\to Q^{j+\kay}_{i}(X\sto Y)
\]
induced by the map 
\[
X^{(j)}\mtimes X^{(\kay-i)}\mtimes Y^{(i)}\iso X^{(j+\kay-i)}\mtimes Y^{(i)}\to 
   Q^{j+\kay}_{i}(X\sto Y)
\]
and the compatible (inductively defined) map
\[
X^{(j)}\mtimes Q^{\kay}_{i-1}(X\sto Y)\to Q^{j+\kay}_{i-1}(X\sto Y)
\to Q^{j+\kay}_{i}(X\sto Y),
\]
which allows us to continue the induction.  In the case when $\mC$ is
the category of topological spaces and $X\to Y$ is a closed inclusion,
the maps
\[
Q^{\kay}_{0}(X\sto Y)\to \dotsb \to Q^{\kay}_{\kay-1}(X\sto Y)\to Y^{(\kay)}
\]
are closed inclusions with $Q^{\kay}_{i}(X\sto Y)$ the subspace of $Y^{\kay}$
where at most $i$ coordinates are in $Y\setminus X$. In the general case, an
inductive argument shows that the map 
\[
\Sigma_{\kay}\times_{\Sigma_{\kay-i}\times \Sigma_{i}}(X^{(\kay-i)}\mtimes Y^{(i)})
\to Q^{\kay}_{i}(X\sto Y)
\]
is a categorical epimorphism and that the map
\[
\oU^{\oO}_{A}(\kay+m)\mtimes_{\Sigma_{\kay-1}}
  (X\mtimes Y^{(\kay-1)})
\to \oU^{\oO}_{A}(\kay+m)\mtimes_{\Sigma_{\kay}}Q^{\kay}_{\kay-1}(X\sto Y)
\]
is a categorical epimorphism.  Since this factors the map
\[
\oU^{\oO}_{A}(\kay+m)\mtimes_{\Sigma_{\kay-1}}
  (X\mtimes Y^{(\kay-1)})
\to \oU^{\oO}_{A}(\kay+m)\mtimes_{\Sigma_{\kay}}Y^{(\kay)},
\]
we get the following more sophisticated
identification of $F^{\kay}\oU^{\oO}_{A'}(m)$ as a pushout.
\begin{equation}\label{eq:poenv}
\begin{gathered}
\xymatrix{%
\oU^{\oO}_{A}(\kay+m)\mtimes_{\Sigma_{\kay}}Q^{\kay}_{\kay-1}(X\sto Y)
  \ar[r]\ar[d]
&\oU^{\oO}_{A}(\kay+m)\mtimes_{\Sigma_{\kay}}Y^{(\kay)}\ar[d]\\
F^{\kay-1}\oU^{\oO}_{A'}(m)\ar[r]&F^{\kay}\oU^{\oO}_{A'}(m)
}
\end{gathered}
\end{equation}
In practice, the map $Q^{\kay}_{\kay-1}(X\to Y)\to Y^{(\kay)}$ is some
kind of cofibration when $X\to Y$ is nice enough; the above
formulation is then useful for deducing homotopical
information in the presence of cofibrantly generated model category
structures, as discussed in Section~\ref{sec:model}.

\section{Enrichment and Geometric Realization}
\label{sec:enrich}

Categories of operadic algebras in spaces or spectra come with a
canonical enrichment in spaces, i.e., they have mapping spaces and an
intrinsic notion of homotopy.  While more abstract notions of
homotopy, for example, in terms of model structures, now play a
more significant role in homotopy theory, the topological enrichment
provides some powerful tools, including and especially geometric
realization of simplicial objects.

We begin with a general discussion of enrichment of operadic algebra
categories.
When $\mC$ satisfies the hypotheses of Proposition~\ref{prop:monad},
Proposition~\ref{prop:algmon} describes the maps in the category of
$\oO$-algebras as an equalizer:
\[
\xymatrix@C-1pc{%
\mC[\oO](A,B)\ar[r]&\mC(A,B)\ar@<-.5ex>[r]\ar@<.5ex>[r]&\mC(\bO A,B)
}
\]
where one arrow $\mC(A,B)\to \mC(\bO A,B)$ is induced by the action
map $\bO A\to A$ and the other arrow is induced by applying the
functor $\bO\colon \mC(A,B)\to \mC(\bO A,\bO B)$ and then using the
action map $\bO B\to B$.  When $\mC$ is enriched over a complete
symmetric monoidal category (for example, when the mapping sets of
$\mC$ are topologized or simplicial), then $\mC[\oO]$ becomes enriched
exactly when $\bO$ has the structure of an enriched functor, defining
the enrichment of $\mC[\oO]$ by the equalizer above.  Clearly it is
not always possible for $\bO$ to be enriched: in the case when $\mC$
is the category of abelian groups and $\oO=\oAss$ or $\oCom$, then
$\bO$ is not an additive functor so cannot be enriched over abelian
groups; this corresponds to the fact that the category of rings and
the category of commutative rings are not enriched over abelian
groups.  On the other hand, enrichments over spaces and simplicial
sets are usually inherited by algebra categories; the reason, as we
now explain, derives from the fact that spaces and simplicial sets are
cartesian.

For convenience, consider the case when $\mC$ is a closed symmetric
monoidal category, let $\mE,\metimes,\mSE$ be a symmetric monoidal
category (which we will eventually assume to be cartesian), and let
$L\colon \mE\to \mC$ be a strong symmetric monoidal functor that is a
left adjoint; let $R$ denote its right adjoint.  For formal reasons
$R$ is then lax symmetric monoidal and in particular $RF$ provides an
$\mE$-enrichment of $\mC$ (where, as always, $F$ denotes the mapping
object in $\mC$).  These hypotheses are not all necessary but avoid
some review of enriched category theory and concisely state a lot
coherence data that more minimal hypotheses would force us to spell
out.  The iterated symmetric monoidal product in $\mC$ then gives a
multivariable enriched functor
\[
RF(A_{1},B_{1}) \metimes \dotsb \metimes RF(A_{m},B_{m})\to 
RF(A_{1}\mtimes \dotsb \mtimes A_{m},B_{1}\mtimes \dotsb \mtimes B_{m}).
\]
Now assume that $\metimes$ is a cartesian monoidal product, meaning
that it is the categorical product, the unit is the final object,
and the symmetry and unit isomorphisms are the universal ones.  With
this assumption, we have a natural diagonal map $E\to E\metimes E$,
which we can apply in particular to the object $RF(A,B)$ to get a
natural map 
\begin{equation}\label{eq:enr}
RF(A,B)\to RF(A,B)\metimes\dotsb \metimes RF(A,B)\to RF(A^{(m)},B^{(m)}).
\end{equation}
This makes the $m$th $\mtimes$-power into an $\mE$-enriched functor
for $m>0$.  In the case $m=0$, we have the final map 
\[
RF(A,B)\to *\to R\mS\overto{\iso} RF(A^{(0)},B^{(0)}).
\]
From here the rest is easy: 
the $\mtimes,F$ adjunction gives a natural (and $\mE$-natural) map
\[
RF(A^{(m)},B^{(m)})\to RF(\oO(m)\mtimes A^{(m)},\oO(m)\mtimes B^{(m)})
\]
and the composite to $RF(\oO(m)\mtimes
A^{(m)},\oO(m)\mtimes_{\Sigma_{m}} B^{(m)})$  admits a canonical factorization
\[
RF(A,B)\to RF(\oO(m)\mtimes_{\Sigma_{m}} A^{(m)},\oO(m)\mtimes_{\Sigma_{m}} B^{(m)})
\]
since the target is a limit (in $\mE$) that exists by right adjoint
considerations when the quotient
$\oO(m)\mtimes_{\Sigma_{m}} B^{(m)}=(\oO(m)\mtimes
B^{(m)})/\Sigma_{m}$ in $\mC$ exists.  When we assume that $\mC$ has
countable coproducts, composing further into
\[
RF(\oO(m)\mtimes_{\Sigma_{m}} A^{(m)},\bO B),
\]
the countable categorical
product over $m$ exists, giving an $\mE$-natural map
\[
RF(A,B)\to RF(\bO A,\bO B)
\]
which provides the $\mE$-enrichment of $\bO$.  We state this as the
following theorem. 

\begin{thm}\label{thm:enrich}
Let $\mC$ be a closed symmetric monoidal category with countable colimits,
and let $\oO$ be an operad in $\mC$.  Let $\mE$ be a cartesian monoidal
category and let $\mE\to \mC$ be a strong
symmetric monoidal functor with a right adjoint.  Regarding $\mC$ as
$\mE$-enriched over the right adjoint, the category $\mC[\oO]$ of
$\oO$-algebras has a canonical $\mE$-enrichment for which the forgetful
functor $\mC[\oO]\to \mC$ is $\mE$-enriched.
\end{thm}

We apply this now in the discussion of geometric realizations of
(co)simplicial objects.  Let 
$\Top$ denote either the category of spaces or of simplicial sets, and
write $C(-,-)$ for the internal mapping objects in $\Top$. To avoid
awkward circumlocutions, we will refer to objects of $\Top$ as spaces
in either case for the rest of the section.  We now assume that $\mC$
is closed symmetric monoidal and has countable coproducts and that we
have a left adjoint symmetric monoidal functor $L$ from $\Top$ to
$\mC$, as above, so that Theorem~\ref{thm:enrich} applies. We write $R$ for the
right adjoint to $L$ as above, so that $RF(-,-)$
provides mapping spaces for $\mC$.  The category $\mC$ then has
\term{tensors} $X\mtensor T$ and \term{cotensors} $T\pitchfork Y$, defined
by the natural isomorphisms
\begin{align*}
RF(X\mtensor T,-)&\iso C(T,RF(X,-))&&\text{(tensor)}\\
RF(-,T\pitchfork Y))&\iso C(T,RF(-,Y))&&\text{(cotensor)}
\end{align*}
for spaces $T$ and objects $X$ and $Y$ of $\mC$,
constructed as follows.  

\begin{prop}
In the context of Theorem~\ref{thm:enrich}, tensors and cotensors with spaces exist and are
given by $X\mtensor T = X\mtimes LT$ and  $T\pitchfork Y=F(LT,Y)$ for
a space $T$ and objects $X,Y$ in $\mC$.
\end{prop}

The proposition is an easy consequence of the formal isomorphism 
\begin{equation}\label{eq:MSComp}
RF(LT,X)\iso C(T,RX),
\end{equation}
natural in spaces $T$ and objects $X$ of $\mC$; the isomorphism in
the forward direction is adjoint to the map 
\[
RF(LT,X)\times T\to RF(LT,X)\times RLT\to R(F(LT,X)\mtimes LT)\to RX
\]
and the isomorphism in the backwards direction is adjoint to the map
$LC(T,RX)\to F(LT,X)$ adjoint to the map
\[
LC(T,RX)\mtimes LT\iso L(C(T,RX)\times T)\to LRX\to X.
\]

Let $RF^{\mC[\oO]}(-,-)$ denote the mapping spaces constructed above
for the category of $\oO$-algebras;
despite the suggestion of the notation, this 
is not typically a composite functor.  For an $\oO$-algebra $A$,
$F(-,A)$ does not typically carry a canonical $\oO$-algebra structure,
but for a space $T$, $F(LT,A)=T\pitchfork A$ does:  the structure map 
\[
\oO(n)\mtimes (T\pitchfork A)^{(n)}\to T\pitchfork A
\]
is adjoint to the map
\[
\oO(n)\mtimes (T\pitchfork A)^{(n)}\mtimes LT =
\oO(n)\mtimes (F(LT,A))^{(n)}\mtimes LT \to A
\]
constructed as the composite
\begin{multline*}
\oO(n)\mtimes (F(LT,A))^{(n)}\mtimes LT \to
\oO(n)\mtimes (F(LT,A))^{(n)}\mtimes (LT)^{(n)} \\\to
\oO(n)\mtimes A^{(n)}\to A
\end{multline*}
using the diagonal map on the space $T$ and the structure map on $A$.
A check of universal properties then shows that $T\pitchfork A$ is the
cotensor of $A$ with $T$ in the category of $\oO$-algebras.  Tensors
in $\mC[\oO]$ can be constructed as reflexive coequalizers
\[
\xymatrix@C-1pc{%
\bO(\bO A\mtensor T)
  \ar@<.5ex>[r]\ar@<-.5ex>[r]
&\bO(A\mtensor T)\ar[r]
&A\mtensor^{\mC[\oO]} T.
}
\]
Writing $\Delta[n]$ for the standard $n$-simplex, we then have the
standard definition of geometric realization of simplicial objects in
$\mC$ and $\mC[\oO]$ (without additional assumptions) and geometric
realization (often called ``Tot'') of cosimplicial objects in $\mC$
and $\mC[\oO]$ when certain limits exist.  Given a simplicial object
$X\subdot$ or a cosimplicial object $Y\supdot$, the degeneracy
subobject $sX_{n}$ of $X_{n}$ is defined as the colimit of the
denegeracy maps and the degeneracy quotient object $sY^{n}$ of $Y^{n}$
is defined as the limit (if it exists) of the degeneracy maps. (In
some literature, $sX_{n}$ is called the ``latching object'' and
$sY^{n}$ the ``matching object''; see
\cite[\S15.2]{Hirschhorn-ModelCategories}.)  The geometric realization
of $X\subdot$ in $\mC$ or $\mC[\oO]$ is then the sequential colimit
of $|X\subdot|_{n}$, where $|X\subdot|_{0}=X_{0}$ and
$|X\subdot|_{n}$ is defined inductively as the pushout
\[
\xymatrix@C-1pc{%
(sX_{n}\mtensor \Delta[n]) \cup_{(sX_{n}\mtensor \partial \Delta[n])} 
  (X_{n}\mtensor \partial \Delta[n])\ar[r]\ar[d]
&X_{n}\mtensor \Delta[n]\ar[d]\\
|X\subdot|_{n-1}\ar[r]&|X\subdot|_{n}
}
\] 
with both the tensor and the pushouts performed in $\mC$ to define the
geometric realization in $\mC$ or performed in $\mC[\oO]$ to define
the geometric realization in $\mC[\oO]$.  The analogous, opposite
construction defines the geometric realization of $Y\supdot$ when all
the limits exist.  Because cotensors and limits (when they exist)
coincide in $\mC$ and $\mC[\oO]$, geometric realization of
cosimplicial objects (when it exists) also coincides in $\mC$ and
$\mC[\oO]$.  Because pushouts generally look very different in $\mC$
than in $\mC[\oO]$, one might expect that geometric realization of
simplicial objects in $\mC$ and in $\mC[\oO]$ would also look very
different; this turns out not to be the case.

\begin{thm}\label{thm:georeal}
Assume $\mC$ satisfies the hypotheses of Theorem~\ref{thm:enrich} for
$\mE$ either the category of spaces or the category of simplicial sets.
\begin{enumerate}
\item Let $A\supdot$ be a cosimplicial object in $\mC[\oO]$. If the
limits defining the geometric realization (``Tot'') exist in $\mC$, then that
geometric realization has the canonical structure of an $\oO$-algebra
and is isomorphic to the geometric realization (``Tot'') in $\mC[\oO]$.
\item Let $A\subdot$ be a simplicial object in $\mC[\oO]$. Then the
geometric realization of $A\subdot$ in $\mC$ has the canonical
structure of an $\oO$-algebra and is isomorphic to the geometric
realization of $A\subdot$ in $\mC$.
\end{enumerate}
\end{thm}

As discussed above, only~(ii) requires additional argument. For
clarity in the argument for the theorem, we will write $|{\cdot}|$ for
geometric realization in $\mC$ and $|{\cdot}|^{\mC[\oO]}$ for
geometric realization in $\mC[\oO]$.  The key fact is the following
lemma.

\begin{lem}\label{lem:prod}
For $\mC$ as in the previous theorem, geometric realization in
$\mC$ is strong symmetric monoidal.
\end{lem}

\begin{proof}
Although we wrote a more constructive definition of geometric
realization above, it can also be described as a coend
\[
|X\subdot|=\int^{\DDelta^{\op}}X\subdot\mtensor \Delta[\ssdot],
\]
where $\DDelta$ denotes the category of simplexes (the category with
objects $[n]=\{0,\dotsc,n\}$ for $n=0,1,2,\dotsc$, and maps the
non-decreasing functions) and $\Delta[n]$ denotes the standard
$n$-simplex in spaces or simplicial sets.
Because the symmetric monoidal product $\mtimes$ for $\mC$ is assumed
to commute with colimits in each variable,  we can
identify the product of geometric realizations also as a coend
\begin{equation*}\label{eq:grc}
|X\subdot|\mtimes|Y\subdot| \iso 
  \int^{\DDelta^{\op}\times \DDelta^{\op}}
    (X\subdot\mtimes Y\subdot)\mtensor 
       (\Delta[\ssdot]\times \Delta[\ssdot]).
\end{equation*}
On the other hand,
\[
|X\subdot \mtimes Y\subdot|=
   \int^{\DDelta^{\op}}\diag(X\subdot\mtimes Y\subdot) \mtensor \Delta[\ssdot].
\]
Next, we need a purely formal observation, which is an adjoint form of
the Yoneda lemma: if coproducts of appropriate cardinality exist in
$\aC$, then given a functor 
$F\colon \aC\to \aD$, functoriality of $F$ induces a natural isomorphism 
\[
\int^{c\in \aC}F(c)\times \aC(c,-)\overto{\iso}F(-)
\]
(where $\times$ denotes coproduct over the given set; the coend
exists and the identification holds with no further hypotheses on
$\aC$ or $\aD$).  Applying this to 
\[
F((\bullet,\bullet))=X\subdot \mtimes Y\subdot\colon
\DDelta^{\op}\times \DDelta^{\op}\to \mC
\]
and pre-composing with $\diag$, we get an isomorphism
\[
X_{p}\mtimes Y_{p}\iso \int^{(m,n)\in \DDelta^{\op}\times
\DDelta^{\op}}(X_{m}\mtimes Y_{n})\times (\DDelta^{\op}(m,p)\times \DDelta (n,p))
\]
of functors $p\in \DDelta^{\op}\to \mC$.  Commuting coends, we can reorganize the
double coend
\[
|X\subdot \mtimes Y\subdot|\iso
   \int^{p\in \DDelta^{\op}}\hspace{-1ex}
\bigg(
   \int^{(m,n)\in \DDelta^{\op}\times \DDelta^{\op}} \hspace{-1.5ex}
   (X_{m}\mtimes Y_{n})
   \times (\DDelta^{\op}(m,p)\times \DDelta^{\op}(n,p))
\bigg) 
\mtensor \Delta[p]
\]
as 
\[
\int^{(m,n)\in \DDelta^{\op}\times \DDelta^{\op}}
   (X_{m}\mtimes Y_{n})
\mtensor
\bigg(
   \int^{p\in \DDelta^{\op}}
   (\DDelta^{\op}(m,p)\times \DDelta^{\op}(n,p))\times 
   \Delta[p]
\bigg).
\]
In the latter formula, the expression in parentheses is the coend
formula for the geometric realization (in spaces) of the product of
standard simplices (in simplicial sets) $\Delta[m]\subdot\times
\Delta[n]\subdot$, which is $\Delta[m]\times \Delta[n]$ by the
``classic'' version of the lemma for geometric realization in
spaces. This then constructs the natural isomorphism
$|X\subdot|\mtimes |Y\subdot|\iso |X\subdot\mtimes Y\subdot|$, and a
little more fiddling with coends shows that this natural
transformation is symmetric monoidal.
\end{proof}

As a consequence of the previous lemma, we have a natural isomorphism
$\bO|X\subdot|\iso |\bO X\subdot|$, making the appropriate diagrams
commute so that the geometric realization (in $\mC$) of a simplicial
object $A\subdot$ in $\mC[\oO]$ obtains the natural structure of an
$\oO$-algebra.  Moreover, the canonical maps $A_{n}\mtensor
\Delta[n]\to |A\subdot|$
induce maps of $\oO$-algebras $A_{n}\mtensor^{\mC[\oO]}\Delta[n]\to
|A\subdot|$ that assemble into a natural map of $\oO$-algebras
\[
|A\subdot|^{\mC[\oO]}\to |A\subdot|.
\]
In the case when $A\subdot=\bO X\subdot$, under the identification of
colimits $|\bO X\subdot|^{\mC[\oO]}=\bO|X\subdot|$, this map is the isomorphism 
$\bO|X\subdot|\to |\bO X\subdot|$
above.  To see that it is an isomorphism for arbitrary $A\subdot$,
write $A\subdot$ as a (reflexive) coequalizer
\[
\xymatrix@C-1pc{%
\bO \bO A\subdot\ar@<.5ex>[r]\ar@<-.5ex>[r]&\bO A\subdot\ar[r]&A\subdot,
}
\]
apply the functors, and compare diagrams.

\section{Model Structures for Operadic Algebras}
\label{sec:model}

The purpose of this section is to review the construction of model
structures on some of the categories of operadic algebras that are of interest in
homotopy theory; we use these in the next section in comparison theorems giving Quillen
equivalences between some of these categories. 
Constructing model structures for algebras over operads is a special
case of constructing model structures for algebras over monads;
chapter~VII of EKMM~\cite{EKMM} seems to be an early reference for
this kind of result, but it concentrates on the category of
LMS-spectra and related categories.
Schwede-Shipley~\cite{SS-AlgMod} studies the
general case of monads in cofibrantly generated monoidal model
categories, which Spitzweck~\cite{Spitzweck-Operads} specializes to the
case of operads.  Because less sharp results hold in the general case
than in the special cases of interest, we state the results on model 
structures as a list of examples.  This is an ``example theorem'' both
in the sense that it gives a list of examples, but also in the sense
that it fits into the 
general rubric of the kind of theorem that should hold very generally under
appropriate technical hypotheses with essentially the same proof outline.
Some terminology and notation is
explained after the statement.

\begin{ethm}\label{ethm:model}
Let $\mC$ be a symmetric monoidal category with a cofibrantly
generated model structure 
and let $\oO$ be an operad in $\mC$
from one of the examples listed below.  Then the category of
$\oO$-algebras in $\mC$ is a closed model category with:
\begin{enumerate}
\item Weak equivalences the underlying weak equivalences in $\mC$
\item Fibrations the underlying fibrations in $\mC$
\item Cofibrations the retracts of regular $\bO I$-cofibrations
\end{enumerate}
This theorem holds in particular in the examples:
\begin{enumerate}\deflist
\item $\mC$ is the category of symmetric spectra (of spaces or
simplicial sets) with its positive stable model structure or
orthogonal spectra with its positive stable model structure or the
category of EKMM $S$-modules with its standard model structure (with
$\mtimes$ the smash product, $\mS$ the sphere spectrum) and $\oO$ is
any operad in $\mC$. \cite[8.1]{ChadwickMandell}
\item $\mC$ is the category of spaces or simplicial sets (with
$\mtimes=\times$, $\mS=*$), or simplicial $R$-modules for some
simplicial commutative ring $R$ (with $\mtimes=\otimes_{R}$, $\mS=R$) and $\oO$
is any operad. 
\item $\mC$ is the category of (unbounded) chain
complexes in $R$-modules for a commutative ring $R$ (with $\mtimes
=\otimes_{R}$, $\mS=R$) and either $R\supset \bQ$ or $\oO$ admits a
map of operads $\oO\to \oO\otimes \oE$ which is a section for the map
$\oO\otimes \oE\to \oO\otimes \oCom\iso \oO$, where $\oE$ is any
$E_{\infty}$ operad that naturally acts on the normalized cohains of
simplicial sets. \cite[3.1.3]{BergerFresse-Combinatorial}
\item $\mC$ is a monoidal model category in the sense of
\cite[3.1]{SS-AlgMod} that satisfies the Monoid
Axiom of \cite[3.3]{SS-AlgMod} and $\oO$ is a
cofibrant operad in the sense of
\cite[\S3]{Spitzweck-Operads}. \cite[\S4, Theorem~4]{Spitzweck-Operads}
\end{enumerate}
\end{ethm}

The category of EKMM $\bL$-spectra \cite[I\S4]{EKMM} also fits into
example~(a) if we allow $\mC$ to be a ``weak'' symmetric monoidal
category in the sense of \cite[II.7.1]{EKMM}; the theorem then covers
categories of operadic algebras in LMS spectra for operads over the
linear isometries operad that have the form $\oO\times \oL\to \oL$;
see \cite[3.5]{ChadwickMandell}.

In part~(c), we note that for an operad that satisfies the section
condition (or when $R\supset \bQ$),
the functor $\oO(n)\times_{R[\Sigma_{n}]}(-)$ preserves
preserve exactness of (homologically)
bounded-below exact sequences of $R$-free $R[\Sigma_{n}]$-modules (for all $n$).
For operads that satisfy this more general condition but not
necessarily the section condition, the algebra category still has a
theory of cofibrant objects and a good homotopy theory for those
objects; see, for example, \cite[\S2]{Einfty}. 

It is beyond the scope of the present chapter to do a full review of
closed model category theory terminology, but we recall that a
``cofibrantly generated model category'' has a set $I$ of ``generating
cofibrations'' and a set $J$ of ``generating acyclic cofibrations''
for which the Quillen small object argument can be done (perhaps
transfinitely, but in the examples of (a), (b), and (c), sequences
suffice).  Then
\[
\bO I = \{ \bO f \mid f\in I\}
\] 
is the set of maps of $\oO$-algebras obtained by applying $\bO$ to
each of the maps in $I$.  The point of $\bO I$ is that a map of
$\oO$-algebras has the left lifting property with respect to $\bO I$
in $\oO$-algebras exactly when the underlying map in $\mC$ has the
left lifting property with respect to $I$.  The same definition and observations
apply replacing $I$ with $J$.  The strategy for proving the previous
theorem is to define the fibrations and weak equivalences of
$\oO$-algebras as in (i),(ii), and define cofibrations in terms of the
left lifting property (obtaining the characterization in (iii) as a
theorem).  The advantage of this approach is that fibrations and
acyclic fibrations are also characterized by lifting properties: a map
of $\oO$-algebras is a fibration 
if and only if it has the right lifting property with respect to $\bO
J$ and a map of $\oO$-algebras is an acyclic fibration if and only if
it has the right lifting property with respect to $\bO I$.  For
these lifting properties, we can attempt the small object argument.  We
now outline the remaining steps in this approach.

Recall that a \term{regular $\bO I$-cofibration} is a map formed as a
(transfinite) composite of pushouts along coproducts of maps in $\bO
I$. This is the generalization of the notion of a \term{relative
$\bO I$-cell complex} which is the colimit of a sequence of pushouts
of coproducts of maps in $\bO I$; in the case of examples~(a), (b),
and (c), in a regular $\bO I$-cofibration the transfinite composite
can always be replaced simply by a sequential composite and so a
regular $\bO I$-cofibration is a relative $\bO I$-cell complex.   The
small object argument for $I$ and $J$ in $\mC$ implies the small
object argument for $\bO I$ and $\bO J$, which gives factorization in
$\oO$-algebras of a map as either a regular $\bO I$-cofibration
followed by an acyclic fibration or a regular $\bO J$-cofibration
followed by a fibration.  (A small wrinkle comes up 
in going from the small object argument in $\mC$ to the small object
argument in $\mC[\oO]$ in the topological examples of (a) and (b): we
need to check that a regular $\bO I$-cofibrations are nice maps, for
example, closed inclusions on the constituent spaces; see the
``Cofibration Hypothesis'' of \cite[VII\S4]{EKMM} or
\cite[5.3]{MMSS}.)

This gets us most of the way to a model structure.  Having defined a
cofibration of $\oO$-algebras as a map that has the left lifting
property with respect to the acyclic fibrations, the free-forgetful
adjunction shows that regular $\bO I$-cofibrations are cofibrations;
moreover, it follows formally that any cofibration is the retract of a
regular $\bO I$-cofibration:
given a cofibration $f\colon A\to B$, factor it as
$p\circ i$ for $i\colon A\to B'$ a regular $\bO I$-cofibration and
$p\colon B'\to B$ an acyclic
fibration, then solving the lifting problem
\[
\xymatrix@-1pc{%
A\ar[r]^{i}\ar[d]_{f}&B'\ar[d]^{p}\\
B\ar@{..>}[ur]_{g}\ar[r]_{\id}&B
}
\]
to produce a map $g\colon B\to B'$ exhibits $f$ as a retract of $i$.
\[
\xymatrix@-1pc{%
A\ar[r]^{\id}\ar[d]_{f}&A\ar[r]^{\id}\ar[d]_{i}&A\ar[d]_{f}\\
B\ar[r]_{g}&B'\ar[r]_{p}&B
}
\]
We can try the same thing with regular $\bO J$-cofibrations; they have
the left lifting property with respect to all fibrations so are in
particular cofibrations, but are they weak equivalences?  This is the
big question and what keeps us from having a fully general result for
Theorem~\ref{ethm:model} (especially in (c)).  If regular $\bO
J$-cofibrations are weak equivalences, 
then the trick in the previous argument shows that every acyclic
cofibration is a retract of a regular $\bO J$-cofibration, and the
lifting property for acyclic cofibrations follows as does the other
factorization, proving the model structure.  (Conversely, if the model
structure exists, because regular $\bO J$-cofibrations have the left
lifting property for all fibrations, it follows that they are weak
equivalences.)  

In many examples, including examples (a) and (b) in the theorem above,
the homogeneous filtration on the pushout that we studied in
Section~\ref{sec:limit} can be used to prove that regular $\bO
J$-cofibrations are weak equivalences. Specifically, for $X\to Y$ a
map in $J$, taking $A'=A\amalg^{\mC[\oO]}_{\bO X}\bO Y$, the case
$m=0$ of the filtration on the enveloping operad for $A$ gives a
filtration on $A'$ by objects of $\mC$ starting from $A$.  Now from
the inductive definition of $Q^{\kay}_{\kay-1}(X\sto Y)$
in~\eqref{eq:Ql}, it can be checked in examples (a) and (b) that the
map $Q^{\kay}_{\kay-1}(X\sto Y)\to Y^{(\kay)}$ is an equivariant
Hurewicz cofibration of the underlying spaces or a monomorphism of the
underlying simplicial sets as well as being a weak equivalence.  The
pushout~\eqref{eq:poenv} then identifies the maps in the filtration of
$A'$ as weak equivalences as well.  (This approach can also be used to
prove versions of the ``Cofibration Hypothesis'' of \cite[VII\S4]{EKMM} or
\cite[5.3]{MMSS} that regular $\bO
I$-cofibrations are closed inclusions on the constituent spaces.)

Example~(d) is similar, except that it uses a filtration argument on
the construction of a cofibrant operad; see \cite[\S4]{Spitzweck-Operads}.  

Example~(c) fits into the case of the general
theorem of Schwede-Shipley \cite[2.3]{SS-AlgMod}, where every object
is fibrant and has a path object.  To complete the argument here, we
need to show that every map $f\colon A\to B$ factors as a weak
equivalence followed by a fibration:
\[
\xymatrix@-.75pc{%
A\ar[r]^-{\simeq}&A'\ar@{->>}[r]&B.
}
\]
We then get the factorization of an acyclic cofibration followed by a
fibration by using the factorization already established:
\[
\xymatrix@-.75pc{%
A\ar@{{ >}{-}{>}}[r]^-{\simeq}&A''\ar@{->>}[r]^-{\simeq}&A'\ar@{->>}[r]&B.
}
\]
In the case of~(c) where we
hypothesize a map of operads $\oO\to \oO\otimes \oE$, this map gives a
natural $\oO$-algebra structure on $B\otimes C^{*}(-)$; the hypothesis
that the composite map on $\oO$ is the identity implies that the
canonical isomorphism
\[
B\iso B\otimes C^{*}(\Delta[0])
\]
is an $\oO$-algebra map.  Looking at the maps between $\Delta[0]$ and
$\Delta[1]$, we get maps of $\oO$-algebras
\[
B\to B\otimes C^{*}(\Delta[1])\to B\times B
\]
and the usual mapping path object construction
\[
\xymatrix@-.75pc{%
A\ar[r]^-{\simeq}&A\times_{B}(B\otimes C^{*}(\Delta[1]))\ar@{->>}[r]&B
}
\]
consists of maps of $\oO$-algebras and 
gives the factorization.  In the case when $R\supset \bQ$, the
polynomial de Rham functor $A^{*}$ reviewed in
Section~\ref{sec:cochains} is a functor from simplicial sets to
commutative differential graded $\bQ$-algebras, which can be used in
the same way to construct a factorization
\[
\xymatrix@-.75pc{%
A\ar[r]^-{\simeq}&A\times_{B}(B\otimes_{\bQ} A^{*}(\Delta[1]))\ar@{->>}[r]&B.
}
\]

In the case of operadic algebras in spaces in example~(b) and EKMM
$S$-modules in example~(a), we have another argument taking advantage
of the topological enrichment.  In these examples, the maps in $J$ are
deformation retractions, and so the maps in $\bO J$ are deformation
retractions in the category of $\oO$-algebras.  It follows that
regular $\bO J$-cofibrations are also deformation retractions and in
particular homotopy equivalences.  Since homotopy equivalences are
weak equivalences, regular $\bO J$-cofibrations are weak equivalences
in examples where this argument can be made.  The specific examples
again also fit into the case of \cite[2.3]{SS-AlgMod} where every
object is fibrant and has a path object.

\section{Comparison and Rectification Theorems for Operadic Algebras}
\label{sec:compare}

This section discusses Quillen equivalences and Quillen adjunctions
between the model categories in Example Theorem~\ref{ethm:model}.
In particular, when we change from simplicial sets to spaces or when
we change the underlying symmetric monoidal category between the 
Quillen equivalent modern categories of spectra, we get Quillen
equivalences of categories of operadic algebras under only mild
technical hypotheses on the operad; this gives several comparison
theorems.  We also consider Quillen adjunctions and Quillen
equivalences obtained by change of operads.
In wide generality, the augmentation map $\oA\to \oAss$ for an $A_{\infty}$
operad induces a
Quillen equivalence between categories of algebras.  Likewise, in the case of modern categories of spectra, the
augmentation map $\oE\to \oCom$ for an $E_{\infty}$ operad induces a
Quillen equivalence between categories of algebras.  These comparison
theorems are rectification theorems in that they show that a
homotopical algebraic structure can be replaced up to weak equivalence
with a strict algebraic structure.  

We begin by reviewing the change of operad adjunction.  Let $f\colon
\oA\to \oB$ be a map of operads in a symmetric monoidal category
$\mC$.  Such a map certainly gives a restriction functor $U_{f}$ from
$\oB$-algebras to $\oA$-algebras, and under mild hypothesis, this
functor has a left adjoint.  As in the discussion of colimits in
Section~\ref{sec:limit}, if we assume that $\mC$ satisfies the
hypotheses of Proposition~\ref{prop:monad} then we can define
$P_{f}\colon \mC[\oA]\to \mC[\oB]$ by the reflexive coequalizer
\[
\xymatrix@C-1pc{%
\bB(\bA A)\ar@<.5ex>[r]\ar@<-.5ex>[r]&\bB A\to P_{f}(A)
}
\]
where $\bA$ and $\bB$ denote the monads associated to $\oA$ and $\oB$,
one arrow is induced by the $\oA$-algebra structure on $A$, and the
other arrow is the composite $\bB\bA\to \bB\bB\to \bB$ induced by the
map of operads $f$ and the monadic product on $\bB$.  As a side
remark, not related to the rest of this section, we note that in this
situation the category $\oB$-algebras can be identified as the
category of algebras for the monad $U_{f}P_{f}$ in $\mC[\oA]$ (for a general
formal proof, see~\cite[II.6.6.1]{EKMM}). 

Now suppose that $\mC$ has a closed model structure and $\mC[\oA]$ and
$\mC[\oB]$ are closed model categories with fibrations and weak
equivalences created in $\mC$.  For a map of operads $f\colon \oA\to
\oB$, we then get a Quillen adjunction
\[
P_{f}\colon \mC[\oA]\tofrom \mC[\oB]\noloc U_{f}.
\]
When can we expect it to be a Quillen equivalence?  It is tempting to
define an equivalence of operads in $\mC$ to be a map $f$ such that
derived adjunction induces an equivalence of homotopy categories; then
we have a tautological result that an equivalence of operads induces a
Quillen equivalence of model structures.  Instead we propose the
following definition, which leads to a theorem with some substance
(Example Theorem~\ref{ethm:rect}).  It is the condition used in
practice in proving comparison and rectification theorems.

\begin{defn}\label{def:dmeq}
Let $\mC$ be a closed model category with countable coproducts and
with a symmetric monoidal product that preserves countable colimits in
each variable.  We say that a map $f\colon \oA\to \oB$ of operads in
$\mC$ is a \term{derived monad equivalence} if the induced map $\bA
Z\to \bB Z$ is a weak equivalence for every cofibrant object $Z$
in~$\mC$.
\end{defn}

Though we have not put enough hypotheses on $\mC$ to ensure it, in
practice countable coproducts of reasonable objects in $\mC$ will
preserve and reflect weak equivalences and then $f$ will be a derived
monad equivalence if and only if each of the maps
\[
\oA(m)\mtimes_{\Sigma_{m}}Z^{(m)}\to \oB(m)\mtimes_{\Sigma_{m}}Z^{(m)}
\]
is a weak equivalence. In our examples of main interest, we have more
intrinsic sufficient conditions for a map of operads to be a derived
monad equivalence.

\begin{example}\label{ex:dmeq1}
In the category of spaces (or more generally, any topological or
simplicial model category), a map of operads $f\colon \oA\to \oB$ that
induces an equivariant homotopy equivalence $\oA(m)\to \oB(m)$ for all
$m$ is a derived monad equivalence.  Indeed, the map $\bA Z\to \bB Z$
is a homotopy equivalence for all $Z$ (and a homotopy equivalence in a
topological or simplicial model category is a weak equivalence).  As a
special case, when $\nsA$ is a non-symmetric operad with $\nsA(m)$
contractible for all $m$, the map of operads $\oA\to \oAss$ is a
derived monad equivalence.
\end{example}

\begin{example}\label{ex:dmeq2}
In the category of symmetric spectra (of spaces or simplicial sets)
with its positive stable model structure or the category of orthogonal
spectra with its positive model structure, a map of operads $f\colon
\oA\to \oB$ that induces a (non-equivariant) weak equivalence
$\oA(n)\to \oB(n)$ is a derived monad equivalence.  This can be proved
by generalizing the argument of~\cite[15.5]{MMSS}
(see~\cite[8.3.(i)]{ChadwickMandell} for slightly more details).  In
the case of EKMM $S$-modules, if $f\colon \oA\to \oB$ is a map of
operads of spaces that is a (non-equivariant) homotopy
equivalence $\oA(n)\to \oB(n)$ for all $n$, then $\Sigma^{\infty}_{+}f$ is a
derived monad equivalence.  This can be proved by generalizing the
argument of~\cite[III.5.1]{EKMM}.
(See~\cite[8.3.(ii)]{ChadwickMandell} for a more general statement.)
In particular, in these categories, the augmentation map $\oE\to
\oCom$ for an $E_{\infty}$ operad (assumed to come from spaces in the
EKMM $S$-module case) is a derived monad equivalence.
\end{example}

\begin{example}\label{ex:dmeq3}
In the context of chain complexes of $R$-modules, a map of operads
$\oA\to \oB$ that is an $R[\Sigma_{n}]$-module chain homotopy
equivalence $\oA(n)\to \oB(n)$ for all $n$ is a derived monad
equivalence.  If the functors
$\oA(n)\otimes_{R[\Sigma_{n}]}(-)$ and
$\oB(n)\otimes_{R[\Sigma_{n}]}(-)$ preserve exactness of (homologically)
bounded-below exact sequences of $R$-free
$R[\Sigma_{n}]$-modules (for all $n$), then a weak equivalence $\oA\to
\oB$ is a derived monad equivalence.  This occurs in particular for
part~(c) of Example Theorem~\ref{ethm:model} when $\oA$ and $\oB$ both
satisfy the stated operad hypotheses.
\end{example}

To go with these examples, we have the following example theorem.

\begin{ethm}\label{ethm:rect}
Let $\mC$ be a symmetric monoidal category and $f\colon \oA\to \oB$
a map of operads in $\mC$, where $\mC$, $\oA$, and $\oB$ fall into one
of the examples of Example Theorem~\ref{ethm:model}.(a)-(c).  If $f$
is a derived monad equivalence then the Quillen adjunction $P_{f}\colon
\mC[\oA]\tofrom \mC[\oB]\noloc U_{f}$  is a Quillen equivalence.
\end{ethm}

Again, as in the previous section, this is an ``example theorem'' in
that it gives an example of the kind of theorem that holds much more
generally with a proof that can also be adapted to work much more
generally.  We outline the proof after the change of categories
theorem below, as the arguments for both are quite similar.

In terms of change of categories, one should expect comparison
theorems of the following form to hold quite generally:
\begin{quotation}
Let 
\[
L\colon \mC\tofrom \mC'\noloc R
\]
be a Quillen equivalence between
monoidal model categories with $L$ strong symmetric monoidal, and let
$\oO$ be an operad in $\mC$.  With some technical hypotheses, the
adjunction 
\[
L\colon \mC[\oO]\tofrom \mC'[L\oO]\noloc R
\]
on 
operadic algebra categories is also a Quillen
equivalence
\end{quotation}
A minimal technical hypothesis is that $L\oO$ be ``the right
thing'' and an easy way to ensure this is to put some kind of
cofibrancy condition on the objects $\oO(n)$.  In our cases of
interest, we could certainly state such a theorem, but it would not
cover the example in modern categories of spectra when $\oO$ is the
suspension spectrum 
functor applied to an operad of spaces; for such an operad, the
spectra $\oO(n)$ will not be cofibrant.  On the other hand, in these
examples the right adjoint preserves all weak equivalences and not
just weak equivalences between fibrant objects; in this setup it
seems more convenient to consider an operad $\oO'$ in $\mC'$ and a map of
operads $\oO\to R\oO'$ (or equivalently, $L\oO\to \oO'$) that induces
a weak equivalence 
\[
\bO Z\to R(\bO' LZ)
\]
for all cofibrant objects $Z$ of $\mC$.  We state such a theorem for 
our examples of interest.

\begin{ethm}\label{ethm:compare}
Let $L\colon \mC\tofrom \mC'\noloc R$ be one of the Quillen
adjunctions of symmetric monoidal categories listed below.  Let $\oA$
be an operad in $\mC$, let $\oB$ be an operad in $\mC'$, and let
$f\colon \oA\to R\oB$ be a map of operads that induces a weak
equivalence 
\[
\bA Z\to R(\bB LZ)
\]
for all cofibrant objects $Z$ of $\mC$.  Then the induced Quillen
adjunction 
\[
P_{L,f}\colon \mC[\oA]\tofrom \mC'[\oB]\noloc U_{R,f}
\]
is a Quillen equivalence.
This theorem holds in particular in the examples:
\begin{enumerate}\deflist
\item $\mC$ is the category of simplicial sets (with the usual
model structure) or the category of
symmetric spectra of simplicial sets, $\mC'$ is the category of spaces
or the category of symmetric spectra in spaces (resp.), and $L,R$ is
the geometric realization, singular simplicial set adjunction.
\item $\mC$ is the category of symmetric spectra, $\mC'$ is the
category of orthogonal spectra and $L,R$ is the prolongation,
restriction adjunction of~\cite[p.~442]{MMSS}.
\item $\mC$ is the category of symmetric spectra or orthogonal
spectra, $\mC'$ is the category of EKMM $S$-modules, and $L,R$ is the
adjunction of \cite{Schwede-EKMM} or \cite[I.1.1]{MMM}.
\end{enumerate}
\end{ethm}

As indicated in the paragraph above the statement, the statement takes
advantage of the fact that in the examples being considered in this
section, the right adjoint preserves
all weak equivalences; a general statement for other examples should
use a fibrant replacement for $\bB LZ$ in place of $\bB LZ$.  The
proof sketch below also takes advantage of this property of the right
adjoint.  In generalizing the argument to the case when
fibrant replacement is required in the statement, the fibrant
replacement of the filtration can be performed in $\mC'$.

The proof of the theorems above uses the homogeneous filtration on a
pushout of the form $A'=A\amalg^{\mC[\oO]}_{\bO X}\bO Y$ studied in
Section~\ref{sec:limit}.  This is the $m=0$ case of the filtration on
the enveloping operad $\oU^{\oO}_{A'}$, and we will need to use the
filtration on the whole operad for an inductive argument even though
we are only interested in the $m=0$ case in the end.  We will use
uniform notation in the sketch proof that follows, taking $\mC'=\mC$
with adjoint functors $L$ and $R$ to be the identity in the case of
Example Theorem~\ref{ethm:rect}.  We use the notation $I$ for the
preferred set of generators for the cofibrations of $\mC$ (as in
Section~\ref{sec:model}).

Because fibrations and weak equivalences in the algebra categories are
created in the underlying symmetric monoidal categories, the
adjunction $P_{L,f}$,$U_{R,f}$ is automatically a Quillen adjunction
(as indicated already in the statements), and we just have to prove
that the unit of the adjunction
\begin{equation}\label{eq:uniteq}
A\to U_{R,f}(P_{L,f}A)
\end{equation}
is a weak equivalence for any cofibrant $\oA$-algebra $A$.  Every
cofibrant $\oA$-algebra is the retract of an $\bA I$-cell
$\oA$-algebra, and so it suffices
to consider the case when $A$ is an $\bA I$-cell $\oA$-algebra; then
$A=\colim A_{n}$ where $A_{0}=\oA(0)$ and each $A_{n+1}$ is formed
from $A_{n}$ by cell attachment (of possibly an infinite coproduct of
cells).  As mentioned parenthetically in Section~\ref{sec:model} and
as we shall see below, the underlying maps $A_{n}\to A_{n+1}$ are nice
enough that $A$ is the homotopy colimit (in $\mC$ or $\mC[\oA]$) of
the system of the finite stages $A_{n}$ (and likewise for $P_{L,f}A$,
which is a cell $\bB LI$-algebra with stages $P_{L,f}A_{n}$).  Thus,
it will be enough to see that~\eqref{eq:uniteq} is a weak equivalence
for each $A_{n}$.  By the hypothesis of the theorem, we know that this
holds for $A_{0}$ (which is the free $\oA$-algebra on the initial
object of $\mC$); moreover, as the enveloping operad of $A_{0}$ is
$\oA$ and the enveloping operad of $P_{L,f}A_{0}$ is $\oB$, we can
assume as an inductive hypothesis that
\[
\bU^{\oA}_{A_{n}}Z\to \bU^{\oB}_{P_{L,f}A_{n}}LZ
\]
is a weak equivalence for all cofibrant $Z$; in other words, we can
assume by induction that the
hypothesis of the theorem holds for the map of enveloping operads
$\oU^{\oA}_{A_{n}}\to R(\oU^{\oB}_{P_{L,f}A_{n}}$).  It then suffices
to prove that the hypothesis of the theorem holds for the map of
enveloping operads $\oU^{\oA}_{A_{n+1}}\to
R(\oU^{\oB}_{P_{L,f}A_{n+1}})$; this is because in the categories
$\mC$ and $\mC'$ of the examples, countable coproducts preserve and
reflect weak equivalences and the unit map $A_{n+1}\to
U_{R,f}(P_{L,f}A_{n+1})$ is the restriction of the map of monads to
the homogeneous degree zero summand (at least in the homotopy category
of $\mC$).

To prove this, let $X\to Y$ be the coproduct of maps in $I$ such that
$A_{n+1}=A_{n}\amalg^{\mC[\oA]}_{\bA X}\bA Y$ and consider the
filtration on $\oU^{\oA}_{A_{n+1}}(m)$ and
$\oU^{\oB}_{P_{L,f}A_{n+1}}(m)$ studied in Section~\ref{sec:limit}.
We note that the induction hypothesis on $A_{n}$ also implies that the
map
\begin{multline*}
\oU^{\oA}_{A_{n}}(m)
\mtimes_{\Sigma_{m_{1}}\times \dotsb \times \Sigma_{m_{i}}}
(Z_{1}^{(m_{1})}\mtimes \dotsb \mtimes Z^{(m_{i})}_{i})
\\\to 
R(\oU^{\oB}_{P_{L,f}A_{n}}(m)
\mtimes_{\Sigma_{m_{1}}\times \dotsb \times \Sigma_{m_{i}}}
(LZ_{1}^{(m_{1})}\mtimes \dotsb \mtimes LZ^{(m_{i})}_{i}))
\end{multline*}
is a weak equivalence for all cofibrant objects $Z_{1},\dotsc,Z_{i}$
(where $m=m_{1}+\dotsb +m_{i}$) as this is a summand of the map
\[
\oU^{\oA}_{A_{n}}(m)
\mtimes_{\Sigma_{m}}
(Z_{1}\amalg \dotsb \amalg  Z_{i})^{(m)}
\to 
R(\oU^{\oB}_{P_{L,f}A_{n}}(m)
\mtimes_{\Sigma_{m}}
L(Z_{1}\amalg \dotsb \amalg  Z_{i})^{(m)}).
\]
Looking at the pushout square~\eqref{eq:Ql} that inductively defines
$Q^{\kay}_{i}(X\sto Y)$, a bit of analysis shows that in our example
categories the maps $Q^{\kay}_{i-1}\to Q^{\kay}_{i}$ are
$\Sigma_{\kay}$-equivariant Hurewicz
cofibrations (or in the simplicial categories, maps that geometrically
realize to such).  It follows that for any cofibrant object $Z$, the
maps
\begin{multline*}
\oU^{\oA}_{A_{n}}(\kay+m)
\mtimes_{\Sigma_{\kay}\times \Sigma_{m}}
(Q^{\kay}_{i-1}(X\sto Y) \mtimes Z^{(m)})\\\to
\oU^{\oA}_{A_{n}}(\kay+m)
\mtimes_{\Sigma_{\kay}\times \Sigma_{m}}
( Q^{\kay}_{i}(X\sto Y) \mtimes Z^{(m)})
\end{multline*}
are (or geometrically realize to) Hurewicz cofibrations (likewise in
$\mC'$) and that the maps
\begin{multline*}
\oU^{\oA}_{A_{n}}(\kay+m)
\mtimes_{\Sigma_{\kay}\times \Sigma_{m}}
(Q^{\kay}_{i}(X\sto Y)\mtimes Z^{(m)})\\\to
R(\oU^{\oB}_{P_{L,f}A_{n}}(\kay+m)
\mtimes_{\Sigma_{\kay}\times \Sigma_{m}}
(Q^{\kay}_{i}(LX\sto LY)\mtimes LZ^{(m)}))
\end{multline*}
are weak equivalences.  Now the pushout square~\eqref{eq:poenv} shows
that for any cofibrant object $Z$, at each filtration level $\kay$,
the map
\[
F^{\kay-1}\oU^{\oA}_{A_{n+1}}(m)\mtimes_{\Sigma_{m}}Z^{(m)}\to 
F^{\kay}\oU^{\oA}_{A_{n+1}}(m)\mtimes_{\Sigma_{m}}Z^{(m)}
\]
is (or geometrically realizes to) a Hurewicz cofibration (likewise
in $\mC'$) and that the maps
\[
F^{\kay}\oU^{\oA}_{A_{n+1}}(m)\mtimes_{\Sigma_{m}}Z^{(m)}
\to
R(F^{\kay}\oU^{\oB}_{P_{L,f}A_{n+1}}(m)\mtimes_{\Sigma_{m}}LZ^{(m)})
\]
are weak equivalences.  The colimit is then weakly equivalent to the
homotopy colimit and we get a weak equivalence 
\[
\oU^{\oA}_{A_{n+1}}(m)\mtimes_{\Sigma_{m}}Z^{(m)}
\to
R(\oU^{\oB}_{P_{L,f}A_{n+1}}(m)\mtimes_{\Sigma_{m}}LZ^{(m)}),
\]
completing the induction and the sketch proof of Example
Theorems~\ref{ethm:rect} and~\ref{ethm:compare}.

The argument above proved the comparison theorems by proving equivalences of
enveloping operads.  Since the unary part of the enveloping operad is
the enveloping algebra, we also get module category comparison
results.  We state this as the following corollary, which says that as
long as the algebras are cofibrant, changing categories by Quillen
equivalences and the algebras by derived monad equivalences results in
Quillen equivalent categories of modules.

\begin{cor}\label{cor:modcomp}
Let $L\colon \mC\tofrom \mC'\noloc R$ be one of the Quillen
adjunctions of symmetric monoidal categories in Example
Theorem~\ref{ethm:compare} or the identity functor adjunction on one
of the categories in Example Theorem~\ref{ethm:rect}.  Let $f\colon
\oA\to R\oB$ be a map of operads that induces a weak equivalence $\bA
Z\to R(\bB LZ)$ for all cofibrant objects $Z$, and let $g\colon A\to
RB$ be a weak equivalence of $\oA$-algebras for an $\oA$-algebra $A$
and a $\oB$-algebra $B$.  If $A$ and $B$ are cofibrant (in $\mC[\oA]$ and
$\mC'[\oB]$, respectively), then $f$ and $g$ induce a Quillen equivalence of
the category of $(\oA,A)$-modules and the category of
$(\oB,B)$-modules.
\end{cor}

\begin{proof}[Sketch proof]
The argument above shows that under the given hypotheses, the map of
$\mtimes$-monoids $U^{\oA}A\to R(U^{\oB}B)$ is a weak
equivalence. The left and right adjoint functors in the Quillen adjunction on
module categories are given by
$U^{\nsB}B\mtimes_{L(U^{\nsA}A)}L(-)$ and $R$, respectively. These
both preserve coproducts, homotopy cofiber sequences, and 
sequential homotopy colimits up to weak equivalence.  It follows that
the unit of the adjunction $X\to R(U^{\nsB}B\mtimes_{L(U^{\nsA}A)}LX)$
is a weak equivalence for every cofibrant $A$-module~$X$.
\end{proof}

The analogous result also holds for modules over algebras on
non-symmetric operads, proved by essentially the same filtration
argument: we have a non-symmetric version $\nU^{\nsO}_{A}(m)$ of
Construction~\ref{cons:eop}.  In this case, the resulting objects do
not assemble into an operad; nevertheless, $\nU^{\nsO}_{A}(1)$
still has the structure of a $\mtimes$-monoid and coincides with the
(non-symmetric) enveloping algebra $\nU^{\nsO}A$.  The
non-symmetric analogue of~\eqref{eq:poenv} holds, and the filtration
argument (under the hypotheses of the previous corollary) proves that
the map $\nU^{\nsA}A\to R(\nU^{\nsB}B)$ is a weak equivalence of
$\mtimes$-monoids.  We conclude that the unit map $X\to
R(\nU^{\nsB}B\mtimes_{L\nU^{\nsA}A}LX)$ is a weak equivalence for every
cofibrant $A$-module $X$.

\section{Enveloping Algebras, Moore Algebras, and Rectification}
\label{sec:rectass}

In the special case of Example~\ref{ex:dmeq1}, Example
Theorem~\ref{ethm:rect} gives an equivalence of the homotopy category
of $A_{\infty}$ algebras (over a given $A_{\infty}$ operad) with the
homotopy category of associative algebras, in
particular constructing an associative algebra rectification of an
$A_{\infty}$ algebra. We know another way to construct an associative algebra
from an $A_{\infty}$ algebra, namely the (non-symmetric) enveloping
algebra.  In the case when the $A_{\infty}$ operad is the operad of little
$1$-cubes $\nsC_{1}$, there is also a classical rectification called
the Moore algebra.  The purpose of this section is to compare these
constructions.

We first consider the rectification of Example Theorem~\ref{ethm:rect} and the
non-symmetric enveloping algebra.  Let $\nsO$ be a non-symmetric operad
and $\epsilon \colon \nsO\to \nsAss$ a weak equivalence.  Under the
hypotheses of Example Theorem~\ref{ethm:rect}, the rectification
(change of operads) functor $P_{\epsilon}$ associated to $\epsilon$
gives a $\mtimes$-monoid $P_{\epsilon}A$ and a map of $\nsO$-algebras
$A\to P_{\epsilon}A$ that is a weak equivalence when $A$ is cofibrant.
As part of the proof of Example Theorem~\ref{ethm:rect}, we get a weak
equivalence of enveloping operads
\[
\oU^{\oO}_{A}\to \oU^{\oAss}_{P_{\epsilon}A}.
\]
As mentioned at the end of the previous section, the non-symmetric
version of this argument also works to give a weak equivalence of
$\mtimes$-monoids
\[
\nU^{\nsO}A\to \nU^{\nsAss}(P_{\epsilon}A).
\]
Moreover, in the case of the associative algebra operad $\nsAss$, we have a natural
isomorphism of $\mtimes$-monoids $\nU^{\nsAss}M\to M$ for any
$\mtimes$-monoid $M$. 
Putting this together, we get the following theorem.

\begin{thm}\label{thm:Uwk}
Let $\mC$ be a symmetric monoidal category and $\oO$ an
$A_{\infty}$ operad that fall into one of the examples of
Theorem~\ref{ethm:model}.(a)-(c).  Write $\epsilon\colon
\nsO\to \nsAss$ for the weak equivalence identifying $\oO$ as an
$A_{\infty}$ operad.  If $A$ is a cofibrant
$\oO$-algebra then the natural maps 
\[
A\to P_{\epsilon}A\iso \nU^{\nsAss}P_{\epsilon}A\from \nU^{\nsO}A
\]
are weak equivalences of $\oO$-algebras.
\end{thm}

We now focus on $A_{\infty}$ algebras for the little $1$-cubes operad
$\nsC_{1}$, where we can describe results both more concretely and in
much greater generality.  For the rest of the section we work in the
context of a symmetric monoidal category enriched over topological
spaces as in Section~\ref{sec:enrich}: let $\mC$ be a closed symmetric
monoidal category with countable colimits, and let $L\colon \Top\to
\mC$ be strong symmetric monoidal left adjoint functor (whose right
adjoint we denote as $R$).  Then as per Theorem~\ref{thm:enrich},
$\mC$ becomes enriched over topological spaces and we have a notion of
homotopies and homotopy equivalences in $\mC$, defined in terms of
mapping spaces or equivalently in terms of tensor with the unit
interval.  We also have $L\nsC_{1}$ as a non-symmetric operad in
$\mC$; for an $L\nsC_{1}$-algebra $A$, we give a concrete construction
of the enveloping algebra $\nU A$, mostly following~\cite[\S2]{Smash}.
We first write the formulas and then explain where they come from.

\begin{cons}\cite[\S2]{Smash}
Let $\bar D$ be the space of subintervals of $[0,1]$ and let
$D$ be the subspace of $\bar D$ of those intervals that do not start at
$0$. We have a canonical isomorphism $\bar D\iso \nsC_{1}(1)$ (sending
a subinterval to the $1$-tuple containing it) that we elide notation for.
Under this isomorphism, the composition law $\Gamma^{1}_{1}$ defines a
pairing $\gamma \colon \bar D\times \bar D\to \bar D$ that satisfies
the formula  
\[
\gamma ([x,y],[x',y'])=[x+(y-x)x',x+(y-x)y'].
\]
We note that $\gamma$ restricts to a pairing $D\times D\to D$ and for
formal reasons $\gamma$ is associative
\begin{multline*}
\gamma (\gamma ([x,y],[x',y']),[x'',y''])=\\
[x+(y-x)x'+(y-x)(y'-x')x'',x+(y-x)x'+(y-x)(y'-x')y'']\\
=\gamma([x,y],\gamma ([x',y'],[x'',y'']))
\end{multline*}
and unital
\[
\gamma ([0,1],[x,y])=[x,y]=\gamma([x,y],[0,1]),
\]
making $\bar D$ a topological monoid and $D$ a sub-semi-group.  Define 
$\alpha \colon D\times D\to \nsC_{1}(2)$ by
\[
\alpha([x,y],[x',y'])=([0,\tfrac{x}{x+(y-x)x'}], [\tfrac{x}{x+(y-x)x'},1]).
\]
Let $DA$ be
the object of $\mC$ defined by the following pushout diagram
\[
\xymatrix@-1pc{%
LD \mtimes \mS \ar[d]\ar[r]&LD \mtimes A\ar[d]\\
L\bar D\mtimes \mS\ar[r]&DA
}
\]
where the top map is induced by the composite of the isomorphism
$\mS\iso L\nsC_{1}(0)$ (from the strong symmetric monoidal structure
on $L$) and the $L\nsC_{1}$-action map $L\nsC_{1}(0)\to A$.  We use
$\gamma$ and $\alpha$ to define a multiplication on $DA$ as follows.
We use the map
\[
(LD\mtimes A)\mtimes (LD\mtimes A)\to LD\mtimes A\to DA
\]
coming from the map
\begin{multline*}
(LD\mtimes A)\mtimes (LD\mtimes A)\iso 
L(D\times D)\mtimes (A\mtimes A)
\to\\
L(D\times \nsC_{1}(2))\mtimes (A\mtimes A)
\iso
LD \mtimes (L\nsC_{1}(2)\mtimes (A\mtimes A))
\to LD\mtimes A
\end{multline*}
induced by the map $(\gamma,\alpha)\colon D\times D\to D\times
\nsC_{1}(2)$ and the $L\nsC_{1}$-action map on $A$.  We note that both
associations
\[
(LD\mtimes A)\mtimes (LD\mtimes A)\mtimes (LD\mtimes A)\to LD\mtimes A
\]
coincide: both factor through the map
\[
(LD\mtimes A)\mtimes (LD\mtimes A)\mtimes (LD\mtimes A)
\iso
L(D\times D\times D)\mtimes A^{(3)}\to L(D\times \nsC_{1}(3))\mtimes A^{(3)}
\]
induced by the map $D\times D\times D\to D\times \nsC_{1}(3)$ given on
the $D$ factor as $\gamma \circ (\gamma \times \id)=\gamma \circ
(1\times \gamma)$ and on the $\nsC_{1}(3)$ factor by the formula
\[
[x,y],[x',y'],[x'',y'']\mapsto ([0,a],[a,b],[b,1])
\]
where 
\[
a=\frac{x}{x+(y-x)(x'+(y'-x')x'')}, \qquad 
b=\frac{x+(y-x)x'}{x+(y-x)(x'+(y'-x')x'')}.
\]
When restricted to maps
\[
(LD\mtimes \mS)\mtimes (LD\mtimes A), 
(LD\mtimes A)\mtimes (LD\mtimes \mS)\to DA,
\]
this map coincides with the map induced by just $\gamma$ and the unit
isomorphism of $\mC$ and so extends to compatible maps
\begin{align*}
(L\bar D \mtimes \mS)\mtimes (L\bar D \mtimes \mS)&\to DA\\
(L\bar D \mtimes \mS)\mtimes (LD \mtimes A)&\to DA\\
(LD \mtimes A)\mtimes (L\bar D \mtimes \mS)&\to DA
\end{align*}
and defines an associative multiplication on $DA$.  The map $\mS\to
DA$ induced by the inclusion of the element $[0,1]$ of $\bar D$ is a
unit for this multiplication.
\end{cons}

To understand the construction, it is useful to think of $D$ as a
subspace of $\nsC_{1}(2)$ rather than a subspace of $\nsC_{1}(1)$ via
the embedding
\[
[x,y]\mapsto ([0,x],[x,y]).
\]
Then we have a map $DA\to \nU A$ sending $L\bar D\mtimes \mS$ and
$LD\mtimes A$ to the 0 and 1 summands 
\[
L\bar D\mtimes \mS\iso L\nsC_{1}\mtimes A^{(0)}\qquad \text{and}\qquad
LD\mtimes A\to L\nsC_{1}(2)\mtimes A
\]
in the coequalizer~\eqref{eq:nUA} for $\nU A$.  We also have a map back that sends
the summand $L\nsC_{1}(n+1)\mtimes A^{(n)}$ (for $n\geq 1$) to
$LD\mtimes A$ by remembering just the last interval and using the rest
to do the multiplication on $A$; specifically, for
$[x_{1},y_{1}],\dotsc,[x_{n+1},y_{n+1}]$, we use the element of
$\nsC_{1}(n)$ corresponding to
\[
[\tfrac{x_{1}}{x_{n+1}},\tfrac{y_{1}}{x_{n+1}}],\dotsc,[\tfrac{x_{n}}{x_{n+1}},\tfrac{y_{n}}{x_{n+1}}]
\]
for the map $A^{(n)}\to A$.  It is straightforward to check that
these give inverse isomorphisms of objects of $\mC$; see
\cite[2.5]{Smash}.  

The isomorphism of the previous paragraph then forces the formula for
the multiplication.  Intuitively speaking, the first box in $D$
(viewed as a subset of $\nsC_{1}(2)$) holds the algebra (from the
tensor) and the second box is a placeholder to plug in the module
variable; the complement $\bar D \setminus D$ corresponds to the first
box having length zero and then only the unit of the algebra can go
there.  For the composition, the right copy gets plugged into the
second box of the left copy to give an element of $\nsC_{1}(3)$ (i.e.,
the operadic composition $\ell\circ_{2}r=\Gamma^{2}_{1,2}(\ell;1,r)$
where $\ell$ is the element of the left copy of $D$ and $r$ is the
element of the right copy of $D$); the first and second boxes are on
the one hand rescaled to an element of $\nsC_{1}(2)$ that does the
multiplication on the copies of $A$ and on the other hand joined to give with the third box the
new element of $D$ (viewed as a subspace of $\nsC_{1}(2)$). The
associativity is straightforward to visualize 
in terms of plugging in boxes when written down on paper. (See
Section~2 of~\cite{Smash}.) In the case when one of the elements
comes from $\bar D \setminus D$, the corresponding copy of $A$ is
restricted to the unit $\mS$ and the first box of zero length also
works like a unit.

Using the isomorphism of $\mtimes$-monoids $\nU A\iso DA$, we now have
the following comparison theorem for the underlying objects of $\nU A$ and $A$.

\begin{prop}\label{prop:UAhty}\cite[1.1]{Smash}
The map of $\nU A$-modules $\nU A\iso \mS\mtimes \nU A\to A$ induced by the map
$\mS\iso L\nsC_{1}(0)\to A$ is a homotopy equivalence of objects of $\mC$.
\end{prop}

\begin{proof}
In concrete terms, the map in the statement is induced by the map
\[
LD\mtimes A\to L\nsC_{1}(1)\mtimes A\to A
\]
for the map $D\to \nsC_{1}(1)$ that sends $[x,y]$ to $([0,x])$, which is
compatible with the map 
\[
L\bar D\mtimes \mS\to \mS\to A.
\]
We can use any element of $D$ to produce a map (in $\mC$) from $A$ to
$\nU A$; a path to the operad identity element $1$ in $\nsC_{1}(1)$
(which corresponds to $[0,1]\subseteq [0,1]$) then induces a homotopy
of the composite map $A\to A$ to the identity map of $A$.  We can
construct a homotopy from the composite to the identity on $\nU A$ using
a homotopy of self-maps of $\nsC_{1}(1)$ from the identity to the
constant map on $1$ (combined with the $\nsC_{1}(1)$ action map on
$A$) and a homotopy of self-maps of the pair $(\bar D,D)$ from the
constant map (on the chosen element of $D$) to the identity map.  For
example, if the chosen element of $D$ corresponds to the subinterval
$[a,b]$ (with $a\neq 0$) then the linear homotopy
\[
[x,y],t \mapsto [xt+a(1-t),yt+b(1-t)]
\]
is such a homotopy of self-maps of the pair.
\end{proof}

In the context of spaces, J.~C.~Moore invented an associative version
of the based loop space by parametrizing loops with arbitrary length
intervals.  This idea extends to the current context to give another
even simpler construction of a $\mtimes$-monoid equivalent (in
$\mC$) to an $L\nsC_{1}$-algebra $A$.

\begin{cons}
Define $MA$ to be the object of $\mC$ defined by the pushout diagram
\[
\xymatrix@-1pc{%
L\bR^{>0} \mtimes \mS \ar[d]\ar[r]&L\bR^{>0} \mtimes A\ar[d]\\
L\bR^{\geq 0}\mtimes \mS\ar[r]&MA
}
\]
(where $\bR^{>0}\subset \bR^{\geq 0}$ are the usual subspaces of
positive and non-negative real numbers, respectively).  We give this
the structure of a $\mtimes$-monoid with the unit $\mS\to MA$ induced
by the inclusion of $0$ in $\bR^{\geq 0}$ and multiplication
$MA\mtimes MA\to MA$ induced by the map
\begin{multline*}
(L\bR^{>0} \mtimes A)\mtimes (L\bR^{>0} \mtimes A)
\iso L(\bR^{>0}\times \bR^{>0})\mtimes (A\mtimes A)\\
\to L(\bR^{>0}\times \nsC_{1}(2))\mtimes (A\mtimes A)
\iso
L\bR^{>0}\mtimes (L\nsC_{1}(2)\mtimes (A\mtimes A))
\to L\bR^{>0}\mtimes A
\end{multline*}
induced by the $\nsC_{1}$-action on $A$ and the map 
\[
c\colon (r,s)\in \bR^{>0}\times \bR^{>0} \mapsto 
(r+s, ([0,\tfrac{r}{r+s}],[\tfrac{r}{r+s},1]))
\in \bR^{>0}\times \nsC_{1}(2).
\]
\end{cons}

The idea is that the element of $\bR^{>0}$ specifies a length (with
the zero length only available for the unit) and the multiplication
uses the proportionality of the two lengths to choose an element of
$\nsC_{1}(2)$ for the multiplication on $A$; the two lengths add to
give the length in the result.  In the case when $\mC$ is the category
of spaces and $A=\Omega X$ is the based loop space of a space $X$,
$MA$ is the Moore loop space.  An element is specified by an element
$r$ of $\bR^{\geq 0}$ together with an element of $\Omega X$ (which
must be the basepoint when $r=0$) but can be visualized as a based
loop parametrized by $[0,r]$ (or for $r=0$ the constant length zero
loop at the basepoint).  The multiplication concatenates loops by
concatenating the parametrizations, an operation that is strictly
associative and unital.

We can compare the $\mtimes$-monoids $MA$ and $\nU A$ through a third $\mtimes$-monoid $NA$
constructed as follows.  Let $N=\bR^{> 0}\times \bR^{>0}\times
\bR^{\geq 0}$, let $\bar N=\bR^{\geq 0}\times \bR^{>0}\times \bR^{\geq
0}$, and define $NA$ by the pushout diagram
\[
\xymatrix@-1pc{%
LN \mtimes \mS \ar[d]\ar[r]&LN \mtimes A\ar[d]\\
L\bar N\mtimes \mS\ar[r]&NA.
}
\]
We have maps $\bar N\times \bar N\to \bar N$ and
$N\times N\to \nsC_{1}(2)$ defined by
\begin{align*}
((r,s,t),(r',s',t'))\in \bar N\times \bar N
&\mapsto (r+sr',ss',st'+t) \in \bar N\\
((r,s,t),(r',s',t'))\in N\times N
&\mapsto c(t,st')=([0,\tfrac{r}{r+sr'}],[\tfrac{r}{r+sr'},1]) \in \nsC_{1}(2),
\end{align*}
which we use to construct the multiplication on $NA$ by the same
scheme as above
\[
(LN\mtimes A)\mtimes (LN\mtimes A)\iso
L(N\times N)\mtimes (A\mtimes A)\to
L(N\times \nsC_{1}(2))\mtimes (A\mtimes A)\to LN\mtimes A.
\]
The unit is the map $\mS\to NA$ induced by the inclusion of $(0,1,0)$
in $\bar N$.

The parametrizing space $N=\{(r,s,t)\}$ generalizes $D$ by allowing
$[r,s]$ to be a subinterval of $[0,r+s+t]$ instead of $[0,1]$, or from
another perspective, generalizes lengths in the definition on the
Moore algebra by incorporating a scaling factor $s$ and padding of
length $t$.
In other words, we have maps
\begin{align*}
[x,y]\in \bar D&\mapsto (x,y-x,1-y)\in \bar N\\
r\in \bR^{\geq 0}&\mapsto (r,1,0)\in \bar N.
\end{align*}
These maps induce maps of $\mtimes$-monoids $\nU A\iso DA\to NA$ and $MA\to NA$,
respectively, and the argument of Proposition~\ref{prop:UAhty} shows
that these maps are homotopy equivalences in $\mC$.  We state this as
the following theorem, repeating the conventions of this part of the
section for easy reference.

\begin{thm}\label{thm:UNM}
Let $\mC$ be a closed symmetric monoidal category admitting countable
colimits and enriched over spaces via a strong symmetric monoidal left
adjoint functor $L$.  Then for algebras over the little $1$-cubes
operad ($L\nsC_{1}$-algebras) the non-symmetric enveloping algebra
$\nU A$ and the Moore algebra $MA$ fit in a natural zigzag of $\mtimes$-monoids 
\[
\nU A\to NA\from MA
\]
where the maps are homotopy equivalences in $\mC$.  Moreover, the
canonical maps $\nU A\to A$ and $MA\to A$ are homotopy equivalences in
$\mC$. 
\end{thm}

To compare $MA$ and $A$ as $A_{\infty}$ algebras, we use a new
$A_{\infty}$ operad $\nsCl$ defined as follows.  

\begin{cons}
Let $\nsCl(0)=\bR^{\geq 0}$ and for $m>0$, let $\nsCl(m)$
be the set of ordered pairs $(S,r)$ with $r$ a positive real number
and $S$ a list of $m$ almost non-overlapping closed subintervals of
$[0,r]$ in their natural order, topologized analogously as in the
definition of $\nsC_{1}$ (as a semilinear submanifold of
$\bR^{2m+1}$).  The operadic composition is defined by scaling and
replacement of the subintervals: the basic composition
\begin{multline*}
\Gamma^{1}_{j}((([x,y]),r),(([x'_{1},y'_{1}],\dotsc,[x'_{j},y'_{j}]),r'))
=\\(([x+ax'_{1},x+ay'_{1}],\dotsc,[x+ax'_{j},x+ay'_{j}]),r+a(r'-1))
\end{multline*}
(with $a:=y-x$)
scales the interval $[0,r']$ to length $ar'$ and inserts that in
place of $[x,y]\subset [0,r]$; the resulting final interval then has
size $r-a+ar'$.  The general composition
$\Gamma^{m}_{j_{1},\dotsc,j_{m}}$ does this operation on each of the
$m$ subintervals: 
\begin{multline*}
\Gamma^{m}_{j_{1},\dotsc,j_{m}}\colon 
(([x^{0}_{1},y^{0}_{1}],\dotsc,[x^{0}_{m},y^{0}_{m}]),r_{1}),\\
(([x^{1}_{1},y^{1}_{1}],\dotsc,[x^{1}_{j_{1}},y^{1}_{j_{1}}]),r_{1}),
\dotsc, 
(([x^{m}_{1},y^{m}_{1}],\dotsc,[x^{m}_{j_{m}},y^{m}_{j_{m}}]),r_{m}),\\
\mapsto\\
(([x^{0}_{1}+a_{1}x^{1}_{1},x^{0}_{1}+a_{1}y^{1}_{1}],\dotsc, 
[s_{m-1}+x^{0}_{m}+a_{m}x^{m}_{j_{m}},
s_{m-1}+x^{0}_{m}+a_{m}y^{m}_{j_{m}}]),
r_{0}+s_{m})
\end{multline*}
where $a_{i}:=y^{0}_{i}-x^{0}_{i}$ and $s_{i}=a_{1}(r_{1}-1)+\dotsb
+a_{i}(r_{i}-1)$.  In the case when one of the $j_{i}$ is zero, that
$j_{i}$ contributes no subintervals but still scales the original
subinterval $[x^{0}_{i},y^{0}_{i}]$ to length $a_{i}r_{i}$ (or removes
it when $r_{i}=0$).  The operad identity element is the element
$(([0,1]),1)\in \nsCl(1)$. 
\end{cons}

The maps $\nsC_{1}(m)\to \nsCl(m)$ that include $\nsC_{1}(m)$ as
the length $1$ subspace assemble to a map of operads $i\colon
\nsC_{1}\to \nsCl$.  We also have a map of operads $j\colon
\nsAss\to \nsCl$ induced by sending the unique element of
$\nsAss(m)$ to the element
\[
(([0,1],[1,2],\dotsc,[m-1,m]),m)
\]
of $\nsCl(m)$.  Using the map $j$, an $L\nsCl$-algebra has
the underlying structure of a $\mtimes$-monoid.  A straightforward
check of universal properties proves the following proposition.

\begin{prop}
The functor that takes a $\nsC_{1}$-algebra $A$ to its Moore algebra
$MA$ is naturally isomorphic to the functor that takes $A$ to the
underlying $\mtimes$-monoid of the pushforward $P_{Li}A$ for the map of
operads $Li\colon L\nsC_{1}\to L\nsCl$.
\end{prop}

The $\nsCl$-action map $L\nsCl(m)\mtimes (MA)^{(m)}\to MA$ is induced
by the map
\[
\nsCl(m)\times (\bR^{>0})^{n}\to \nsCl(m)\times \nsCl(1)^{n}
\overto{\Gamma^{m}_{1,\dotsc,1}} \nsCl(m)\iso \bR^{>0}\times \nsC_{1}(m)
\]
that includes $\bR^{>0}$ in $\nsCl(1)$ by $r\mapsto (([0,r]),r)$,
where the isomorphism is the map that takes an element
$(([x_{1},y_{1}],\dotsc,[x_{m},y_{m}]),r)$ of $\nsCl(m)$ to the
element $(r,([x_{1}/r,y_{1}/r],\dotsc,[x_{m}/r,y_{m}/r]))$ of
$\bR^{>0}\times \nsC_{1}(m)$.

The map of $\nsC_{1}$-algebras that is the unit of the change of
operads adjunction $A\to P_{Li}A$ is induced by the inclusion of $1$
in $\bR^{>0}$ and is a homotopy equivalence by a (simpler) version of
the homotopy argument of Proposition~\ref{prop:UAhty}.  We do not see
how to do a similar argument for the pushforward $P_{Lj}$ from $\mtimes$-monoids
to $\nsCl$-algebras, so we do not have a direct comparison of
$\nsC_{1}$-algebras between $A$ (or $P_{Li}A$) and $MA$ with the
$\nsC_{1}$-algebra structure inherited from its $\mtimes$-monoid structure
without some kind of rectification result (such as Example
Theorem~\ref{ethm:rect}) comparing the category of $L\nsCl$-algebras
with the category of $\nsAss$-algebras.

The argument in~\cite[2.5]{Smash} that identifies $\nU^{\nsC_{1}}A$ as
$DA$ generalizes to identify $\nU^{\nsCl}P_{Li}A$ as $NA$; the maps in
Theorem~\ref{thm:UNM} can then be viewed as the natural maps on
enveloping algebras induced by maps of operads and maps of algebras.

\section{\texorpdfstring{$E_{n}$}{En} spaces and Iterated Loop Space Theory}\label{sec:loop}

The recognition principle for iterated loop spaces provided the first
application for operads.  Although the summary here has been spiced up
with model category notions and terminology (in the adjoint functor
formulation of~\cite[\S8]{May-What}), the mathematics has not 
changed significantly from the original treatment by
May in~\cite{May-GILS}, except for the improvements noted in the
appendix to~\cite{CohenLadaMay}, which extend the results from
connected to grouplike $E_{n}$ spaces. ($E_{n}$ spaces $=$
$E_{n}$ algebras in spaces.)

The original idea for the little $n$-cubes operads $\oC_{n}$ and the
start of the relationship between $E_{n}$ spaces and $n$-fold
loop spaces is the Boardman-Vogt observation  that
every $n$-fold loop spaces comes with the natural structure of a
$\oC_{n}$-algebra.  The action map
\[
\oC_{n}(m)\times \Omega^{n}X \times \dotsb \times \Omega^{n}X\to
\Omega^{n} X
\]
is defined as follows.  We view $S^{n}$ as $[0,1]^{n}/\partial$. Given
an element $c\in \oC_{n}(m)$, and elements $f_{1},\dotsc,f_{m}\colon
S^{n}\to X$ of $\Omega^{n}X$, let $f_{c;f_{1},\dotsc,f_{n}}\colon
S^{n}\to X$ be the function that sends a point $x$ in $S^{n}$ to the
base point if $x$ is not in one of the embedded cubes; the $i$th
embedded cube gets sent to $X$ using the inverse of the embedding and the
quotient map $[0,1]^{n}\to S^{n}$ followed by the map $f_{i}\colon
S^{n}\to X$.  This is a continuous 
based map $S^{n}\to X$ since the boundary of each embedded cube gets
sent to the base point.  Phrased another way, $c$
defines a based map 
\[
S^{n}\to S^{n}\vee \dotsb \vee S^{n}
\]
with the $i$th embedded cube mapping to the $i$th wedge summand of
$S^{n}$ by collapsing all points not in an open cube to the base point
and rescaling; we then apply $f_{i}\colon S^{n}\to X$ to the $i$th
summand to get a composite map $S^{n}\to X$.  

The construction of the previous paragraph factors $\Omega^{n}$ as a
functor from based spaces to $\oC_{n}$-spaces ($=$ $\oC_{n}$-algebras
in spaces).  It is clear that not every $\oC_{n}$-space arises as
$\Omega^{n}X$ because $\pi_{0}\Omega^{n}X$ is a group (for its
canonical multiplication), whereas for the free $\oC_{n}$-space
$\bC_{n}X$,  $\pi_{0}\bC_{n}X$ is
not a group unless $X$ is the empty set; for example,
$\pi_{0}\bC_{n}X\iso \bN$ when $X$ is path connected.  We say that a
$\oC_{n}$-space $A$ is \term{grouplike} when $\pi_{0}A$ is a group
(for its canonical multiplication).  The following is the fundamental
theorem of iterated loop space theory; it gives an equivalence of homotopy
theories between $n$-fold loop spaces and grouplike $\oC_{n}$-spaces.

\begin{thm}[{May~\cite{May-GILS}, Boardman-Vogt~\cite[\S6]{BV-PROP}}]\label{thm:FTILST}
The functor $\Omega^{n}$ from based\break spaces to $\oC_{n}$-spaces is a
Quillen right adjoint.  The unit of the derived adjunction
\[
A\to \Omega^{n}B^{n}A
\]
is an isomorphism in the homotopy category of $\oC_{n}$-spaces if (and
only if) $A$ is grouplike.  The counit of the derived adjunction
\[
B^{n}\Omega^{n}X\to X
\]
is an isomorphism in the homotopy category of spaces if (and only if)
$X$ is $(n-1)$-connected; in general it is an $(n-1)$-connected cover.
\end{thm}

We have written the derived functor of the left adjoint in
Theorem~\ref{thm:FTILST} as $B^{n}$, suggesting an iterated bar
construction.  Although neither the point-set adjoint functor nor the
model for its derived functor used in the argument of
Theorem~\ref{thm:FTILST} is constructed iteratively,
Dunn~\cite{Dunn-Uniqueness} shows that the derived functor is
naturally equivalent to an iterated bar construction.

As a consequence of the statement of the theorem, the unit
of the derived adjunction $A\to \Omega^{n}B^{n}A$ is the initial map
in the homotopy category of $\oC_{n}$-spaces from $A$ to a grouplike
$\oC_{n}$-space and so deserves to be called ``group completion''.
Group completion has various characterizations and for the purposes of
sketching the ideas behind the proof of the theorem, it works best
to choose one of them as the definition and state the property of the
unit map as a theorem.  One such characterization uses the classifying
space construction, which we understand as the Eilenberg-Mac\ Lane bar
construction (after converting the underlying $\oC_{1}$-spaces to
topological monoids) or the Stasheff bar construction (choosing
compatible maps from the Stasheff associahedra into the spaces
$\oC_{n}(m)$). 

\begin{defn}\label{def:groupcompletion}
A map $f\colon A\to G$ of $\oC_{n}$-spaces is a \term{group completion} if $G$
is grouplike and $f$ induces a weak equivalence of classifying spaces.
\end{defn}

In the case $n>1$ (and under some hypotheses if $n=1$),
Quillen~\cite{Quillen-GroupCompletion} gives a homological criterion
for a map to be group completion: if $G$ is grouplike, then a map
$A\to G$ of $\oC_{n}$-spaces is group completion if and only if
\[
H_{*}(A)[(\pi_{0}A)^{-1}]\to H_{*}(G)
\]
is an isomorphism. Counterexamples exist in the case $n=1$ (indeed,
McDuff~\cite{McDuff-DiscreteMonoids} gives a counterexample for every
loop space homotopy type), but recent work of Braun, Chuang, and
Lazarev~\cite{BCL-DerivedLocalization} give an analogous derived
category criterion in terms of derived localization at the
multiplicative set $\pi_{0}A$.  Using
Definition~\ref{def:groupcompletion} or any equivalent independent
characterization of group completion, we have the following addendum
to Theorem~\ref{thm:FTILST}.

\begin{add}
The unit of the derived adjunction in Theorem~\ref{thm:FTILST} is
group completion.
\end{add}

The homotopical heart of the proof of Theorem~~\ref{thm:FTILST} is the
May-Cohen-Segal Approximation
(Theorem~\cite[\S6--7]{May-GILS}, \cite{Cohen-Bulletin},
\cite{Segal-ConfigurationSpaces}),
which we now review.  This theorem studies a version of the
free $\oC_{n}$-algebra functor $\tbC_{n}$ whose domain is the category
of based spaces, where the base point becomes the identity element in
the $\oC_{n}$-algebra structure.  This version of the free functor has
the advantage that for a connected space $X$, $\tbC X$ is also a connected space;
May's Approximation Theorem identifies $\tbC X$ in this case as a model for
$\Omega^{n}\Sigma^{n}X$.  Cohen (following conjectures of May) and
Segal (working independently) then extended this to non-connected
spaces: the group completion of $\tbC X$ is a model for
$\Omega^{n}\Sigma^{n}X$. 

For a based space $X$, $\tbC_{n}X$ is formed as a quotient of
\[
\bC X= \coprod \oC_{n}(m)\times_{\Sigma_{m}}X^{m}
\]
by the equivalence relation that identifies
$(c,(x_{1},\dotsc,x_{i},*,\dotsc,*))\in \oC_{n}(m)\times X^{m}$ with
$(c',(x_{1},\dotsc,x_{i}))\in \oC_{n}(i)\times X^{i}$ for
$c'=\Gamma(c;1,\dotsc,1,0,\dotsc,0)$ where $1$ denotes the identity
element in $\oC_{n}(1)$ and $0$ denotes the unique element in
$\oC_{n}(0)$.  This is actually an instance of the operad
pushforward construction: let $\oIdp$ be the operad with
$\oIdp(0)=\oIdp(1)=*$ and $\oIdp(m)=\emptyset$ for $m>1$.  The functor
associated to $\oIdp$ is the functor $(-)_{+}$ that adds a disjoint
base point with the monad structure $((-)_{+})_{+}\to (-)_{+}$ that
identifies the two disjoint base points; the category of algebras for
this monad is the category of based spaces.  The functor $\tbC_{n}$
from based spaces to $\oC_{n}$-algebras is the pushforward $P_{f}$ for
$f$ the unique map of operads $\oIdp\to \oC_{n}$: formally $P_{f}$ is the
coequalizer described in Section~\ref{sec:compare}, that in this case
takes the form
\[
\xymatrix@-.75pc{%
\bC_{n}(X_{+})\ar@<.5ex>[r]\ar@<-.5ex>[r]&\bC_{n}X\ar[r]&\tbC_{n}X.
}
\]
As mentioned in an aside in that section (or as can be seen concretely
here using the operad multiplication on $\oC_{n}$ directly), the
endofunctor $\tbC_{n}$ on based spaces (i.e., $U_{f}P_{f}$) has the
structure of a monad, and we can identify the category of
$\oC_{n}$-spaces as the category of algebras over the
monad~$\tbC_{n}$.

The factorization of the functor $\Omega^{n}$ through
$\oC_{n}$-spaces has the formal consequence of producing a map of
monads (in based spaces)
\[
\tbC_{n}\to \Omega^{n}\Sigma^{n}.
\]
Formally the map is induced by the composite 
\[
\tbC_{n}X\overto{\tbC_{n}\eta}\tbC_{n}\Omega^{n}\Sigma^{n}X\overto{\xi} \Omega^{n}\Sigma^{n}X,
\]
where $\eta$ is the unit of the $\Sigma^{n},\Omega^{n}$-adjunction and
$\xi$ is the $\oC_{n}$-action map.  This map has the following
concrete description: an element $(c,(x_{1},\dotsc,x_{m}))\in
\oC_{n}(m)\times X^{m}$ maps to the element $\gamma \colon S^{n}\to
\Sigma^{n}X$ of $\Omega^{n}\Sigma^{n}X$ given by the composite of the
map 
\[
S^{n}\to S^{n}\vee\dotsb \vee S^{n}
\]
associated to $c$ (as described above) and the map 
\[
S^{n}\iso \Sigma^{n} \{x_{i}\}_{+}\subset \Sigma^{n}X
\]
on the $i$th factor of $S^{n}$.  Either using this concrete
description, or following diagrams in a formal categorical argument,
it is straightforward to check that this defines a map of monads. 
We can now state the May-Cohen-Segal Approximation Theorem.

\begin{thm}[{May-Cohen-Segal Approximation Theorem~\cite[6.1]{May-GILS},
\cite[3.3]{Cohen-Bulletin}, \cite[Theorem~2]{Segal-ConfigurationSpaces}}]\label{thm:MA}
\noindent \par\noindent 
For any non-degenerately based space $X$, the map of $\oC_{n}$-spaces
$\tbC_{n}X\to \Omega^{n}\Sigma^{n}X$ is group completion.
\end{thm}

(``Non-degenerately based'' means that the inclusion of the base point is
a cofibration.  Both $\tbC_{n}$ and $\Omega^{n}\Sigma^{n}$ preserve
weak equivalences in non-degenerately based spaces, but for other
spaces, either or both may have the wrong weak homotopy type.)

From here a sketch of the proof of Theorem~\ref{thm:FTILST} goes as
follows.  Since $\Omega^{n}$ as a functor from based spaces to based
spaces has left adjoint $\Sigma^{n}$, a check of universal properties
shows that the functor from $\oC_{n}$-spaces to based spaces defined
by the coequalizer
\[
\xymatrix@C-1pc{%
\Sigma^{n}\tbC_{n}A\ar@<-.5ex>[r]\ar@<.5ex>[r]
&\Sigma^{n}A\ar[r]&\Ln A
}
\]
is the left adjoint to $\Omega^{n}$ viewed as a functor from based spaces to
$\oC_{n}$-spaces. (In the coequalizer, one map is induced by the $\oC_{n}$-action map
on $A$ and the other is adjoint to the map of monads $\tbC\to
\Omega^{n}\Sigma^{n}$.)  Because $\Omega^{n}$ preserves fibrations and
weak equivalences, this is a Quillen adjunction.

The main tool to study the $\Ln(-),\Omega^{n}$-adjunction is the
two-sided monadic bar construction, invented in~\cite[\S9]{May-GILS} for
this purpose.  Given a monad $\bT$ and a right action of $\bT$ on a
functor $F$  (say, to based spaces), the two-sided monadic bar
construction is the functor on $\bT$-algebras $B(F,\bT,-)$ defined as
the geometric realization of the simplicial object 
\[
B_{m}(F,\bT,A)=F\underbrace{\bT\dotsb \bT}_{m}A
\]
with face maps induced by the action map $F\bT\to F$, the
multiplication map $\bT\bT\to\bT$ and the action map $\bT A\to A$, and
degeneracy maps induced by the unit map $\Id\to \bT$.  In the case
when $F=\bT$, the simplicial object $B\subdot(\bT,\bT,A)$ has an extra
degeneracy and the map from $B\subdot(\bT,\bT,A)$ to the constant
simplicial object on $A$ is a simplicial homotopy equivalence (in the
underlying category for $\bT$, though not generally in the category of
$\bT$-algebras). 

Because geometric realization commutes with colimits and finite
cartesian products, we have a canonical isomorphism
\[
\tbC_{n}B(\tbC_{n},\tbC_{n},A)\to 
B(\tbC_{n}\tbC_{n},\tbC_{n},A)
\]
and the multiplication map $\tbC_{n}\tbC_{n}\to \tbC_{n}$ then gives
$B(\tbC_{n},\tbC_{n},A)$ the natural structure of a $\oC_{n}$-algebra.
(See Section~\ref{sec:enrich} for a more general discussion.)
For the same reason, the canonical map
\[
\Ln B(\tbC_{n},\tbC_{n},A)\to B(\Ln\tbC_{n},\tbC_{n},A)=B(\Sigma^{n},\tbC_{n},A)
\]
is an isomorphism.  The latter functor clearly\footnote{At the time
when May wrote the argument, this was far from clear: some of the
first observations about when geometric realization of simplicial
spaces preserves levelwise weak equivalences were developed
in~\cite[\S11]{May-GILS} precisely for this argument.} preserves weak
equivalences of $\oC_{n}$-spaces $A$ whose underlying based spaces are
non-degenerately based.  (In addition to being a hypothesis of
May-Cohen-Segal Approximation Theorem, non-degenerately based here also
ensures that the inclusion of the degenerate subspace (or latching
object) is a cofibration.) As a consequence of
Theorem~\ref{thm:georeal} it follows that when the underlying based
space of $A$ is cofibrant (which is the case in particular when $A$ is
cofibrant as a $\oC_{n}$-space), then $B(\tbC_{n},\tbC_{n},A)$ is a
cofibrant $\oC_{n}$-space.  Because $\Ln(-)$ is a Quillen left
adjoint, it preserves weak equivalences between cofibrant objects, and
looking at a cofibrant approximation $A'\overto{\sim}A$, we see from
the weak equivalences
\[
B(\Sigma^{n},\tbC_{n},A)\overfrom{\sim}B(\Sigma^{n},\tbC_{n},A')\iso 
\Ln B(\tbC_{n},\tbC_{n},A')\overto{\sim}
\Ln A'
\]
that $B(\Sigma^{n},\tbC_{n},A)$ models the derived functor $B^{n}A$ of
$\Ln(-)$ whenever $A$ is non-degenerately based.

To complete the argument, we need the theorem of~\cite[\S12]{May-GILS}
that $\Omega^{n}$ commutes up to weak equivalence with geometric
realization of (proper) simplicial spaces that are $(n-1)$-connected
in each level.  Then for $A$ non-degenerately based, we have that the
vertical maps are weak equivalences of $\oC_{n}$-spaces
\[
\xymatrix{%
B(\tbC_{n},\tbC_{n},A)\ar[r]\ar[d]
&B(\Omega^{n}\Sigma^{n},\tbC_{n},A)\ar[d]\\
A&\Omega^{n}B(\Sigma^{n},\tbC_{n},A)
}
\]
while by the May-Cohen-Segal Approximation Theorem, the horizontal map is
group completion.  This proves that the unit of the derived adjunction
is group completion.

For the counit of the derived adjunction, we have from the
model above that $B^{n}$ is always $(n-1)$-connected and the unit
\[
\Omega^{n}X\to \Omega^{n}B^{n}\Omega^{n}X
\]
on $\Omega^{n}X$ is a weak equivalence.  Looking at $\Omega^{n}$ of the
counit,
\[
\Omega^{n}B^{n}\Omega^{n}X\to \Omega^{n}X,
\]
the composite with the unit is the identity on $\Omega^{n}X$, and so
it follows that $\Omega^{n}$ of the counit is a weak equivalence.
Thus, the counit of the derived adjunction is an $(n-1)$-connected
cover map.

\section{\texorpdfstring{$E_{\infty}$}{Einfty} Algebras in Rational and \texorpdfstring{$p$}{p}-Adic Homotopy Theory}\label{sec:cochains}

In the 1960's and 1970's, Quillen~\cite{Quillen-RHT} and
Sullivan~\cite{Sullivan73,Sullivan-ICT} showed that the rational
homotopy theory of simply connected spaces (or simplicial sets) has an
algebraic model in terms of rational differential graded commutative
algebras or coalgebras.  In the 1990's, the author proved a mostly
analogous theorem relating $E_{\infty}$ differential graded algebras
and $p$-adic homotopy theory and a bit later some results for using
$E_{\infty}$ differential graded algebras or $E_{\infty}$ ring spectra
to identify integral homotopy types.  In this section, we summarize
this theory following mostly the memoir of
Bousfield-Gugenheim~\cite{BG}, and the
papers~\cite{Einfty}\footnote{In the published version, in addition to
several other unauthorized changes, the copy editors changed the
typefaces with the result that the same symbols are used for multiple
different objects or concepts; the preprint version available at the
author's home page
\url{https://pages.iu.edu/~mmandell/papers/einffinal.pdf} does not
have these changes and should be much more readable.}
and~\cite{Integral}.  In what follows $k$ denotes a commutative ring,
which is often further restricted to be a field.

In both the rational commutative differential graded algebra case and
the $E_{\infty}$ $k$-algebra case, the theory simplifies by working
with simplicial sets instead of spaces, and the functor is some
variant of the cochain complex.  Sullivan's approach to rational
homotopy theory uses a rational version of the de Rham complex,
originally due to Thom (unpublished), consisting of forms that are
polynomial on simplices and piecewise matched on faces: 

\begin{defn}
The algebra 
$\nabla^{*}[n]$ of polynomial forms on the standard simplex $\Delta[n]$ is the rational
commutative differential graded algebra free on generators
$t_{0},\dotsc,t_{n}$ (of degree zero), $dt_{0},\dotsc,dt_{n}$ (of
degree one) subject to the relations $t_{0}+\dotsb+t_{n}=1$ and
$dt_{0}+\dotsb+dt_{n}=0$ (as well as the differential relation
implicit in the notation). 
\end{defn}

Viewing $t_{0},\dotsc,t_{n}$ as the
barycentric coordinate functions on $\Delta[n]$ determines their
behavior under face and degeneracy maps, making $\nabla^{*}[\bullet]$
a simplicial rational commutative differential graded algebra.  

\begin{defn}
For a
simplicial set $X$, the rational de Rham complex $A^{*}(X)$ is the
rational graded commutative algebra of maps of simplicial sets from
$X$ to $\nabla^{*}[\bullet]$, or equivalently, the end over the
simplex category
\[
A^{*}(X):=\DDelta^{\op}\Set(X,\nabla^{*}[\bullet])=
\int_{\DDelta^{\op}}\Set(X_{n},\nabla^{*}[n])=
\int_{\DDelta^{\op}} \,\prod_{X_{n}}\nabla^{*}[n]
\]
(the last formula indicating how to regard $A^{*}(X)$ as a rational
commutative differential graded algebra).  
\end{defn}

More concretely,
$A^{*}(X)$ is the rational commutative differential graded algebra where an
element of degree $q$ consists of a choice of element of
$\nabla^{q}[n]$ for each non-degenerate $n$-simplex of $X$ (for all
$n$) which agree under restriction by face maps, with multiplication
and differential done on each simplex. (When $X$ is a finite
simplicial complex $A^{*}(X)$ also has a Stanley-Reisner ring style
description; see \cite[G.{i})]{Sullivan73}.)  The
simplicial differential graded $\bQ$-module $\nabla^{q}[n]$ is a
contractible Kan complex for each fixed $q$ (``the
extension lemma''~\cite[1.1]{BG}) and is acylic in the sense that the inclusion of
the unit $\bQ\to \nabla^{*}[n]$ is a chain homotopy equivalence for
each fixed $n$ (``the Poincar\'e lemma''~\cite[1.3]{BG}).  These
formal properties imply that the cohomology of $A^{*}(X)$ is
canonically naturally isomorphic to $H^{*}(X;\bQ)$, the rational
cohomology of $X$ (even uniquely naturally isomorphic, relative to the
canonical isomorphism $\bQ\iso A^{*}(\Delta[0])$).  The canonical
isomorphism can be realized as a chain map to the normalized cochain
complex $C^{*}(X;\bQ)$ defined in terms of integrating differential
forms; see~\cite[1.4,2.1,2.2]{BG}.

In the $p$-adic case, we can use the normalized cochain complex
$C^{*}(X;k)$ directly as it is naturally an $E_{\infty}$ $k$-algebra.
In the discussion below, we use the
$E_{\infty}$ $k$-algebra structure constructed by
Berger-Fresse~\cite[\S2.2]{BergerFresse-Combinatorial} for the
Barratt-Eccles operad $\oE$ (the normalized chains of the Barratt-Eccles
operad of categories or simplicial sets described in
Example~\ref{ex:be}).  Hinich-Schechtmann~\cite{HinichSchechtman}
and (independently) Smirnov~\cite{Smirnov} appear to be the first to
explicitly describe a natural operadic algebra structure on cochains;
McClure-Smith~\cite{McClureSmith-MultivariableCubes} describes a
natural $E_{\infty}$ structure that generalizes classical $\cup_{i}$
product and bracket operations.  The ``cochain theory'' theory
of~\cite{Cochains} shows that all these structures are equivalent in
the sense that they give naturally quasi-isomorphic functors into a
common category of $E_{\infty}$ $k$-algebras, as does the polynomial
de Rham complex functor $A^{*}$ when $k=\bQ$.

Both $A^{*}(X)$ and $C^{*}(X;k)$ fit into adjunctions of the
contravariant type that send colimits to limits.  Concretely, for a
rational commutative differential graded algebra $A$ and an
$E_{\infty}$ $k$-algebra $E$, define simplicial sets by the formulas 
\[
T(A):=\aC_{\bQ}(A,\nabla^{*}[\bullet]), \qquad
U(E):=\aE_{k}(E,C^{*}(\Delta[\bullet])),
\] 
where $\aC_{\bQ}$ denotes the category of rational commutative
differential graded algebras and $\aE_{k}$ denotes the category of
$E_{\infty}$ $k$-algebras (over the Barratt-Eccles operad). An easy
formal argument shows that 
\[
A^{*}\colon \DDelta^{\op}\Set \tofrom \aC_{\bQ}^{\op}\noloc T,
\qquad
C^{*}\colon \DDelta^{\op}\Set \tofrom \aE_{k}^{\op}\noloc U,
\]
are adjunctions.  As discussed in Section~\ref{sec:model}, both
$\aC_{\bQ}$ and $\aE_{k}$ have closed model structures with weak
equivalences the quasi-isomorphisms and fibrations the surjections.
Because both $A^{*}$ and $C^{*}$ preserve homology isomorphisms and
convert injections to surjections, these are Quillen adjunctions.  The
main theorems of \cite{BG} and \cite{Einfty} then identify
subcategories of the homotopy categories on which the adjunction
restricts to an equivalence.

Before stating the theorems, first recall the $H_{*}(-;k)$-local model
structure on simplicial sets: this has cofibrations the inclusions and
weak equivalences the $H_{*}(-;k)$ homology isomorphisms.  When $k$ is
a field, the weak equivalences depend only on the characteristic, and
we also call this the \term{rational model structure} (in the case of
characteristic zero) or the \term{$p$-adic model structure} (in the
case of characteristic $p>0$); we call the associated homotopy
categories, the \term{rational homotopy category} and \term{$p$-adic
homotopy category}, respectively.  As with any localization, the
local homotopy category is the homotopy category of local objects
(that is to say, the fibrant objects): in the case of 
rational homotopy theory, the local objects are the Kan complexes of
the homotopy type of rational spaces.  In 
$p$-adic homotopy theory, the local objects are the Kan complexes that
satisfy a $p$-completeness property described explicitly
in~\cite[\S5,7--8]{Bousfield-LocSpace}.

We say that a simplicial set $X$ is
\term{finite $H_{*}(-;k)$-type} (or \term{finite rational type} when
$k$ is a field of characteristic zero or \term{finite $p$-type} when
$k$ is a field of characteristic $p>0$) when $H_{*}(X;k)$ is finitely
generated over $k$ in each degree (or, equivalently if $k$ is a field,
when $H^{*}(X;k)$ is finite dimensional in each degree).  Similarly a
rational commutative differential graded algebra or $E_{\infty}$
$k$-algebra $A$ is \term{finite type} when its homology is finitely
generated over $k$ in each degree.  It is \term{simply connected} when
the inclusion of the unit induces an isomorphism $k\to H^{0}(A)$,
$H^{1}(A)\iso 0$, and $H^{n}(A)\iso 0$ for $n<0$ (with the usual
cohomological grading convention that $H^{n}(A):=H_{-n}(A)$).  With
this terminology, the main theorem of \cite{BG} is the following:

\begin{thm}[{\cite[Section~8, Theorem~9.4]{BG}}]\label{thm:rht}
The polynomial de Rham complex functor, $A^{*}\colon \DDelta^{\op}\Set
\to \aC_{\bQ}^{\op}$, is a left Quillen adjoint for the rational model
structure on simplicial sets.  The left derived functor restricts to
an equivalence of the full subcategory of the rational homotopy
category consisting of the simply connected simplicial sets of finite
rational type and the full subcategory of the homotopy category of
rational commutative differential graded algebras consisting of the
simply connected rational commutative differential graded algebras of
finite type.
\end{thm}

For the $p$-adic version below, we need to take into account Steenrod
operations.  For $k=\bF_{p}$, the Steenrod operations arise from the
coherent homotopy commutativity of the $p$-fold multiplication, which
is precisely encoded in the action of the $E_{\infty}$ operad.  Specifically, the
$p$th complex $\oE(p)$ of the operad is a $k[\Sigma_{p}]$-free resolution of
$k$, and by neglect of structure, we can regard it as a $k[C_{p}]$-free resolution of
$k$ where $C_{p}$ denotes the cyclic group of order $p$.  The operad
action induces a map
\[
\oE(p)\otimes_{k[C_{p}]}(C^{*}(X;k))^{(p)}\to 
\oE(p)\otimes_{k[\Sigma _{p}]}(C^{*}(X;k))^{(p)}\to C^{*}(X;k).
\]
The homology of $\oE(p)\otimes_{k[C_{p}]}(C^{*}(X;k))^{(p)}$ is a
functor of the homology of $C^{*}(X;k)$ and the Steenrod operations
$P^{s}$ are precisely the image of certain classes under this map;
see, for example, \cite[2.2]{May-Steenrod}.  This process works for
any $E_{\infty}$ $k$-algebra, not just the cochains on spaces, to
give natural operations on the homology of $\oE$-algebras, usually
called Dyer-Lashoff operations.  The numbering conventions for the
Dyer-Lashoff operations are the opposite of those of the Steenrod
operations: on the cohomology of $C^{*}(X;\bF_{p})$, the Dyer-Lashoff
operation $Q^{s}$ performs the Steenrod operation $P^{-s}$.  If 
$k$ is characteristic $p$ but not $\bF_{p}$, the operations
constructed this way 
are $\bF_{p}$-linear but satisfy $Q^{s}(ax)=\phi(a)Q^{s}(x)$ for $a\in
k$, where $\phi$ denotes the Frobenius automorphism of $k$.  

The $\bF_{p}$ cochain algebra of a space has the special property that
the Steenrod operation $P^{0}=Q^{0}$ is the identity operation on its
cohomology; this is not true of the zeroth Dyer-Lashof operation in
general. Indeed for a commutative $\bF_{p}$-algebra regarded as
$E_{\infty}$ $\bF_{p}$-algebra, $Q^{0}$ is the Frobenius.  (The fact
that $Q^{0}$ is the identity for the $\bF_{p}$-cochain algebra of a
space is related to the fact that it comes from a cosimplicial
$\bF_{p}$-algebra where the Frobenius in each degree is the identity.)
So when $X$ is finite $p$-type, $C^{*}(X;k)$ in each degree has a
basis that is fixed by $Q^{0}$.  We say that a finite type
$E_{\infty}$ $k$-algebra is \term{spacelike} when in each degree its
homology has a basis that is fixed by $Q^{0}$.  The main theorem
of~\cite{Einfty} is the following:

\begin{thm}[{\cite[Main Theorem, Theorem~A.1]{Einfty}}]\label{thm:pht}
The cochain complex with coefficients in $k$, $C^{*}(-;k)\colon
\DDelta^{\op}\Set \to \aE_{k}^{\op}$, is a left Quillen adjoint for
the $H_{*}(-;k)$-local model structure on simplicial sets.  If $k=\bQ$
or $k$ is characteristic $p$ and $1-\phi$ is surjective on $k$, then
the left derived functor restricts to an equivalence of the full
subcategory of the $H_{*}(-;k)$-local homotopy category consisting of
the simply connected simplicial sets of finite $H_{*}(-;k)$-type and
the full subcategory of the homotopy category of $E_{\infty}$
$k$-algebras consisting of the spacelike simply connected $E_{\infty}$
$k$-algebras of finite type.  
\end{thm}

Given the Quillen equivalence between rational commutative
differential graded algebras and $E_{\infty}$ $\bQ$-algebras
(Theorem~\ref{ethm:rect}) and the natural quasi-isomorphism (zigzag)
between $A^{*}(-)$ and $C^{*}(-;\bQ)$ \cite[p.~549]{Cochains}, the
rational statement in Theorem~\ref{thm:pht} is equivalent to
Theorem~\ref{thm:rht}.  The Sullivan theory in Theorem~\ref{thm:rht}
often includes observations on \term{minimal models}.  A simply
connected finite type rational commutative differential graded algebra
$A$ has a cofibrant approximation $A'\to A$ whose underlying graded
commutative algebra is free and such that the differential of every
element is decomposable (i.e., is a sum of terms, all of which are
word length greater than 1 in the generators); $A'$ is called a
minimal model and is unique up to isomorphism.  As a consequence,
simply connected simplicial sets of finite rational type are
rationally equivalent if and only if their minimal models are
isomorphic.  The corresponding theory also works in the context of
$E_{\infty}$ $\bQ$-algebras with the analogous definitions and proofs.
The corresponding theory does not work in the context of $E_{\infty}$
algebras in characteristic $p$ for reasons closely related to the fact
that unlike the rational homotopy groups, the $p$-adic homotopy groups
of a simplicial set are not vector spaces.

The equivalences in Theorems~\ref{thm:rht} and~\ref{thm:pht} also
extend to the nilpotent simplicial sets of finite type, but the
corresponding category of $E_{\infty}$ $k$-algebras does not have a
known intrinsic description in the $p$-adic homotopy case; in the
rational case, the corresponding algebraic category consists of the
finite type algebras whose homology is zero in negative cohomological
degrees and whose $H^{0}$ is isomorphic as a $\bQ$-algebra to the
cartesian product of copies of $\bQ$ (cf.~\cite[\S3]{EqEinfty}).

For other fields not addressed in the second part of
Theorem~\ref{thm:pht}, the adjunction does not necessarily restrict to
the indicated subcategories and even when it does, it is never an
equivalence.  To be an equivalence, the unit of the derived adjunction
would have to be an $H_{*}(-;k)$-isomorphism for simply connected
simplicial sets of finite type.  If $k\neq \bQ$ is characteristic zero, then
the right derived functor of $U$ takes $C^{*}(S^{2};k)$ to a
simplicial set
with $\pi_{2}$ isomorphic to $k$; if $k$ is characteristic $p$, then
the right derived functor of $U$ takes $C^{*}(S^{2};k)$ to a simplicial set
with $\pi_{1}$ isomorphic to the cokernel of $1-\phi$.  See
\cite[App.~A]{Einfty} for more precise results.  Because the algebraic
closure of a field $k$ of characteristic $p$ does have $1-\phi$
surjective, even when $C^{*}(-;k)$ is not an equivalence, it can be
used to detect $p$-adic equivalences.  The paper~\cite{Integral}
extends this kind of observation to the case $k=\bZ$:

\begin{thm}[{\cite[Main Theorem]{Integral}}]\label{thm:integral}
Finite type nilpotent spaces or simplicial sets $X$ and $Y$ are weakly equivalent if and
only if $C^{*}(X;\bZ)$ and $C^{*}(Y;\bZ)$ are quasi-isomorphic as
$E_{\infty}$ $\bZ$-algebras.
\end{thm}

Using the spectral version of Theorem~\ref{thm:pht}
in~\cite[App.~C]{Einfty}, the proof of the previous theorem in
\cite{Integral} extends to show that when $X$ and $Y$ are finite
nilpotent simplicial sets then $X$ and $Y$ are weakly equivalent if and only if
their Spanier-Whitehead dual spectra are weakly equivalent as
$E_{\infty}$ ring spectra.  (This was the subject of a talk by the
author at the Newton Institute in December 2002.)

We use the rest of the section to outline the argument for
Theorems~\ref{thm:rht} and~\ref{thm:pht}, using the notation of the
latter.  We fix a field $k$, which is either $\bQ$ or is
characteristic $p>0$ and has $1-\phi$ surjective.  We write $C^{*}$
for $C^{*}(-;k)$ or when $k=\bQ$ and we are working in the context of
Theorem~\ref{thm:rht}, we understand $C^{*}$ as $A^{*}$.  We also use
$C^{*}$ to denote the derived functor and write $\rU$ for the derived
functor of its adjoint.  The idea of the proof, going back to
Sullivan, is to work with Postnikov towers, and so the first step is
to find cofibrant approximations for $C^{*}(K(\pi,n))$.  For $k=\bQ$,
this is easy since $H^{*}(K(\bQ,n);\bQ)$ is the free graded
commutative algebra on a generator in degree $n$.

\begin{prop}\label{prop:rems}
If $k=\bQ$ then $C^{*}(K(\bQ,n))$ is quasi-isomorphic to the free
($E_{\infty}$ or commutative differential graded) $\bQ$-algebra on a
generator in cohomological degree $n$.
\end{prop}

We use the notation $\bE k[n]$ to denote the free $E_{\infty}$
$k$-algebra on a generator in cohomological degree $n$.  When $k$ is
characteristic $p$, there is a unique map in the homotopy category
from $\bE k[n]\to C^{*}(K(\bZ/p,n))$ that sends the generator $x_{n}$
to a class $i_{n}$ representing the image of the tautological element
of $H^{n}(K(\bZ/p,n);\bZ/p)$.  Unlike the characteristic zero case,
this is not a quasi-isomorphism since $Q^{0}[i_{n}]=[i_{n}]$ in
$H^{*}(C^{*}(K(\bZ/p,n)))$, but $Q^{0}[x_{n}]\neq [x_{n}]$ in
$H^{*}(\bE k[n])$.  Let $B_{n}$ be the homotopy pushout of a map $\bE
k[n]\to \bE k[n]$ sending the generator to a class representing
$[x_{n}]-Q^{0}[x_{n}]$ and the map $\bE k[n]\to k$ sending the
generator to $0$.  Then the map $\bE k[n]\to C^{*}(K(\bZ/p,n))$
factors through a map $B_{n}\to C^{*}(K(\bZ/p,n))$.  (The map in the
homotopy category turns out to be independent of the choices.) The
following is a key result of \cite{Einfty}, whose proof derives from a
calculation of the relationship between the Dyer-Lashof algebra and
the Steenrod algebra.

\begin{thm}[{\cite[6.2]{Einfty}}]\label{thm:pems}
Let $k$ be a field of characteristic $p>0$; then $B_{n}\to
C^{*}(K(\bZ/p,n))$ is cofibrant approximation.
\end{thm}

(As suggested by the hypothesis, we do not need $1-\phi$ to be
surjective in the previous theorem; indeed, the easiest way to proceed
is to prove it in the case $k=\bF_{p}$ and it then follows easily for
all fields of characteristic $p$ by extension of scalars.)

The two previous results can be used to calculate
$\rU(C^{*}(K(\bQ,n)))$ and\break $\rU(C^{*}(K(\bZ/p,n)))$.  In the rational
case, 
\[
\rU(C^{*}(K(\bQ,n)))\simeq U(\bE \bQ[n]) = Z(C^{n}(\Delta[\bullet])),
\]
the simplicial set of $n$-cocycles of $C^{*}(\Delta[\bullet];\bQ)$;
this is
the original model for $K(\bQ,n)$, and a straightforward argument
shows that the unit map $K(\bQ,n)\to K(\bQ,n)$ is a weak equivalence
(the identity map with this model).  In the context of
Theorem~\ref{thm:rht}, the same kind of argument is made in
\cite[10.2]{BG}.  In the $p$-adic case, we likewise have that $U(\bE
k[n])$ is the original model for $K(k,n)$, and so we get a fiber
sequence 
\[
\Omega K(k,n)\to \rU(K(\bZ/p,n))\to K(k,n)\to K(k,n).
\]
The map $K(k,n)\to K(k,n)$ is calculated in \cite[6.3]{Einfty} to be
the map that on $\pi_{n}$ induces $1-\phi$.  The kernel of $1-\phi$ is
$\bF_{p}$ and the unit map $K(\bZ/p,n)\to \rU(C^{*}(K,\bZ/p,n))$ is an
isomorphism on $\pi_{n}$.  As a consequence, when $1-\phi$ is
surjective (as we are assuming), the unit map is a weak equivalence
for $K(\bZ/p,n)$.

The game now is to show that for all finite type simply connected (or
nilpotent) simplicial sets, the derived unit map $X\to \rU C^{*}(X)$ is a
rational or $p$-adic equivalence.  The next result tells how to construct a cofibrant
approximation for a homotopy pullback; it is not a formal consequence
of the Quillen adjunction, but rather a version of the Eilenberg-Moore
theorem. 

\begin{prop}[{\cite[\S3]{BG}, \cite[\S3]{Einfty}}]\label{prop:em}
Let 
\[
\xymatrix@-1pc{%
W\ar[r]\ar[d]&Y\ar[d]\\Z\ar[r]&X
}
\]
be a homotopy fiber square of simplicial sets.  If $X,Y,Z$ are finite
$H_{*}(-;k)$-type and $X$ is simply connected, then
\[
\xymatrix@-1pc{%
C^{*}(X)\ar[r]\ar[d]&C^{*}(Y)\ar[d]\\C^{*}(Z)\ar[r]&C^{*}(W)
}
\]
is a homotopy pushout square of $E_{\infty}$ $k$-algebras or rational
commutative differential graded algebras. 
\end{prop}

Since we can write $K(\bZ/p^{m},n)$ as the homotopy fiber of a map 
\[
K(\bZ/p^{m-1},n)\to K(\bZ/p,n+1),
\]
we see that the unit of the derived adjunction is a weak equivalence
also for $K(\bZ/p^{m},n)$ (when $k$ is characteristic $p$).  Likewise,
since products are homotopy pullbacks, we also get that the unit of
the derived adjunction is a weak equivalence for $K(A,n)$ when $A$ is
a $\bQ$ vector space (when $k=\bQ$) or when $A$ is a finite $p$-group
(when $k$ is characteristic $p$). Although also not a formal
consequence of the adjunction, it is elementary to see that when a
simplicial set $X$ is the homotopy limit of a sequence $X_{j}$ and the
map $\colim H^{*}(X_{j};k)\to H^{*}(X;k)$ is an isomorphism, then
$C^{*}(X)$ is the homotopy colimit of $C^{*}(X_{j})$ and $\rU
C^{*}(X)$ is the homotopy limit of $\rU C^{*}(X_{j})$. It follows that
for $K(\bZ\phat,n)$, the unit of the derived adjunction is a weak
equivalence (when $k$ is characteristic $p$).  For any finitely
generated abelian group, the map $K(A,n)\to K(A\otimes \bQ,n)$ is a
rational equivalence and the map $K(A,n)\to K(A\phat,n)$ is a $p$-adic
equivalence.  Putting these results and tools all together, we see
that the unit of the derived equivalence is an $H_{*}(-;k)$
equivalence for any $X$ that can be built as a sequential homotopy
limit $\holim X_{j}$ where $X_{0}=*$, the connectivity of the map
$X\to X_{j}$ goes to infinity, and each $X_{j+1}$ is the homotopy
fiber of a map $X_{j}\to K(\pi_{j+1},n)$ for $\pi_{j+1}$ a finitely
generated abelian group, or the rationalization (when $k=\bQ$) or
$p$-completion (when $k$ is characteristic $p$) of a finitely
generated abelian group.  In particular, for a simply connected
simplicial set, applying this to the Postnikov tower, we get the
following result.

\begin{thm}
Assume $k=\bQ$ or $k$ is characteristic $p>0$ and $1-\phi$ is surjective.
If $X$ is a simply connected simplicial set of finite
$H_{*}(-;k)$-type, then the unit of the derived adjunction $X\to \rU
C^{*}(X)$ is an $H_{*}(-;k)$-equivalence.
\end{thm}

The previous theorem formally implies that $C^{*}$ induces an
equivalence of the $H_{*}(-;k)$-local homotopy category of simply connected simplicial sets of
finite $H_{*}(-;k)$-type with the full subcategory of the homotopy
category $E_{\infty}$ $k$-algebras or rational commutative
differential graded algebras of objects in its image.  The remainder
of Theorems~\ref{thm:rht} and~\ref{thm:pht} is identifying this image
subcategory.  In the case when $k=\bQ$, it is straightforward to see
that a finite type simply connected algebra has a cofibrant
approximation that $\rU$ turns into a simply connected principal
rational finite type Postnikov tower.  The argument for $k$ of
characteristic $p$ is analogous, but more complicated;
see~\cite[\S7]{Einfty}. 


\bibliographystyle{amsplain}\let\bibcomma\relax
\bibliography{operadbib}

\end{document}